\documentclass[reqno, 11pt, centertags,draft]{amsart}
\usepackage{amsmath,amsthm,amscd,amssymb,latexsym,upref,stmaryrd}
\usepackage{color,soul}

\numberwithin{equation}{section}

\newtheorem{theorem}{Theorem}[section]
\newtheorem{proposition}[theorem]{Proposition}
\newtheorem{corollary}[theorem]{Corollary}
\newtheorem{lemma}[theorem]{Lemma}
\newtheorem{remark}[theorem]{Remark}
\newtheorem{definition}[theorem]{Definition}

\DeclareMathOperator{\col}{col}
\DeclareMathOperator{\card}{card}

\DeclareMathOperator{\diag}{diag}
\DeclareMathOperator{\codiag}{codiag}
\DeclareMathOperator{\dom}{dom}

\DeclareMathOperator{\Span}{span}
\DeclareMathOperator{\Lip}{Lip}

\DeclareMathOperator{\const}{const}
\DeclareMathOperator*{\esssup}{ess\,sup}

\newcommand{\eps}{\varepsilon}

\renewcommand{\l}{\lambda}

\renewcommand{\Im}{{\rm Im\,}}
\renewcommand{\Re}{{\rm Re\,}}
\newcommand{\wt}{\widetilde}
\newcommand{\ol}{\overline}

\def\cI{\mathcal{I}}
\def\cK{\mathcal{K}}
\def\cL{\mathcal{L}}

\def\fH{\mathfrak{H}}

\def\bC{\mathbb{C}}
\def\bD{\mathbb{D}}
\def\bN{\mathbb{N}}
\def\bQ{\mathbb{Q}}
\def\bR{\mathbb{R}}
\def\bZ{\mathbb{Z}}

\begin{document}

\sloppy

\title[On the Riesz basis property for Dirac type operators]
{On the Riesz basis property of root vectors system for $2
\times 2$ Dirac type operators}

\author{Anton A. Lunyov}
\address{Institute of Applied Mathematics and Mechanics, NAS of Ukraine,
R. Luxemburg str. 74, \, 83114 Donetsk, Ukraine} \curraddr{}
\email{A.A.Lunyov@gmail.com}

\author{Mark~M.~Malamud}
\address{Institute of Applied Mathematics and Mechanics, NAS of Ukraine,
R. Luxemburg str. 74, \, 83114 Donetsk, Ukraine} \curraddr{}
\email{mmm@telenet.dn.ua}

\date{}

\begin{abstract}
The paper is concerned with the Riesz basis property of a boundary value problem
associated in $L^2[0,1] \otimes \bC^2$  with the following  $2 \times 2$ Dirac type equation
\begin{equation} \label{eq:Ly.abstract}
    L y = -i B^{-1} y' + Q(x) y = \lambda y , \quad
    B = \begin{pmatrix} b_1 & 0 \\ 0 & b_2 \end{pmatrix}, \quad
    y= \col(y_1, y_2),
\end{equation}
with a \textbf{summable} potential matrix $Q \in L^1[0,1]
\otimes \bC^{2 \times 2}$ and  $b_1 < 0 < b_2$. If $b_2 = -b_1
=1$ this equation is equivalent to one dimensional Dirac
equation. It is proved that the system of root functions of a
linear boundary value problem constitutes  a Riesz basis in
$L^2[0,1] \otimes \bC^2$ provided that the boundary conditions
are strictly regular.

By analogy with  the case of ordinary differential equations,
boundary conditions are called strictly regular if the
eigenvalues of the corresponding unperturbed $(Q=0)$ operator
are asymptotically simple and separated. As distinguished  from
the Dirac case there is no simple algebraic criterion of the
strict regularity whenever $b_1 + b_2 \not = 0$. However under
certain restrictions on coefficients of the boundary linear
forms we present certain algebraic criteria of the strict
regularity in the latter case. In particular, it is shown that
regular separated boundary conditions are always strictly
regular while periodic (antiperiodic) boundary conditions are
strictly regular if and only if $b_1 + b_2 \not = 0.$

The proof of the main result is based on existence of triangular
transformation operators for  system~\eqref{eq:Ly.abstract}.
Their existence is  also established here in the case  of  a
summable $Q$. In the case of regular (but not strictly regular)
boundary conditions we prove  the  Riesz basis property with
parentheses. The main results are applied  to establish   the
Riesz basis  property of the dynamic generator of spatially
non-homogenous damped Timoshenko beam model.
\end{abstract}

\maketitle{}

\renewcommand{\contentsname}{Contents}
\tableofcontents

\section{Introduction} \label{sec:Intro}
%
%
Spectral theory of non-selfadjoint boundary value problems (BVP)
on a finite interval $\cI=(a,b)$ for $n$th order ordinary
differential equations (ODE)
\begin{equation}\label{eq:ODE}
    y^{(n)} + q_1y^{(n-2)} + ... + q_{n-1}y = \l ^n y, \qquad x\in (a,b),
\end{equation}
with coefficients $q_j\in L^1[a,b]$ takes its origin in the
classical papers by Birkhoff~\cite{Bir08, Bir08exp} and
Tamarkin~\cite{Tam12, Tam17, Tam28}. They introduced the concept
of \emph{regular boundary conditions} for ODE and investigated
the asymptotic behavior of eigenvalues and eigenfunctions of
related BVP. Moreover, they proved that the system of root
functions, i.e. eigenfunctions and associated functions, of the
regular BVP is complete. Their results are also treated in the
classical monographs (see~\cite[Section 2]{Nai69}
and~\cite[Chapter 19]{DunSch71}).

More subtle is the question of whether the system of root
functions is a Riesz basis in $L^2[a,b]$.
V.P.~Mikhailov~\cite{Mikh62} and G.M.~Keselman~\cite{Kes64}
independently proved that the system of root functions of a
boundary value problem for equation~\eqref{eq:ODE} forms a Riesz
basis provided that the boundary conditions are \emph{strictly
regular}. Similar results are also obtained in~\cite[Chapter
19.4]{DunSch71}. Moreover, for boundary conditions which are
\emph{only regular} but not strictly regular,
A.A.~Shkalikov~\cite{Shk79, Shk82} proved that in the case
$q_j\in L^1(a,b),\ j\in \{1,\ldots, n-1\},$ the system of root
functions forms a Riesz basis with parentheses. Recently
A.M.~Minkin~\cite{Minkin06} proved that the converse  statement
is almost true. Namely, he proved that if multiplicities of the
eigenvalues are uniformly bounded, the Riesz basis property for
the system of root functions of BVP implies \emph{the
regularity} (not necessarily strict regularity) of the boundary
conditions.

Numerous papers are devoted to the completeness and Riesz basis
property for the  Sturm-Liouville operator (see the recent
review~\cite{Mak12} by A.S.~Makin and the references cited
therein). We especially mention the recent achievements for
periodic (anti-periodic) Sturm-Liouville operator
$-\frac{d^2}{dx^2} + q(x)$ on $[0,\pi]$. Namely, F.~Gesztesy and
V.A.~Tkachenko~\cite{GesTka09,GesTka12} for $q \in L^2[0,\pi]$
and P.~Djakov and B.S.~Mityagin~\cite{DjaMit12Crit} for $q \in
W^{-1,2}[0,\pi]$ established by different methods a
\emph{criterion} for the system of root functions to contain a
Riesz basis.

In this paper we consider a special case of the following first
order system of ODE
\begin{equation} \label{1.1}
    Ly := L(Q)y := -i B^{-1} y' + Q(x)y = \l y,
    \quad y = \col(y_1,...,y_n),
\end{equation}
where $B$ is a nonsingular diagonal $n\times n$ matrix with
complex entries,
\begin{equation} \label{1.2}
    B = \diag(b_1, b_2, \ldots, b_n) \in \bC^{n\times n},
\end{equation}
and $Q(\cdot) =: (q_{jk}(\cdot))_{j,k=1}^n \in L^1([0,1];
\bC^{n\times n})$ is a potential matrix.

To obtain a BVP, we adjoin to equation~\eqref{1.1} the following
boundary conditions (BC)
\begin{equation}\label{1.3}
Cy(0) +  D y(1)=0,  \qquad    C= (c_{jk}),\ \  D = (d_{jk}) \in {\Bbb C}^{n\times n}.
\end{equation}
Moreover, in what follows we always impose the maximality
condition ${\rm rank}(C\,\,D)=n.$

Note that, systems~\eqref{1.1} form  a more general object than
ordinary differential equations. Namely, the $n$th-order
differential equation~\eqref{eq:ODE} can be reduced to the
system~\eqref{1.1} with $r=n$ and ${b}_j = \exp\left(2\pi
ij/n\right)$ (see~\cite{Mal99}). The systems~\eqref{1.1} are of
significant interest in some theoretical and practical problems.
For instance, if $n=2m$, $B=\diag(I_m, -I_m)$ and $Q_{11} =
Q_{22}=0$, the system~\eqref{1.1} is equivalent to the Dirac
system~\cite[{Section VII.1}]{LevSar88},~\cite{Mar77}. Note also
that equation~\eqref{1.1} is used to integrate the problem of
$N$ waves arising in the nonlinear optics~\cite[{Section
III.4}]{ZMNovPit80}.

With the system~\eqref{1.1} one associates, in a natural way,
the maximal operator $L_{\max} = L_{\max}(Q)$ acting in
$L^2([0,1]; \bC^n)$ on the domain
\begin{equation}\label{Max_oper_Intro}
    \dom(L_{\max}) = \{y \in W^1_1([0,1]; \bC^n) : L_{\max}y \in L^2([0,1]; \bC^n)\}.
\end{equation}

We denote by $L_{C, D} := L_{C, D}(Q)$ the operator associated
in $L^2([0,1]; {\Bbb C}^n)$ with the
BVP~\eqref{1.1}--\eqref{1.3}. It is defined as the  restriction
of $L = L(Q)$ to the domain
\begin{equation}\label{1.3BB}
\dom(L_{C, D}) = \{y\in \dom(L_{\max}) :\ Cy(0) + D y(1)=0\}.
\end{equation}

Apparently,  the spectral problem~\eqref{1.1}--\eqref{1.3} has
first been investigated by G.~D.~Birkhoff and
R.~E.~Langer~\cite{BirLan23}. Namely, they have extended certain
previous results of~Birkhoff and Tamarkin on non-selfadjoint BVP
for ODE to the case of BVP~\eqref{1.1}--\eqref{1.3}.  More
precisely, they introduced the concepts of \emph{regular and
strictly regular boundary conditions}~\eqref{1.3} and
investigated the asymptotic behavior of eigenvalues and
eigenfunctions of the corresponding operator $L_{C, D}$.
Moreover, they proved \emph{a pointwise convergence result} on
spectral decompositions of the operator $L_{C, D}$ corresponding
to the BVP~\eqref{1.1}--\eqref{1.3} with regular boundary
conditions.

The  completeness problem of the root vectors system  \emph{of
general BVP}~\eqref{1.1}--\eqref{1.3}  has first been
investigated in the recent paper~\cite{MalOri12} by one of the
authors and L.L.~Oridoroga. In this paper the concept of
\emph{weakly regular} boundary conditions for the
system~\eqref{1.1} was introduced and the completeness of the
root vectors for this class of BVP was proved. For the Dirac
type system ($B = B^*$) the concept of {weakly regular} boundary
conditions~\eqref{1.3} coincides with that of  regular ones  and
reads as follows:
\begin{equation}\label{2.4Intro}
\det(CP_{+}  +\  DP_{-}) \ne 0  \quad\text{and}  \quad \det(CP_{-} +\ DP_{+}) \ne 0.
\end{equation}
Here  $P_+$ and $P_-$ are the spectral {projections} onto
"positive"\ and "negative"\ parts of the spectrum of $B=B^*$,
respectively.  In the recent
papers~\cite{LunMal14JST,LunMal14Arx} the completeness  of root
vectors was established for certain classes of non-regular and
even degenerated boundary conditions under certain algebraic
assumptions on the  boundary values $Q(0)$, $Q(1),$ of the
matrix $Q(\cdot)$.

Further, if Dirac type  operator $L_{C, D}$ is dissipative, then
regularity of conditions~\eqref{1.3} is equivalent to the first
of conditions~\eqref{2.4Intro} only. It is proved
in~\cite{LunMal14IEOT}  that the resolvent $(L_{C, D}-\l)^{-1}$
of any complete dissipative Dirac type  operator $L_{C, D}$
admits the spectral synthesis. In particular, the latter happens
if the first of conditions~\eqref{2.4Intro} holds.

Finally, in~\cite{LunMal14Arx, LunMal14JST} it was established
the Riesz basis property with parentheses for system~\eqref{1.1}
subject to various classes of boundary conditions with a
potential $Q \in L^\infty([0,1]; \bC^{n\times n})$.
In~\cite{MykPuy13} a stronger result was obtained for the
Dirichlet BVP  for  $2m \times 2m$ Dirac equation ($n=2m$,
$B=\diag(I_m, -I_m)$) with a potential matrix $Q \in L^2([0,1];
\bC^{2m\times 2m})$.

In this paper we investigate the Riesz basis property for
$2\times 2$ Dirac type system
\begin{equation}\label{eq:systemIntro}
    -i B^{-1} y'+Q(x)y=\l y, \qquad y={\rm col}(y_1,y_2), \qquad x\in[0,1],
\end{equation}
subject to \emph{regular} and  \emph{strictly regular} boundary
conditions~\eqref{1.3}. Here
\begin{equation}\label{eq:BQ}
    B = {\rm diag}(b_1, b_2), \quad b_1 < 0 < b_2, \quad \text{and}\quad
    Q = \begin{pmatrix}
        0      & Q_{12}\\
        Q_{21} & 0
    \end{pmatrix} \in L^1([0,1];\bC^{2 \times 2}).
\end{equation}
First  we note that in this case boundary conditions~\eqref{1.3}
are \emph{regular}, i.e. conditions~\eqref{2.4Intro} are valid,
if and only if they are \emph{equivalent} to the following
conditions
\begin{equation} \label{eq:BC.Reg_B_Con}
   {\widehat U}_{1}(y) = y_1(0) + b y_2(0) + a y_1(1) = 0,\quad
    {\widehat U}_{2}(y) = d y_2(0) + c y_1(1) + y_2(1) = 0,
\end{equation}
with certain $a,b,c,d \in \bC$ satisfying $ad-bc \ne 0$. Clearly
separated, periodic, and antiperiodic boundary conditions are
regular.

Next we recall that  regular BC~\eqref{eq:BC.Reg_B_Con}  are
called \textbf{strictly regular}, if the sequence $\Lambda_0 =
\{\l_n^0\}_{n \in \bZ}$ of the eigenvalues of the unperturbed
($Q=0$) BVP~\eqref{eq:systemIntro}--\eqref{eq:BC.Reg_B_Con}, is
asymptotically separated, i.e. there exist  $\delta > 0$ and
$n_0\in \bN$ such that
$$
|\l_j^0 - \l_k^0| > 2 \delta \quad \text{for}\quad  \ j \ne k \quad  \text{and}\quad
|j|, |k|\ge n_0.
$$
In particular, the eigenvalues $\{\l_n^0\}_{|n| > n_0}$ are
geometrically and algebraically simple.

For Dirac operator ($B = \diag(-1,1)$) the strict regularity of
BC  reads as follows: $(a - d)^2 \ne -4bc.$

Going over to
BVP~\eqref{eq:systemIntro}--\eqref{eq:BC.Reg_B_Con} we note that
a special case of   $2\times 2$ Dirac operators $L_{C,D}(Q)$,
have been investigated much deeper. For instance, P.~Djakov and
B.~Mityagin~\cite{DjaMit12Equi} imposing certain smoothness
condition on $Q$ proved equiconvergence of the spectral
decompositions for $2 \times 2$ Dirac equations subject to
\emph{general regular boundary conditions}.

Moreover,  the Riesz basis property for  $2\times 2$ Dirac
operators $L_{C,D}(Q)$ has been investigated in numerous papers
(see~\cite{TroYam01,
TroYam02,HasOri09,DjaMit10BariDir,Bask11,DjaMit12UncDir,DjaMit12TrigDir,DjaMit12Crit,DjaMit13CritDir}
and references therein). The most complete result was  obtained
by P. Djakov and B. Mityagin in~\cite{DjaMit12UncDir}. Namely,
assuming that  $Q \in L^2[0,1] \otimes \Bbb C^{2 \times 2}$ it
is proved in~\cite{DjaMit12UncDir} that the system of root
vectors of the Dirac operator $L(Q)$ with \emph{regular boundary
conditions} constitutes a \emph{Riesz basis with parentheses} in
$L^2[0,1] \otimes \Bbb C^2$ and  ordinary Riesz basis provided
that BC are \emph{strictly regular}. Note, that
\emph{non-degenerate separated} boundary conditions are always
\emph{strictly regular}, hence the root vectors of the
corresponding BVP constitute  a \emph{Riesz
basis}~\cite{DjaMit12UncDir} (see Remark~\ref{rem_separ_BC} in
this connection).

The following theorem is the  main result of the paper.
%
%
\begin{theorem} \label{th:basis.strict}
Let $L_{C, D}(Q)$ be the operator associated in $L^2([0,1];
{\Bbb C}^2)$ with the
BVP~\eqref{eq:systemIntro}--\eqref{eq:BC.Reg_B_Con} and let
$Q_{12}, Q_{21} \in L^1[0,1]$. Assume that  boundary
conditions~\eqref{eq:BC.Reg_B_Con} are strictly regular. Then
root vectors system of the operator $L_{C, D}(Q)$ forms a Riesz
basis in $L^2[0,1] \otimes \bC^2$.
\end{theorem}
%
%
While definition of the strict regularity in  the case of
$b_1\not = -b_2$ is rather implicit, for certain classes of
boundary conditions it can be expressed in purely algebraic
terms. For instance, if $bc=0$ and $ad\not = 0$, then
BC~\eqref{eq:BC.Reg_B_Con} are strictly regular whenever $b_1
\ln |d| + b_2 \ln |a| \ne 0.$ In particular, \emph{periodic
$(a=d= -1)$ and antiperiodic $(a=d= 1)$ BC are strictly regular
if and only if} $b_1 + b_2\not =0$.  Therefore
Theorem~\ref{th:basis.strict}  implies the following surprising
result.
%
%
\begin{corollary} \label{th:basis.Perid_Cond}
Let $Q_{12}, Q_{21} \in L^1[0,1]$ and  $b_1+ b_2 \not =0$.  Then
the system of root vectors  of the \textbf{periodic}
(\textbf{antiperiodic}) operator $L_{C, D}(Q)$ forms a Riesz
basis in $L^2[0,1] \otimes \bC^2$.
\end{corollary}
%
%
This result demonstrates substantial difference between Dirac
and Dirac type operators. Note in this connection that  a
criterion for the system of root functions of the
\emph{periodic} (necessarily non-strictly regular) BVP for $2
\times 2$ Dirac equation to contain a Riesz basis (without
parentheses) was obtained by P. Djakov and B. Mityagin
in~\cite{DjaMit12Crit}.

We also prove that the root vectors system of the operator
$L(Q)$ forms a Riesz basis with parentheses provided that BC are
regular (see Proposition~\ref{prop:regul.basis.paren}).

Emphasize that methods used in~\cite{DjaMit12UncDir, Bask11}
essentially use condition $Q \in L^2[0,1] \otimes \bC^{2 \times
2}$ (i.e. the fact that the Fourier coefficients belong to
$l^2(\Bbb Z)$) and most likely could not be applied even to
Dirac operators with  $L^1$-potentials $Q$. Note also that
traditional methods of perturbations theory are also not
applicable here since as opposed to the $L^2$-case, the
multiplication operator by $Q (\in L^1[0,1] \otimes \bC^{2
\times 2})$ is neither $B^{-1}\frac{d}{dx}$~-- compact nor  even
subordinated to the unperturbed operator $B^{-1}\frac{d}{dx}$.

The Riesz basis property for abstract operators is investigated
in numerous papers. Due to the lack of space we only
mention~\cite{Katsn67,MarMats84,Markus88,Agran99}, the recent
paper~\cite{Shk10} and the references therein.

The main results of the paper including Theorem
\ref{th:basis.strict}  were announced in~\cite{LunMal14Dokl}
(partially with proofs). After appearance of~\cite{LunMal14Dokl}
there appeared the paper by A.M. Savchuk and
A.A. Shkalikov~\cite{SavShk14} where Theorem~\ref{th:basis.strict}
was proved for the $2\times 2$ Dirac operator. Note that
approaches in~\cite{LunMal14Dokl} and~\cite{SavShk14}
substantially differ. Moreover the case of Dirac type operators
$(b_1 + b_2 \not =0)$ has  interesting features (see e.g.
Corollary~\ref{th:basis.Perid_Cond}) and   turns out to be  more
complicated.

The paper is organized as follows. In
Section~\ref{sec:Transform} we prove the existence of triangular
transformation operators for equation
\eqref{eq:systemIntro}--\eqref{eq:BQ}. In
Section~\ref{sec:AsympSol} we apply these operators to obtain
asymptotic formulas for solutions to
equation~\eqref{eq:systemIntro}. In turn, these formulas  are
applied in Section~\ref{sec:Regular} to obtain the following
asymptotic formula
\begin{equation} \label{eq:l.n=l.n0+o(1)_Intro}
    \l_n = \l_n^0 + o(1), \quad\text{as}\quad n \to \infty, \quad n \in \bZ.
\end{equation}
for the eigenvalues  $\Lambda = \{\l_n\}_{n \in \bZ}$ of the
operator $L_{C, D}(Q)$ with regular BC. In
Section~\ref{sec:StrictRegular} we present certain necessary and
sufficient \emph{algebraic} conditions for
equations~\eqref{eq:BC.Reg_B_Con} to determine strictly regular
BC. In particular, we show in
Proposition~\ref{prop_strict_regular} that if $\alpha :=
-b_1/b_2  \not \in \bQ$ and $a=0$, $bc, d \in \bR \setminus
\{0\}$, then BC~\eqref{eq:BC.Reg_B_Con}  are \emph{strictly
regular} if and only if
\begin{equation} \label{eq:a=0.criter_Str_Regularity_Intro}
    d \ne -(\alpha+1)\left(|bc| \alpha^{-\alpha}\right)^{\frac{1}{\alpha+1}}.
\end{equation}
So, under the above  restrictions  condition
\eqref{eq:a=0.criter_Str_Regularity_Intro} gives the
\emph{algebraic criterion of the strict regularity} of boundary
conditions~\eqref{eq:BC.Reg_B_Con}. In
Section~\ref{sec:RieszBasis} we prove our main results on Riesz
basis property of the root vectors system of the operator
$L_{C,D}(Q)$ (Theorem~\ref{th:basis.strict} and
Proposition~\ref{prop:regul.basis.paren}). Finally, in
Section~\ref{sec:Timoshenko} we apply
Theorem~\ref{th:basis.strict} to prove the Riesz basis property
with parentheses for the dynamic generator of the Timoshenko
beam model (see
e.g.~\cite{Tim55,KimRen87,Shub02,XuYung04,XuHanYung07,WuXue11}).

{\bf Notation.} Let $T$ be a closed operator in a Hilbert space
$\fH$. Denote by $\rho(T)$  the set of regular points  of  $T$;
$\sigma(T) = \bC\setminus \rho(T)$ and $\sigma_p(T)$ denote the
spectrum of $T$ and  the point spectrum of  $T$, respectively.

For  the eigenvalue $\l_0\in \sigma_p(T)$  denote by $m_a(\l_0)$
and $m_g(\l_0)$ the algebraic and geometric multiplicities of
$\l_0$, respectively. Recall that $m_g(\l_0) =
\dim(\ker(L-\l_0))$ and $m_a(\l_0)$ is a dimension of the root
subspace corresponding to $\l_0$.

$\bD_r(z) \subset \bC$ denotes the disc of radius $r$ with a
center $z$.

$\left\langle \cdot,\cdot\right\rangle$ denotes the inner
product in $\bC^{n}$; $\bC^{n\times n}$ denotes the set of
$n\times n$ matrices with complex entries; $I_n (\in
\bC^{n\times n})$ denotes the identity matrix.
%
%
\section{Triangular transformation operators}
\label{sec:Transform}
%
%
\subsection{The Banach spaces $X_1$  and $X_\infty$ }
Following~\cite{Mal94} denote by  $X_1:= X_1(\Omega)$ and
$X_\infty := X_\infty(\Omega)$  the linear  spaces composed of
(equivalent classes of) measurable functions defined on $\Omega
= \{(x,t) : 0 \le t \le x \le 1\}$ satisfying
\begin{align}
      \label{eq:B2.norm.def}
        \|f\|_{X_1} &:=  \esssup_{t \in [0,1]} \int_t^1 |f(x,t)| dx < \infty,\\
    \label{eq:B1.norm.def}
        \|f\|_{X_\infty} &:= \esssup_{x \in [0,1]} \int_0^x |f(x,t)| dt < \infty,
\end{align}
respectively.  It can easily be shown that the spaces  $X_1$ and
$X_\infty$ equipped with the norms~\eqref{eq:B2.norm.def} and
\eqref{eq:B1.norm.def} form  Banach spaces that are not
separable. Denote by $X_{1,0}$ and $X_{\infty,0}$  the subspaces
of $X_1$ and $X_\infty$, respectively, obtained by taking the
closure  of continuous functions $f\in C(\Omega)$. Clearly, the
set $C^1(\Omega)$ of smooth functions is also dense in both
spaces $X_{1,0}$ and $X_{\infty,0}$.

To motivate appearance  of the spaces $X_1$ and $X_\infty$
consider a Volterra type operator
    \begin{equation}\label{2.3opNew}
N:\  f\to \int^x_0 N(x,t)f(t)dt
    \end{equation}
with a measurable kernel $N(\cdot,\cdot)$ and denote by $\|N\|_p
:= \|N\|_{L^p[0,1]\to L^p[0,1]}$ the $L_p$-norm of the operator
$N$ provided that it is bounded. The following simple lemma
(cf.~\cite{Mal94}) sheds light on appearance of the spaces $X_1$
and $X_\infty$.

Recall that a Volterra  operator in a Banach space is a compact
operator with zero spectrum.
\begin{lemma}\label{lem_Volterra_operGeneral}
Let $N(\cdot,\cdot)\in X_1(\Omega)\cap  X_\infty(\Omega)$ and
generate the Volterra type operator~\eqref{2.3opNew}. Then:

\textbf{(i)} The  operator $N$ is  bounded  in  $L^p[0,1]$ for
each $p\in [1,\infty]$ and
\begin{equation}\label{2.4Aop}
    \|N\|_p  \le  \|N\|_{X_1(\Omega)}^{1/p} \cdot \|N\|_{X_\infty(\Omega)}^{1-1/p}.
\end{equation}
Moreover,
\begin{equation}\label{2.4opNew}
    \|N\|_{1} = \|N\|_{X_1(\Omega)}, \qquad  \|N\|_{\infty} =   \|N\|_{X_\infty}.
\end{equation}

\textbf{(ii)} If $N(\cdot,\cdot)\in X_{1,0}(\Omega) \cap
X_{\infty,0}(\Omega)$, then $N$ is a Volterra operator in
$L^p[0,1]$ for each  $p\in [1,\infty]$.
\end{lemma}
%
%
\begin{proof}
\textbf{(i)} The relations~\eqref{2.4opNew} are well known (see
e.g.~\cite{Mal99}) and can easily  be proved. Combining the M.
Riesz's interpolation theorem  with relations~\eqref{2.4opNew}
yields
   \begin{equation}\label{2.4Bop}
\|N\|_p  \le  \|N\|_{1}^{1/p} \cdot \|N\|_{\infty}^{1-1/p} = \|N\|_{X_1(\Omega)}^{1/p}
\cdot \|N\|_{X_\infty(\Omega)}^{1-1/p}, \quad p\in [1,\infty].
   \end{equation}
which proves  estimate~\eqref{2.4Aop}.

\textbf{(ii)}  Since $N(\cdot,\cdot)\in X_{1,0}(\Omega) \cap
X_{\infty,0}(\Omega)$, there exists a sequence $N_k(\cdot,
\cdot)\in C^1(\Omega)$, $k\in \Bbb N,$ such that
$\lim_{k\to\infty}(\|N-N_k\|_{X_1(\Omega)} +
\|N-N_k\|_{X_\infty(\Omega)}) = 0.$ In accordance with
~\eqref{2.4Aop}
   \begin{equation}\label{2.7op}
\|N - N_k\|_p  \le  \|N - N_k\|_{X_1(\Omega)}^{1/p} \cdot \|N -
N_k\|_{X_\infty(\Omega)}^{1-1/p},\quad p\in [1,\infty].
   \end{equation}
Since $N_k(\cdot, \cdot)\in C^1(\Omega)$, the operator $N_k$ of
the form \eqref{2.3opNew} is a Volterra operator. It follows
from~\eqref{2.7op} that the operator $N$ of the form
\eqref{2.3opNew}  is the uniform limit in $L^p[0,1]$  of the
Volterra operators $N_k$ and is therefore itself a Volterra
operator.
     \end{proof}
The following simple properties of the class
$X_{\infty,0}(\Omega)$ will be useful in the sequel.
  \begin{lemma}\label{Trace_lemma}
For each $a\in[0,1]$ the trace mapping
   \begin{equation}
i_a:\  C(\Omega)\to C[0,a], \qquad i_a\bigl(N(x,t)\bigr):=N(a,t),
   \end{equation}
originally defined on $C(\Omega)$ admits a continuous extension
(also denoted by $i_a$) as a mapping $X_{\infty,0}(\Omega) \to
L^1[0,a]$ from $X_{\infty,0}(\Omega)$ onto $L^1[0,a]$.
    \end{lemma}
   \begin{proof}
Let $N(\cdot,\cdot)\in X_{\infty,0}(\Omega)$ and let
$N_k(\cdot,\cdot)\in C(\Omega)$ be a sequence approaching $N$ in
$X_{\infty}(\Omega)$. It follows from definition
\eqref{eq:B1.norm.def} of the norm in  $X_{\infty}(\Omega)$ that
  \begin{equation}
\int^a_0|N_k(a,t) - N_m(a,t)|dt \le  \|N_k-N_m\|_{X_{\infty}}\to 0\quad \text{as}\quad
n,m\to\infty,
  \end{equation}
i.e. the sequence  $N_k(a,\cdot)$ is a Cauchy sequence in
$L^1[0,a].$  Thus, there exists $f_a(\cdot)\in L^1[0,a]$ such
the $\|f_a-N_k(a,\cdot)\|_{L^1[0,a]}\to 0$ as $k\to\infty$. We
put $N(a,\cdot):= f_a(\cdot)$ and extend the mapping $i_a$ to
the space $X_{\infty,0}$  by setting $i_a:\  N(\cdot,\cdot) \to
f_a(\cdot) = N(a,\cdot)$. It is easily seen that this extension
is well defined. Indeed, if $\wt N_k(\cdot,\cdot)$ is another
sequence approaching $N(\cdot,\cdot)$ in $X_{\infty}$. Then
$\lim_{k\to\infty} \|N_k - \wt N_k\|_{X_{\infty}}=0$  and
    \begin{equation}
\lim_{k\to\infty} \|N_k(a,\cdot)- \wt N_k(a,\cdot)\|_{L^1[0,a]} \le \lim_{k\to\infty}
\|N_k - \wt N_k\|_{X_{\infty}}=0.
    \end{equation}
Hence   $\lim_{k\to\infty} \|N(a,\cdot)- \wt
N_k(a,\cdot)\|_{L^1[0,a]} = 0$  and the extension $i_a$ is well
defined.
     \end{proof}

Going over to the vector case we introduce  the Banach  spaces
  \begin{equation}\label{2.8opNew}
X_1^{{2\times 2}} := X_1^{{2\times 2}}(\Omega) := X_1(\Omega)\otimes \Bbb C^{2\times 2}
\quad \text{and} \quad X_\infty^{{2\times 2}} := X_\infty^{{2\times 2}}(\Omega) :=
X_\infty(\Omega)\otimes \Bbb C^{2\times 2}
  \end{equation}
consisting of $2\times 2$ matrix functions $f =
(f_{jk})_{j,k=1}^2$ with entries from $X_1$ and $X_\infty$,
respectively, and  equipped with the norms
  \begin{equation}\label{2.46}
\|f\|_{X_1} :=  \|f\|_{X_1\otimes \Bbb C^{2\times 2}} := \max\{\|f_{jk}\|_{X_1}:
j,k\in\{1,2\}\},
  \end{equation}
    \begin{equation}\label{2.45}
\|f\|_{X_\infty} :=  \|f\|_{X_\infty \otimes \Bbb C^{2\times 2}} :=
\max\{\|f_{jk}\|_{X_\infty}: j,k\in\{1,2\}\}.
    \end{equation}
We also put
$$
X_{1,0}^{{2\times 2}} := X_{1,0}^{{2\times 2}}(\Omega) := X_{1,0}(\Omega)\otimes \Bbb
C^{2\times 2} \quad \text{and} \quad X_{\infty,0}^{{2\times 2}} :=
X_{\infty,0}^{{2\times 2}}(\Omega) := X_{\infty,0}(\Omega)\otimes \Bbb C^{2\times 2}.
$$

Further, equip  the space $L^p([0,1],\bC^2) :=
L^p[0,1]\otimes\bC^2$ of vector functions with the following
norm
    \begin{equation}
\|f\|_p:=\|\col (f_1,{f_2})\|_p := \|f_1\|_p + \|f_2\|_p, \quad  p\in[1,\infty].
    \end{equation}
where $\|f_j\|_p := \|f_j\|_{L^p[0,1]},$\ $j\in\{1,2\}$.

With each measurable kernel  $\bigl(N_{jk}(\cdot,
\cdot)\bigr)^2_{j,k=1}$  one associates a Volterra type operator
    \begin{equation}\label{2.39opNew}
N: \binom{f_1}{f_2}\to \int^x_0 N(x,t)\binom{f_1(t)}{f_2(t)}dt =
 \int^x_0 \begin{pmatrix}
N_{11}(x,t) & N_{12}(x,t) \\
N_{21}(x,t) & N_{22}(x,t)
  \end{pmatrix}
\binom{f_1(t)}{f_2(t)}\,dt.
    \end{equation}
Let us set  $\|N\|_p := \|N\|_{L^p[0,1]\otimes \bC^{2}\to
L^p[0,1]\otimes \bC^{2}}$,  $ p\in[1,\infty]$ provided that the
norm  is bounded.

\begin{lemma}\label{lem_Volterra_Vector_oper}
Let $N(\cdot,\cdot) = \bigl(N_{jk}(\cdot,
\cdot)\bigr)^2_{j,k=1}, \in X_1^{{2\times 2}}(\Omega)\cap
X_\infty^{{2\times 2}}(\Omega)$ and generate the Volterra type
operator by formula~\eqref{2.39opNew}. Then:

\textbf{(i)} The Volterra type  operator $N$ is a bounded
operator in $L^p([0,1],\bC^2)$ for each $p\in [1,\infty]$ and
\begin{equation}\label{2.4AopNew}
\|N\|_p  \le  \|N\|_{X_1^{{2\times 2}}}^{1/p} \cdot \|N\|_{X_\infty^{{2\times
2}}}^{1-1/p}.
   \end{equation}
Moreover,
    \begin{equation}\label{2.14op}
\|N\|_{1} = \|N\|_{X_1^{2\times 2}}  \qquad \text{and} \qquad  \|N\|_{\infty} =
\|N\|_{X_\infty^{2\times 2}}.
    \end{equation}

\textbf{(ii)} If $N(\cdot,\cdot)\in X_{1,0}^{{2\times
2}}(\Omega) \cap  X_{\infty,0}^{2\times 2}(\Omega)$, then $N$ is
a Volterra operator in $L^p([0,1],\bC^2)$  for each $p\in
[1,\infty]$.
\end{lemma}
%
%
The proof is similar to that of
Lemma~\ref{lem_Volterra_operGeneral} and is omitted.

Next we demonstrate that the assumption  $N(\cdot,\cdot)\in
X_{1,0}(\Omega) \cap X_{\infty,0}(\Omega)$ in
Lemma~\ref{lem_Volterra_operGeneral}(ii) is essential for the
operator $N$ to be  a Volterra operator.
   \begin{proposition}\label{prop2.3}
Let $k(\cdot) \in L^1[0,1]$, $k(s)s^{-1}\in L^1[0,1]$, and let
$\mathcal M_k(\alpha):= \int^1_0 k(s)s^{\alpha-1}\,ds$ be the
Mellin transform of $k(\cdot)$, $\alpha \in \Bbb C_r:=\{z\in\Bbb
C:\Re z>0\}$. Then the Volterra type operator
   \begin{equation}
\cK :\ f \to   \frac{1}{x}\int^x_0 k\left(\frac{t}{x}\right)f(t)dt.
   \end{equation}
has  the following  properties:

\textbf{(i)} The operator $\cK$ is bounded in $L^p[0,1]$ for
each $p\in[1,\infty]$.

\textbf{(ii)} Its point spectrum is given by $\sigma_p(\cK) =
\text{range}(\mathcal M_k(\cdot)) = \{\int^1_0
k(s)s^{\alpha-1}\,ds:\ \alpha \in \Bbb C_r\}.$ In particular,
$\cK$ is not compact.

\textbf{(iii)} If $k(\cdot)\ge 0$, then the  spectral radius  of
$\cK$ is equal to its norm $\|\cK\|_p$ in each $L^p[0,1]$.
\end{proposition}
   \begin{proof}
\textbf{(i)} Let us check that
$N(x,t)=\frac{1}{x}k(\frac{t}{x})\in X_1(\Omega)\cap
X_{\infty}(\Omega)$. Indeed setting $t/x =s$ one easily gets
    \begin{equation}
\int^1_t|N(x,t)|dx = \int^1_t \frac{1}{x}|k\left(\frac{t}{x}\right)|dx =
 \int^1_t s^{-1}|k(s)|\,ds \le \int^1_0
s^{-1}|k(s)|ds,
    \end{equation}
and
\begin{align}
    \int^x_0|N(x,t)|dt &= \frac{1}{x}  \int^x_0 |k\left(\frac{t}{x}\right)|\,dt =
    \int^1_0\frac{1}{x}|k(s)|x\,ds \nonumber \\
    &= \int^1_0 |k(s)|\,ds \le \int^1_0 s^{-1}|k(s)|ds.
\end{align}
The boundedness of $\cK$ in $L^p[0,1]$, $p\in [1,\infty]$, is
now implied by Lemma \ref{lem_Volterra_operGeneral}.

\textbf{(ii)} Clearly, $f_{\alpha} = x^{\alpha-1} \in L^1[0,1]$
for $\alpha\in \Bbb C_r$ and
   \begin{equation}
(\cK f_{\alpha})(x) =  \frac{1}{x}\int^x_0 k\left(\frac{t}{x}\right) t^{\alpha -1}\,dt =
\frac{1}{x}\int^1_0 k(s)(x s)^{\alpha -1}x\,ds =  M_{k}(\alpha)x^{\alpha -1}.
   \end{equation}
Due to the assumption $k(\cdot)f_{0}(\cdot)\in L^1[0,1]$ one has
   \begin{equation}
|M_k(\alpha)| \le \int^1_0|k(s)s^{\alpha -1}|ds\le\int^1_0|k(s)|s^{-1}\,ds =:
M_{|k|}(0).
   \end{equation}
Since the function $M_k(\cdot)$ is holomorphic and bounded in
$\Bbb C_r$,   it might have at most countable (discrete in $\Bbb
C_r(-1))$ set of zeros. For the rest of $\alpha$th $x^{\alpha}$
is the eigenvector of $R$ in $L^1[0.1]$ belonging to a non-zero
eigenvalue $c(\alpha)$, i.e. $x^{\alpha}\in \ker(\cK -
c(\alpha))$. Hence $\cK$ is not a compact operator.
    \end{proof}

\subsection{Transformation operators}

   \begin{theorem} \label{th:Trans}
Let $Q_{12}, Q_{21}\in L^1[0,1].$  Assume that
$e_{\pm}(\cdot;\l)$ are the solutions of the
system~\eqref{eq:systemIntro} corresponding to the initial
conditions $e_\pm(0;\l)=\binom{1}{\pm1}$. Then
$e_{\pm}(\cdot;\l)$ admits the following representation  by
means of the triangular transformation operator
  \begin{equation}\label{eq:e=(I+K)e0}
    e_{\pm}(x;\l) = (I+K^{\pm})e^0_{\pm}(x;\l)
    = e^0_{\pm}(x;\l) + \int^x_0 K^{\pm}(x,t) e^0_{\pm}(t;\l)dt,
  \end{equation}
where
\begin{equation}\label{eq:e=e0}
    e^0_{\pm}(x;\l)=\binom{e^{ib_1\l x}}{\pm e^{ib_2\l x}}, \qquad
    K^{\pm}(x,t)=\bigl(K^{\pm}_{jk}(x,t)\bigr)^2_{j,k=1},
\end{equation}
and $K^{\pm}(\cdot,\cdot) \in X_{1,0}^{{2\times 2}}(\Omega) \cap
X_{\infty,0}^{2\times 2}(\Omega).$  In particular, the operator
$K^{\pm}: f\to \int^x_0 K^{\pm}(x,t) f(t)dt,$ is a  Volterra
operator in each $L^p[0,1]$, $p\in [1,\infty]$, hence
$\sigma(K^{\pm})=\{0\}$.
\end{theorem}
%
%
Our further considerations will substantially be relied on the
following result which is a special case of~\cite[Theorem
1.2]{Mal99} where the general case of $n\times n$ system
\eqref{2op}--\eqref{4op}  with  the  matrix $B=B^*\in
\bC^{n\times n}$ and $Q \in C[0,1]\otimes\bC^{n\times n}$ was
treated.

\begin{proposition}\cite{Mal99}\label{prop_transfor_oper}
Let  $Q= \codiag(Q_{12}, Q_{21})\in C^1[0,1]\otimes\bC^{2\times
2}$.  Then the boundary value problem
    \begin{equation}\label{1op}
B^{-1}D_x K^{\pm}(x,t) + D_t K^{\pm}(x,t)B^{-1} + iQ(x)K^{\pm}(x,t)=0,
    \end{equation}
\begin{equation}\label{2op}
K^{\pm}(x,x)B^{-1} -  B^{-1}K^{\pm}(x,x) = iQ(x),\quad x\in [0,1],
\end{equation}
  \begin{equation}\label{3op}
K^{\pm}(x,0)B \binom{1}{\pm 1} = 0, \quad x\in [0,1].
\end{equation}
has the unique solution $K^{\pm}(\cdot,\cdot)=
\bigl(K_{jk}^{\pm}(\cdot,\cdot) \bigr)^2_{j,k=1} \in
C^1(\Omega)\otimes\bC^{2\times 2}$. Moreover,
$K^{\pm}(\cdot,\cdot)$ is the matrix  kernel  of the
transformation operator \eqref{eq:e=(I+K)e0}.
     \end{proposition}
The proof of this result in~\cite{Mal99} is divided in two
steps. At first it is proved solvability (and uniqueness)  of
the certain  auxiliary boundary value problem which in
conformity to the $2\times 2$-case reads as follows
    \begin{equation}\label{4op}
B^{-1}D_x R(x,t) + D_t R(x,t)B^{-1} + iQ(x)R(x,t)=0,
    \end{equation}
\begin{equation}\label{5op}
R(x,x)B^{-1} -  B^{-1}R(x,x) = iQ(x),\quad x\in [0,1],
\end{equation}
   \begin{equation}\label{6op}
R_{11}(x,0)= R_{22}(x,0) =0, \quad x\in [0,1].
  \end{equation}
where $R(x,t) = \bigl(R_{jk}(x,t)\bigr)^2_{j,k=1}$. Let us
recall the corresponding statement  from~\cite{Mal99}.
  \begin{proposition}\label{Prop_smooth_sol}
Let  $Q = \codiag(Q_{12}, Q_{21}) \in
C^1[0,1]\otimes\bC^{2\times 2}$. Then the auxiliary
problem~\eqref{4op}--\eqref{6op} has a solution   $R\in
C^1(\Omega)\otimes\bC^{2\times 2}$.
 Moreover, it is unique  in $X_{\infty}^{2\times 2}(\Omega)$.
    \end{proposition}
Our first auxiliary result reads as follows.
  \begin{proposition}\label{prop_estimate_for_difference}
Assume that $Q, \wt Q \in C^1[0,1]\otimes\bC^{2\times 2}$ and
$\|Q\|_{L^1[0,1]\otimes\bC^{2\times 2}}, \|\wt Q\|_{L^1[0,1]}\le
r$. Then there exists a constant $C = C(r, b_1, b_2)$  such that
     \begin{equation}\label{2.12op}
\|R-\wt R\|_{X_{\infty}^{{2\times 2}}} \le C\|Q-\wt Q\|_{L^1[0,1]\otimes\bC^{2\times
2}}.
  \end{equation}
\end{proposition}
\begin{proof}
We put
    \begin{equation}\label{7op}
a_j := b^{-1}_j  \quad \text{and} \quad  \kappa_{jk} := \frac{a_k}{a_j} =
\frac{b_j}{b_k}, \qquad j,k\in\{1,2\}.
    \end{equation}
and
\begin{equation}
\xi_{jk}(x,t)=
\begin{cases}
(a_k x- a_j t)(a_k - a_j)^{-1}, &j\not =k,\\
x-t,  &j=k.
\end{cases}
\end{equation}

Let us rewrite  boundary value problem~\eqref{4op} --
\eqref{6op} in the scalar form
    \begin{equation}\label{8op}
a_k\left(D_x R_{kk}(x,t) + D_t R_{kk}(x,t)\right) =  -i Q_{kj}(x)R_{jk}(x,t), \quad k\in
\{1,2\},
    \end{equation}
 \begin{equation}\label{9op}
a_j D_x R_{jk}(x,t) + a_k D_t R_{jk}(x,t) =  -i Q_{jk}(x)R_{kk}(x,t),\   k\in \{1,2\},\
j\not = k,
   \end{equation}
  \begin{equation}\label{10op}
R_{jk}(x,x)=\frac{i Q_{jk}(x)}{a_k-a_j}, \quad x\in[0,1],
  \end{equation}
%
    \begin{equation}\label{11op}
R_{11}(x,0) = 0, \quad R_{22}(x,0) = 0, \quad x\in[0,1].
   \end{equation}
The system~\eqref{8op} --~\eqref{9op}  is hyperbolic.
Integrating the Goursat problem~\eqref{8op} --~\eqref{11op}
along characteristics   we arrive at the following equivalent
system of integral equations
  \begin{equation}\label{12op}
R_{kk}(x,t) = - \frac{i}{a_j}\int^x_{x-t}Q_{kj}(\xi)R_{jk} \bigl(\xi, \xi -x +
t\bigr)d\xi, \quad k\in \{1,2\},
  \end{equation}
\begin{multline}\label{13op}
    R_{jk}(x,t) = \frac{i}{a_k- a_j}Q_{jk}\left(\frac{a_k x - a_j t}{a_k - a_j}\right) \\
                - \frac{i}{a_k}\int^x_{\xi_{jk}(x,t)}Q_{jk}(\xi)R_{kk}\bigl(\xi,\kappa_{jk}(\xi-x)+t\bigr)d\xi.
\end{multline}
Here $j,k \in \{1,2\}, \  j\not = k$

The functions  $\wt R_{jk}$ satisfy the same
system~\eqref{12op}--\eqref{13op}  with $\wt Q_{jk}$ instead of
$Q_{jk}$, $j,k \in \{1,2\}$. Next we put
  \begin{equation}\label{14op}
\widehat R_{jk}=\wt R_{jk} - R_{jk}, \quad   \widehat Q_{jk}=\wt Q_{jk}-Q_{jk},\qquad
j,k\in\{1,2\}.
     \end{equation}
and
   \begin{equation}\label{15op}
\widehat {\frak I}_{jk}(x) := \int^x_0|\widehat R_{jk}(x,t)|dt, \qquad \frak I_{jk}(x)
:= \int^x_0| R_{jk}(x,t)|dt, \qquad j,k\in\{1,2\}.
 \end{equation}
Making use the change of variables $\xi=u,$\ $ \xi -x +t = v$,
we obtain   from \eqref{12op} that
\begin{align}\label{16op}
\widehat{\frak I}_{11}(x) := \int^x_0|\widehat R_{11}(x,t)|dt \le |b_1| \int^x_0\,dt
\int^x_{x-t}|\widehat Q_{12}(\xi)
 R_{21}(\xi,\xi-x+t)|d\xi  \nonumber \\
+ \  |b_1|  \int^x_0\,dt \int^x_{x-t}|\wt Q_{12}(\xi) \widehat
R_{21}(\xi,\xi-x+t)|d\xi \nonumber   \\
\le |b_1| \int^x_0 |\widehat Q_{12}(u)|\,du \int^u_0 |R_{21}(u, v)|dv
+ |b_1| \int^x_0 |\wt Q_{12}(u)|\,du \int^u_0 | \widehat R_{21}(u,v)|dv \nonumber   \\
= |b_1| \int^x_0 |\widehat Q_{12}(u)| \frak I_{21}(u)\,du  + |b_1| \int^x_0 |\wt
Q_{12}(u)| \widehat{\frak I}_{21}(u)\,du.
\end{align}
Similarly, it follows from~\eqref{13op} and~\eqref{15op}
   \begin{align}
\widehat{\frak I}_{21}(x) =  \int^x_0|\widehat R_{21}(x,t)|dt \le \frac{1}{|a_1-a_2|}
\int^x_0 \left |\widehat Q_{21}\left(\frac{a_1 x-a_2
t}{a_1-a_2}\right)\right| \,dt \nonumber  \\
+ \frac{1}{a_2}\int^x_0 dt\int^x_{\xi_{21}(x,t)}|\widehat
Q_{21}(\xi)R_{21}\bigl(\xi,\kappa_{21}(\xi-x)+t\bigr)|d\xi   \nonumber  \\
+ \frac{1}{a_2}\int^x_0 dt\int^x_{\xi_{21}(x,t)}|\wt Q_{21}(\xi) \widehat
R_{21}\bigl(\xi,\kappa_{21}(\xi-x)+t\bigr)|d\xi.
  \end{align}

Making use the change of variables $\xi=u,$\
$\frac{a_1}{a_2}(\xi-x)+t=v$, we obtain
       \begin{align}\label{17op}
\widehat{\frak I}_{21}(x) = \frac{1}{a_2} \int^x_{\frac{a_1 x}{a_1-a_2}}|\widehat
Q_{21}(u)|du + \frac{1}{a_2}\int^x_{\frac{a_1 x}{a_1-a_2}}|\widehat
Q_{21}(u)|du\int^u_{\frac{a_1}{a_2}(u-x)}|R_{11}(u,v)|dv \nonumber  \\
+ \frac{1}{a_2}\int^x_{\frac{a_1 x}{a_1-a_2}}|\wt
Q_{21}(u)|du\int^u_{\frac{a_1}{a_2}(u-x)}|\widehat R_{11}(u,v)|dv \nonumber   \\
\le \frac{1}{a_2} \int^x_0 |\widehat Q_{21}(u)|du + \frac{1}{a_2}\int^x_0 |\widehat
Q_{21}(u)|du\int^u_0 |R_{11}(u,v)|dv \nonumber  \\
+ \frac{1}{a_2}\int^x_0 |\wt
Q_{21}(u)|du\int^u_0 |\widehat R_{11}(u,v)|dv \nonumber   \\
= \frac{1}{a_2} \int^x_0|\widehat Q_{21}(t)|dt + \frac{1}{a_2}\int^x_0\frak
I_{11}(u)|\widehat Q_{21}(u)|du
+ \frac{1}{a_2}\int^x_0 \widehat{\frak I}_{11}(u)|\wt Q_{21}(u)|du.
\end{align}
Let
\begin{equation}\label{18op}
C_j=\|\frak I_{j1}\|_{L^{\infty}[0,1]}=\|R_{j1}\|_{X_\infty(\Omega)},\qquad j\in\{1,2\}.
  \end{equation}
Estimate~\eqref{16op} with account of~\eqref{18op} yields
   \begin{equation}\label{19op}
\widehat{\frak I}_{11}(x)   \le |b_1| C_2 \|\widehat Q_{12}\|_{L^1} +
|b_1|\int^x_0\widehat{\frak I}_{21}(u)|\wt Q_{12}(u)|du.
   \end{equation}
Combining this inequality with~\eqref{17op} implies
    \begin{align}
\widehat{\frak I}_{21}(x) \le \left(\frac{1}{a_2} + \frac{C_1}{a_2}\right)
\int^x_0|\widehat
Q_{21}(u)|du + \frac{1}{a_2} \int^x_0\widehat{\frak I}_{11}(u)|\wt Q_{21}(u)|du  \nonumber \\
\le C_1'\|\widehat Q_{21}\|_{L^1} + |b_1b_2|C_2\|\widehat
Q_{12}\|_{L^1}\cdot\|\wt Q_{21}\|_{L^1}
+ |b_1 b_2|\int^x_0|\wt Q_{21}(u)|du \int^u_0\widehat{\frak I}_{21}(t)|\wt Q_{12}(t)|dt \nonumber \\
\le C_3 \left(\|\widehat Q_{21}\|_{L^1} + \|\widehat Q_{12}\|_{L^1}\right) + |b_1
b_2|\int^x_0\widehat{\frak I}_{21}(t)|\wt
Q_{12}(t)|dt\int^x_t|\wt Q_{21}(u)|du  \nonumber \\
\le  C_3 \left(\|\widehat Q_{21}\|_{L^1} + \|\widehat Q_{12}\|_{L^1}\right) +    |b_1
b_2| \cdot \|\wt Q_{21}\|_{L^1} \int^x_0\widehat {\frak I}_{21}(t)|\wt Q_{12}(t)|dt,
  \end{align}
where
$$
C_1' := \frac{1 + C_1}{a_2} \qquad \text{and}\qquad  C_3 := \max\{C_1', |b_1b_2|C_2
\|\wt Q_{21}\|_{L^1}\}.
$$

Applying Cronwall's lemma to this inequality implies
    \begin{equation}\label{2.29op}
\widehat{\frak I}_{21}(x) \le  C_3 \left(\|\widehat Q_{21}\|_{L^1} + \|\widehat
Q_{12}\|_{L^1}\right) \exp\left(|b_1 b_2|\cdot\|\wt Q_{21}\|_{L^1}\int^x_0|\wt
Q_{12}(t)|dt\right).
    \end{equation}
Inserting this inequality in~\eqref{19op} we arrive at the
inequality
    \begin{equation}\label{2.30op}
\widehat{\frak I}_{11}(x) \le  |b_1| \left(\|\widehat Q_{12}\|_{L^1} + \|\widehat
Q_{21}\|_{L^1}\right) \left(C_2 + C_3\|\wt Q_{12}\|_{L^1}\exp\left(|b_1 b_2|\cdot\|\wt
Q_{12}\|_{L^1}\cdot\|\wt Q_{21}\|_{L^1}\right)\right)
    \end{equation}
Similar reasoning leads to similar  estimates for
$\widehat{\frak I}_{12}$ and $\widehat{\frak I}_{22}$. Combining
these estimates with~\eqref{2.29op} and~\eqref{2.30op} we arrive
at~\eqref{2.12op}.
          \end{proof}
%

   \begin{proposition}\label{prop2.5}
Let $Q = \codiag(Q_{12}, Q_{21}) \in L^1[0,1]\otimes
\bC^{2\times 2}.$ Then the system of integral
equations~\eqref{12op} --~\eqref{13op} has a unique solution $R
= (R_{jk})_{j,k=1}^2$ belonging to  $X_{\infty,0}^{2\times
2}(\Omega)$. Moreover, $(R_{jk})_{j,k=1}^2 \in X_{1,0}^{{2\times
2}}(\Omega) \cap  X_{\infty,0}^{2\times 2}(\Omega)$.

Further, let $Q_n = \codiag(Q_{12,n}, Q_{21,n}) \in
C^1[0,1]\otimes \bC^{2\times 2}$ be any sequence approaching $Q$
in $L^1[0,1]$-norm and let $R_{n}= (R_{jk,n})_{j,k=1}^2 \in
C^1(\Omega)\otimes \bC^{2\times 2}$ be the corresponding system
of solutions of the problem~\eqref{4op}--\eqref{6op} with $Q_n$
instead of $Q$. Then there exists a constant $C = C(Q, b_1,b_2)$
not depending on $n$ and such that the following estimates hold
  \begin{equation}\label{main_estimate}
\|R_{jk}-R_{jk,n}\|_{X_1} + \|R_{jk}-R_{jk,n}\|_{X_\infty}\le  C \|Q -
Q_{n}\|_{L^1\otimes \bC^{2\times 2}}, \quad j,k\in \{1,2\},\ n\in \bN.
  \end{equation}
\end{proposition}
\begin{proof}
(i) Choose sequences $\{Q_{12, n}\}_{n\in \bN},
\{Q_{21,n}\}_{n\in \bN} \subset C^1[0,1]$ such that
    \begin{equation}\label{2.32}
\|Q_{12}-Q_{12,n}\|_{L^1} + \|Q_{21}-Q_{21,n}\|_{L^1} \to 0\quad \text{as}\  n\to\infty.
    \end{equation}
By Proposition~\ref{Prop_smooth_sol}, for each pair
$Q_n=\{Q_{12, n},Q_{21,n}\}$ there exists the unique matrix
solution $R_n=(R_{j k, n})^2_{j,k=1}\in C^1(\Omega) \otimes
\bC^{2\times 2},$ $n \in \bN$  of the system
\eqref{12op}--\eqref{13op}. It follows from~\eqref{2.32} and
\eqref{2.12op}  that there exists $R = (R_{jk})^2_{j,k=1} \in
X_{\infty}^{{2\times 2}}(\Omega)$ such that
      \begin{equation}\label{2.33A}
\lim_{n\to\infty}\|R_{jk,n} - R_{jk}\|_{X_{\infty}(\Omega)} = 0, \quad j,k\in \{1,2\}.
      \end{equation}
Let us show that $\{R_{jk}(\cdot,\cdot)\}^2_{j,k=1}$ satisfies
the system \eqref{12op}--\eqref{13op}. Let for instance equation
$j\not = k$. Writing down equation \eqref{13op} for
$R_{jk,n}(\cdot,\cdot)$ and integrating it with respect to $t$
from $0$ to $x$ one gets (cf.~\eqref{17op})
    \begin{equation}\label{2.33}
\int^x_0 R_{jk,n}(x,t)\,{dt} =  -\frac{i}{a_j} \int^x_{\frac{a_k x}{a_k - a_j}}
Q_{jk,n}(u)\,du -  \frac{i}{a_j}\int^x_{\frac{a_k x}{a_k - a_j}}
Q_{jk,n}(u)\,du\int^u_{\frac{a_k}{a_j}(u-x)}R_{kk,n}(u,v)\,dv.
    \end{equation}
It follows from estimate~\eqref{2.29op} that
$\lim_{n\to\infty}\int^{v_2}_{v_1}R_{kk, n}(u,v)dv =
\int^{v_2}_{v_1}R_{kk}(u,v)dv$ for any pair $v_1,v_2\in[0,1]$.
Therefore and due to~\eqref{2.33A} the dominated convergence
theorem applies  as $n\to\infty$ in~\eqref{2.33} and gives
  \begin{align}
\int^x_0 R_{jk}(x,t)\,{dt} = - \frac{i}{a_j} \int^x_{\frac{a_k x}{a_k - a_j}}
Q_{jk}(u)\,du -  \frac{i}{a_j}\int^x_{\frac{a_k x}{a_k - a_j}}
Q_{jk}(u)\,du\int^u_{\frac{a_k}{a_j}(u-x)}R_{kk}(u,v)\,dv \nonumber \\
= \int^x_0  \left[\frac{i}{a_k- a_j}Q_{jk}\left(\frac{a_k x - a_j t}{a_k - a_j}\right) -
\frac{i}{a_k}\int^x_{\xi_{jk}(x,t)}Q_{jk}(\xi)R_{kk}\bigl(\xi,\kappa_{jk}(\xi-x)+t\bigr)d\xi\right]dt.
    \end{align}
The latter  is equivalent to~\eqref{13op}. The equations for
$R_{jj}(\cdot,\cdot)$, $j\in \{1,2\}$, is  obtained similarly.

(ii)  Since the sequence $Q_n(\cdot)$ approaches $Q(\cdot)$ in
$L^1$-norm, it is bounded, $\|Q_n\|_{L^1\otimes\bC^{2\times
2}}\le C_1 = C_1(Q, B)$, $n\in \bN.$ Therefore
Proposition~\ref{prop_estimate_for_difference} applies and gives
  \begin{equation}\label{main_estimate_infty}
 \|R_{jk}-R_{jk,n}\|_{X_\infty}\le  C \|Q - Q_{n}\|_{L^1\otimes \bC^{2\times 2}},
\quad j,k\in \{1,2\}.
  \end{equation}

Next we prove  similar estimate in $X_1(\Omega)$-norm. We let
  \begin{equation}\label{14opSecond}
\widehat R_{jkn}= R_{jk} - R_{jkn}, \quad   \widehat Q_{jkn}= Q_{jk} - Q_{jkn},  \quad
j,k\in\{1,2\},\ n\in \bN.
     \end{equation}

First we prove estimate~\eqref{main_estimate} for  the case
$j\not =k$. To this end  we
 note that
    \begin{equation}
 \int^1_t \widehat Q_{jkn}\left(\frac{a_k x - a_j t}{a_k - a_j}\right)\,dx =
\int^1_t \widehat Q_{jkn}\left(\xi_{jk}(x,t)\right)\,dx  = \frac{{a_k - a_j}}{a_k}
\int^{\xi_{jk}(1,t)}_t \widehat Q_{jkn}(u)\,du.
    \end{equation}
where  $\xi_{jk}(\cdot,\cdot)$  is given by~\eqref{7op}.

Further, note that $R_{jkn}(\cdot,\cdot)$  satisfies
equation~\eqref{13op} with $Q_{jkn}$ in place  of  $Q_{jk}$,
$j,k\in \{1,2\}.$ Taking difference of this equation
\eqref{13op}, then integrating the difference with respect to
$x\in [t,1],$ and making use the change of variables $u = \xi,$
$v = (\xi-x)\kappa_{lk} + t = (\xi-x)\frac{a_k}{a_j} + t$, we
obtain
   \begin{align}\label{2.39op_New}
 \int^1_t|\widehat R_{jk,n}(x,t)|dx = \int^1_t|R_{jk}(x,t)-R_{jk,n}(x,t)|dx \nonumber \\
\le \frac{1}{|a_k|} \int^{\xi_{jk}(1,t)}_t |\widehat Q_{jk,n}(u)|\,du
+ \frac{1}{|a_j|} \int^1_t dx\int^x_{\xi_{jk}(x,t)}|Q_{jk}(\xi)\widehat R_{kk,n}(\xi,\kappa_{jk}(\xi-x)+ t)|d\xi \nonumber \\
+ \frac{1}{|a_j|} \int^1_t dx\int^x_{\xi_{jk}(x,t)}|\widehat Q_{jk,n}(\xi)R_{kk,n}(\xi,\kappa_{jk}(\xi-x)+t)|d\xi   \nonumber \\
= \frac{1}{|a_k|} \int^{\xi_{jk}(1,t)}_t |\widehat Q_{jk,n}(u)|\,du + \frac{1}{|a_j|}
\int_t^{\xi_{jk}(1,t)}\,dv\int^{(v-t)\frac{a_j}{a_k} + 1}_v|Q_{jk}(u) \widehat R_{kk,n}(u,v)|du \nonumber \\
+ \frac{1}{|a_j|} \int_t^{\xi_{jk}(1,t)} dv\int^{(v-t)\frac{a_j}{a_k}+1}_v|\widehat
Q_{jk,n}(u) R_{kk,n}(u,v)|du \nonumber   \\
\le \frac{1}{|a_k|} \int^1_t |\widehat Q_{jk,n}(u)|\,du + \frac{1}{|a_j|}
\int_t^1\,dv\int^1_v|Q_{jk}(u) \widehat R_{kk,n}(u,v)|du \nonumber \\
+ \frac{1}{|a_j|} \int_t^1\,dv \int^1_v|\widehat Q_{jk,n}(u) R_{kk,n}(u,v)|du \nonumber \\
=  \frac{1}{|a_k|} \int^1_t |\widehat Q_{jk,n}(u)|\,du +  \frac{1}{|a_j|}
\int_t^1\,|Q_{jk}(u)|\,du\int^u_t|\widehat R_{kk,n}(u,v)|\,dv \nonumber \\
+ \frac{1}{|a_j|} \int_t^1 |\widehat Q_{jk,n}(u)|\,du \int^1_v |R_{kk,n}(u,v)|dv \nonumber \\
\le |b_k|\cdot\|\widehat Q_{jk,n}\|_{L^1[t,1]} + |b_j|\cdot \|Q_{jk}\|_{L^1[t,1]}
\cdot\|\widehat R_{kk,n}\|_{X_{\infty}}  \nonumber \\
 + \  |b_j|\cdot \|\widehat Q_{jk,n}\|_{L^1[t,1]} \cdot\|R_{kk,n}\|_{X_{\infty}}.
    \end{align}
Here we use simple inequalities $\xi_{jk}(1,t)\le 1$ and
${(v-t)\frac{a_j}{a_k} + 1}\le 1$. The latter holds since $t\ge
v$ and ${a_j}{a_k}<0$. It follows  from \eqref{2.39op_New}  with
account of definition
\eqref{eq:B2.norm.def}--\eqref{eq:B1.norm.def} that
\begin{multline}\label{2.40op}
    \|R_{jk} - R_{jk,n}\|_{X_1} \le  |b_k|\cdot\|\widehat
    Q_{jk,n}\|_{L^1[0,1]} \\ + |b_j|\left(\|Q_{jk}\|_{L^1}\cdot
    \|\widehat R_{kk,n}\|_{X_{\infty}} +
    \|\widehat Q_{jk,n}\|_{L^1}\cdot \|R_{kk,n}\|_{X_{\infty}}\right), \ j\not = k.
\end{multline}
On the other hand,  estimate~\eqref{main_estimate_infty} implies
$\lim_{n\to\infty}\|R_{jk}-R_{jk,n}\|_{X_\infty}=0.$ Therefore
there exists $C_2>0$ such that $\max\{\|R_{jk,n}\|_{X_\infty}:
j,k\in \{1,2\}, n\in \bN\}\le C_2$. Combining this estimate with
\eqref{main_estimate_infty} yields the following estimate
  \begin{equation}\label{main_estimate_L1_new}
 \|R_{jk}-R_{jk,n}\|_{X_1}\le  C_3 \|Q - Q_{n}\|_{L^1\otimes \bC^{2\times 2}},
\quad j,k\in \{1,2\},
  \end{equation}
with a ceratin positive constant $C_3>0$ not depending on $n\in
\Bbb N.$ Combining this estimate
with~\eqref{main_estimate_infty} implies~\eqref{main_estimate}
for  $j \not =k$.

(iii)  Going over to the case $j = k$ we start with equation
~\eqref{12op} and similar equation for $R_{kk,n}(\cdot,\cdot)$
which holds with $Q_{jk,n}$ in place  of  $Q_{jk}$, $j,k\in
\{1,2\}$.  Taking difference of this equation and~\eqref{12op},
then integrating the difference with respect to $x\in [t,1],$
and then making use the change of variables $\xi=u, \xi-x+t=v$,
obtain as
   \begin{align}
|a_j|\int^1_t|\widehat R_{jj,n}(x,t)|dx = |a_j|\int^1_t|R_{jj}(x,t)-R_{jj,n}(x,t)|dx   \nonumber \\
\le\int^1_t dx\int^x_{x-t}|Q_{jk}(\xi)\widehat R_{kj,n}(\xi,\xi-x+t)|d\xi +
\int^1_t dx\int^x_{x-t}|\widehat Q_{jk,n}(\xi)R_{kj,n}(\xi,\xi-x+t)|d\xi \nonumber  \\
= \int^t_0 dv\int^{v-t+1}_v|Q_{jk}(u) \widehat R_{kj,n}(u,v)|du  +
\int^t_0 dv\int^{v-t+1}_v|\widehat Q_{jk,n}(u) R_{kj,n}(u,v)|du \nonumber \\
\le \int^1_0 dv\int^{1}_v|Q_{jk}(u) \widehat R_{kj,n}(u,v)|du  +
\int^1_0 dv\int^{1}_v|\widehat Q_{jk,n}(u) R_{kj,n}(u,v)|du \nonumber  \\
 =  \int^1_0|Q_{jk}(u)|du\int^u_0|\widehat R_{kj,n}(u,v)|dv + \int^1_0|\widehat
Q_{jk,n}(u)|du\int^u_0|R_{kj,n}(u,v)|dv.
   \end{align}
It follows with account of definition~\eqref{eq:B2.norm.def} --
\eqref{eq:B1.norm.def} that
  \begin{equation}
\|R_{jj}-R_{jj,n}\|_{X_1}\le |b_j|\left(\|Q_{jk}\|_{L^1}\cdot \|\widehat
R_{kj,n}\|_{X_{\infty}} +
 \|\widehat Q_{jk,n}\|_{L^1}\cdot \|R_{kj,n}\|_{X_{\infty}}\right), \quad j\in \{1,2\}.
  \end{equation}
Since $\|R_{jk,n}\|_{X_\infty} \le C_2$ for $n\in \bN$,  this
estimate together with \eqref{main_estimate_infty} leads to the
estimate~\eqref{main_estimate} with $j=k$.
    \end{proof}
%
%
\begin{lemma}\label{lem_Volterra_oper}
Let $Q_{12}, Q_{21}\in L^1[0,1]$ and let $R(\cdot,\cdot) =
(R_{jk}(\cdot,\cdot))_{j,k=1}^2$ be a solution of the system of
integral equations \eqref{12op}--\eqref{13op}. Then
$R(\cdot,\cdot)\in X_{1,0}^{{2\times 2}}(\Omega) \cap
X_{\infty,0}^{2\times 2}(\Omega)$ and  the operator
    \begin{equation}\label{2.39op}
R: \binom{f_1}{f_2}\to \int^x_0 R(x,t)\binom{f_1(t)}{f_2(t)}dt =
 \int^x_0 \begin{pmatrix}
R_{11}(x,t) & R_{12}(x,t) \\
R_{21}(x,t) & R_{22}(x,t)
  \end{pmatrix}
\binom{f_1(t)}{f_2(t)}\,dt
    \end{equation}
is a Volterra operator in $L^p[0,1]\otimes\bC^2$ for each  $p\in
[1,\infty]$.
    \end{lemma}
\begin{proof}
Let  $Q_n = \codiag(Q_{12,n}, Q_{21,n}) \in C^1[0,1]\otimes
\bC^{2\times 2}$ be any sequence approaching $Q$ in
$L^1[0,1]$-norm and let $R_{n}= (R_{jk,n})_{j,k=1}^2 \in
C^1(\Omega)\otimes \bC^{2\times 2}$, $n\in \bN,$ be the
corresponding system of solutions of the
problem~\eqref{4op}--\eqref{6op} with $Q_n$ instead of $Q$.

Since $R_{n}= (R_{jk,n})_{j,k=1}^2$ is a smooth kernel and
$\lim_{n\to\infty}\|Q - Q_n\|_{L_1[0,1]\otimes \bC^{2\times 2}}
=0,$ it follows from Proposition~\ref{prop2.5} (see estimates
\eqref{main_estimate}) that  $R(\cdot,\cdot)\in
X_{1,0}^{{2\times 2}}(\Omega) \cap X_{\infty,0}^{2\times
2}(\Omega)$. Therefore by Lemma~\ref{lem_Volterra_Vector_oper},
$R$  is a Volterra operator in $L^p[0,1]\otimes \bC^{2},$ $p\in
[1,\infty]$. This completes the proof.
    \end{proof}

\begin{proof}[Proof of Theorem 1]
Let  $P^{\pm} = \diag(P^{\pm}_1, P^{\pm}_2)$ be a diagonal
matrix function with entries $P^{\pm}_j\in L^1[0,1],$
$j\in\{1,2\}$. Define the convolution operator
  \begin{equation}
P^{\pm}:\  f\to   \int^x_0 P^{\pm}(x-t)f(t)dt,\quad  f = \col(f_1,f_2) \in
L^1[0,1]\otimes \bC^{2}.
  \end{equation}
Let $R(x,t)=(R_{jk}(x,t))_{j,k=1}^2$   be the solution of the
system of integral equations~\eqref{12op}--\eqref{13op}.
Starting from the operator $I+R$ and following the reasoning
of~\cite[Theorem 1.2]{Mal99}  we define the operator $K^{\pm}$
by the equality
  \begin{equation}
I + K^{\pm} =(I + R)(I + P^{\pm}).
 \end{equation}
The latter means  that the kernel $K^{\pm}(\cdot,\cdot)$ of
$K^{\pm}$ is given by
   \begin{equation}\label{2.51op}
K^{\pm}(x,t) = R(x,t) + P^{\pm}(x-t) + \int^x_t R(x,s) P^{\pm}(s-t)ds.
   \end{equation}
Let us show that $K^{\pm}(\cdot,\cdot)$ is the kernel of the
transformation operator, i.e.
representation~\eqref{eq:e=(I+K)e0}--\eqref{eq:e=e0} holds.
First,  we  choose $P_1^{\pm}(\cdot)$ so that
$K^{\pm}(\cdot,\cdot)$ will satisfy condition~\eqref{3op}, i.e.
$$
a_1 K^{\pm}_{j1}(x,0) + a_2 K^{\pm}_{j2}(x,0) =0, \quad j\in \{1,2\}.
$$
Inserting representation~\eqref{2.51op} for
$K^{\pm}(\cdot,\cdot)$ in these relations leads to the following
system of Volterra type integral equations
   \begin{equation}\label{2.52op}
\begin{cases}
a_1 P_1^{\pm}(x) + \int^x_0[a_1 R_{11}(x,t)P^{\pm}_1(t)  \pm a_2 R_{12}(x,t)
P^{\pm}_2(t)]dt=\mp a_2 R_{12}(x,0) =: g_1^{\pm}(x),   \\
a_2 P^{\pm}_2(x) + \int^x_0[a_1 R_{21}(x,t)P^{\pm}_1(t) \pm a_2
R_{22}(x,t)P^{\pm}_2(t)]dt = - a_1 R_{21}(x,0) =: g_2^{\pm}(x).
\end{cases}
  \end{equation}
Here the relations~\eqref{11op}  have been taken  into account.
It follows from equation~\eqref{13op} that the functions
$R_{jk}(x,0)$, $j,k\in \{1,2\},$ are well defined. Moreover, the
estimate~\eqref{2.39op_New}  ensures  that $g_{j}(\cdot) \in
L^1[0,1],$ $j\in \{1,2\}.$

On the other hand, by Lemma~\ref{lem_Volterra_oper} the operator
$R$ of the form~\eqref{2.39op} is a Volterra operator in
$L^1[0,1]$. Therefore  system~\eqref{2.52op} is the system of
Volterra equations in $L^1[0,1]\otimes \bC^2$ with respect to
$\col\{a_1 P_1^{\pm}(\cdot), a_2 P_2^{\pm}(\cdot)\},$  hence has
the unique solution $\col\{a_1 P_1^{\pm}(\cdot), a_2
P_2^{\pm}(\cdot)\}\in L^1[0,1]\otimes \bC^2.$

Further, choose a sequence  $Q_n = \codiag(Q_{12,n}, Q_{21,n})
\in C^1[0,1]\otimes \bC^{2\times 2}$ approaching $Q$ in
$L^1[0,1]$-norm.  Then according to Proposition
\ref{Prop_smooth_sol}  there exists the corresponding sequence
of matrix solutions $R_{n} = (R_{jk,n})_{j,k=1}^2 \in
C^1(\Omega)\otimes \bC^{2\times 2}$  of the problem
\eqref{4op}--\eqref{6op} with $Q_n$ instead of $Q$.  Moreover,
by Proposition \ref{prop2.5}, the estimate~\eqref{main_estimate}
holds, hence  $R_{n}$ approaches $R$ in $X_1$ and $X_\infty$
norms.  Choose a sequence  $P^{\pm}_n = \diag(P^{\pm}_{1,n},
P^{\pm}_{2,n})$ of  diagonal matrix functions  with entries
$P^{\pm}_{j,n}(\cdot) \in C^1[0,1],$ $j\in\{1,2\},\ n\in \bN,$
and assume that $P^{\pm}_n(\cdot)$ satisfies the following
system of Volterra  integral equations
   \begin{equation}\label{2.52opNew}
  \begin{cases}
a_1 P_{1,n}^{\pm}(x) + \int^x_0[a_1 R_{11,n}(x,t)P^{\pm}_{1,n}(t)  \pm a_2
R_{12,n}(x,t)P^{\pm}_{2,n}(t)]dt = \mp a_2 R_{12,n}(x,0) =: g_{1,n}^{\pm}(x),   \\
a_2 P^{\pm}_{2,n}(x) +  \int^x_0[a_1 R_{21,n}(x,t) P^{\pm}_{1,n}(t) \pm a_2
R_{22,n}(x,t)P^{\pm}_{2,n}(t)]dt = - a_1 R_{21,n}(x,0) =: g_{2,n}^{\pm}(x).
  \end{cases}
  \end{equation}
Next we define the kernels $K^{\pm}_n(\cdot,\cdot)$ by setting
(cf. formula \eqref{2.51op})
   \begin{equation}\label{2.54op}
K^{\pm}_n(x,t) = R_n(x,t) + \Phi^{\pm}_n(x-t) + \int^x_t R_n(x,s) P^{\pm}_n(s-t)ds,
\quad n\in \bN.
   \end{equation}

Clearly, $K^{\pm}_n(\cdot,\cdot)\in C^1(\Omega)$ and  in
accordance with~\cite[Theorem 1.2]{Mal99}, it  is the unique
solution of the boundary value problem~\eqref{1op}--\eqref{3op}.
Note for instance, that condition~\eqref{2op} for the kernel
$K^{\pm}_n(\cdot,\cdot)$ is satisfied since  $R_n(\cdot,\cdot)$
satisfies this condition,  $K^{\pm}_n(x,x) = R_n(x,x) +
P^{\pm}_n(0)$, and the matrix $P^{\pm}_n(0)$ is diagonal.

Further,  by Proposition~\ref{prop_transfor_oper},
$K^{\pm}_n(\cdot,\cdot)$ is the kernel of transformation
operator for equation~\eqref{eq:system}--\eqref{eq:BQ} with
$Q_n$ in place of $Q$, i.e. the solution $e_{\pm,n}(\cdot;\l)$
of this equation satisfying the initial condition
$e_{\pm,n}(0;\l)=\binom{1}{\pm1}$ admits a representation
 \begin{equation}\label{eq:e=(I+K)e0New}
    e_{\pm,n}(x;\l) = (I+K^{\pm}_n)e^0_{\pm}(x;\l)
    = e^0_{\pm}(x;\l) + \int^x_0 K^{\pm}_n(x,t) e^0_{\pm}(t;\l)dt, \quad n\in \bN.
 \end{equation}

Our aim is to pass to the limit in~\eqref{2.54op}  and
\eqref{eq:e=(I+K)e0New} as $n\to\infty.$  It follows from
\eqref{2.39op_New} with $t=0$ and the estimate
\eqref{main_estimate} that
$$
\lim_{n\to\infty} (\|g_{1,n}^{\pm} - g_{1}^{\pm}\|_{L^1[0,1]} + \|g_{2,n}^{\pm} -
g_{2}^{\pm}\|_{L^1[0,1]}) = 0.
$$
Combining this relation with Proposition~\ref{prop2.5}  we
obtain from~\eqref{2.52op} and~\eqref{2.52opNew} that
  \begin{equation}\label{2.53op}
\binom{a_1 P_{1,n}^{\pm}(\cdot)}{a_2 P_{2,n}^{\pm}(\cdot)} = (I +
R_n)^{-1}\binom{g_{1,n}^{\pm}}{g_{2,n}^{\pm}} \to (I +
R)^{-1}\binom{g_{1}^{\pm}}{g_{2}^{\pm}} = \binom{a_1 P_{1}^{\pm}(\cdot)}{a_2
P_{2}^{\pm}(\cdot)}.
   \end{equation}
Further, setting $\widehat P^{\pm}_n:= P_n - P$ we derive
from~\eqref{2.51op} and~\eqref{2.54op}
   \begin{align}
\int^1_t|\widehat K^{\pm}_n(x,t)|dx = \int^1_t|K^{\pm}_n(x,t)- K^{\pm}(x,t)|dx  \le
\int^1_t|\widehat R_n(x,t)|\,dx +  \int^1_t |\widehat P_n^\pm(x-t)|\,dx  \nonumber \\
+ \int^1_t|P^{\pm}_n(s-t)|ds \int^1_s|\widehat R_n(x,s)|\,dx +
\int^1_t |\widehat P^{\pm}_n(s-t)|ds \int^1_s |R(x,s)|\,dx  \nonumber  \\
\le\|\widehat R_n\|_{X_1(\Omega)} \left(1 + \|P^{\pm}_n\|_{L^1[0,1]}\right) + \|\widehat
P^{\pm}_n\|_{L^1[0,1]} \left(1 + \|R\|_{X_1(\Omega)}\right).\quad \label{2.72}
   \end{align}
On the other hand, by Proposition~\ref{prop2.5},
$\lim_{n\to\infty}\|\widehat R_{n}\|_{X_1^{2\times 2}} =
\lim_{n\to\infty}\|R_{n}-R\|_{X_1^{2\times 2}} = 0$, and due to
\eqref{2.53op} $\lim_{n\to\infty}\|\widehat
P^{\pm}_n\|_{L^1[0,1]} = 0$. Combining these relations with
\eqref{2.72} yields
    \begin{equation}\label{2.73}
\lim_{n\to\infty}\|\widehat  K^{\pm}_n\|_{X_1^{2\times 2}} = \lim_{n\to\infty}\| K^{\pm}
-  K^{\pm}_n\|_{X_1^{2\times 2}} = 0.
    \end{equation}
The latter means that $K^{\pm}\in X_{1,0}^{2\times 2}(\Omega)$.
In just the same way one proves the relation
    \begin{equation}\label{2.74}
\lim_{n\to\infty}\|\widehat  K^{\pm}_n\|_{X_\infty^{2\times 2}} = \lim_{n\to\infty}\|
K^{\pm} -  K^{\pm}_n\|_{X_\infty^{2\times 2}} = 0.
    \end{equation}

Using  relation~\eqref{2.74}  we can pass to the limit as
$n\to\infty$ in formula~\eqref{eq:e=(I+K)e0New} and arrive at
the required formula~\eqref{eq:e=(I+K)e0}.
    \end{proof}
  \begin{remark}
(i)    For Dirac $2\times 2$ system $(B = \diag(-1,1))$ with
continuous $Q$ the triangular transformation operators  have
been constructed in~\cite[Ch.10.3]{LevSar88}
and~\cite[Ch.1.2]{Mar77}. For $Q \in L^1[0,1]\otimes
\bC^{2\times 2}$ it is proved in~\cite{AlbHryMyk05} by an
appropriate generalization of the Marchenko  method.

(ii) Let $J: f\to \int_0^xf(t)dt$ be a Volterra operator on
$L^p[0,1]$.  Note that the similarity of Volterra operators
given by~\eqref{2.39opNew} to the simplest Volterra operators of
the form $B\otimes J$   acting in the  spaces $L^p[0,1]\otimes
\bC^{2}$ has been investigated in~\cite{Mal99, Rom08}.
     \end{remark}

\section{Asymptotic behavior of solutions} \label{sec:AsympSol}
%
%
Let  $K^{\pm}(x,t)=\bigl(K^{\pm}_{jk}(x,t)\bigr)^2_{j,k=1}$ be
the kernel of a triangular  transformation operator constructed
in  Theorem~\ref{th:Trans} (see formulas
\eqref{eq:e=(I+K)e0}--\eqref{eq:e=e0}). To state the next result
we put
\begin{equation} \label{eq:Rjk}
    R_{jk}^{\pm} := 2^{-1}({K_{jk}^{+} + K_{jk}^{-}}), \quad j,k \in \{1,2\},
\end{equation}
and let
\begin{equation}\label{3.2}
    \Phi(\cdot, \l) =
    \begin{pmatrix} \varphi_{11}(\cdot, \l)  & \varphi_{12}(\cdot, \l)\\
    \varphi_{21}(\cdot,\l) & \varphi_{22}(\cdot,\l)
    \end{pmatrix} =: \begin{pmatrix} \Phi_1(\cdot, \l) & \Phi_2(\cdot, \l)
    \end{pmatrix}, \qquad \Phi(0, \l) = I_2,
\end{equation}
be a fundamental matrix solution of the
system~\eqref{eq:systemIntro}. Here $\Phi_k(\cdot, \l)$ is the
$k$th column of  $\Phi(\cdot, \l)$.

Our investigating of the perturbation determinant relies on the
following result.
%
%
\begin{proposition} \label{prop:phi.jk=e+int}
Let $Q \in L^1[0,1] \otimes \bC^{2 \times 2}$ and let
$\varphi_{jk}(\cdot, \l)$, $j,k\in \{1,2\},$ be the entries of
the  fundamental matrix solution~\eqref{3.2}.
 Then the functions $\varphi_{jk}(\cdot,
\l)$ admit the following representations
\begin{align}
    \varphi_{11}(x,\l) &= e^{i b_1 \l x} + \int_0^x R_{11}^{+}(x,t)e^{ib_1\l t}dt
        + \int_0^x R_{12}^{-}(x,t)e^{ib_2\l t}dt, \label{eq:phi11} \\
    \varphi_{12}(x,\l) &= \, \ \ \ \ \ \ \ \ \ \int_0^x R_{11}^{-}(x,t)e^{ib_1\l t}dt
        + \int_0^x R_{12}^{+}(x,t)e^{ib_2\l t}dt, \label{eq:phi12} \\
    \varphi_{21}(x,\l) &= \, \ \ \ \ \ \ \ \ \ \int_0^x R_{21}^{+}(x,t)e^{ib_1\l t}dt
        + \int_0^x R_{22}^{-}(x,t)e^{ib_2\l t}dt, \label{eq:phi21} \\
    \varphi_{22}(x,\l) &= e^{i b_2 \l x} + \int_0^x R_{21}^{-}(x,t)e^{ib_1\l t}dt
        + \int_0^x R_{22}^{+}(x,t)e^{ib_2\l t}dt, \label{eq:phi22}
\end{align}
where $R_{jk}^{\pm}\in  X_{1,0}(\Omega) \cap
X_{\infty,0}(\Omega),$ \ $j,k \in \{1,2\}$.
   \end{proposition}
  \begin{proof}
Comparing initial conditions  and applying the Cauchy uniqueness
theorem one easily gets $\Phi_1(\cdot,\l) = \begin{pmatrix}
\varphi_{11}(\cdot, \l) \\ \varphi_{21}(\cdot,\l)
    \end{pmatrix}  =   e_{+}(\cdot;\l) + e_{-}(\cdot;\l).$  Inserting  in place of  $e_{+}(\cdot;\l)$ and
$e_{-}(\cdot;\l)$ their expressions from~\eqref{eq:e=(I+K)e0}
one arrives at \eqref{eq:phi11}  and~\eqref{eq:phi21}. Relations
\eqref{eq:phi12}  and~\eqref{eq:phi22} are proved similarly.
  \end{proof}
%
%
\begin{lemma} \label{lem:Rxt.eiblt->0}
Let $N(\cdot,\cdot) \in X_{\infty,0}(\Omega)$, $b \in \bR
\setminus \{0\}$ and $h > 0$. Then the following asymptotic
holds  uniformly in $x \in [0,1]$
\begin{equation} \label{eq:int0x.Rxt.e.iblt->0}
    \int_{0}^x N(x,t) e^{i b \l t} dt \to 0
    \quad\text{as}\quad \l \to \infty, \quad |\Im \l| \leqslant h.
\end{equation}
\end{lemma}
%
%
\begin{proof}
By the definition of the space  $X_{\infty,0}(\Omega)$, the
inclusion $N \in X_{\infty,0}(\Omega)$ ensures that for any
$\eps > 0$ there exists $N_{\eps} \in C^1(\Omega)$ such that
\begin{equation}
    \| N - N_{\eps} \|_{X_{\infty}} = \esssup_{x \in [0,1]}
    \int_0^x \left| N(x,t) - N_{\eps}(x,t) \right| dt < \eps.
\end{equation}
In particular, we get the following uniform estimate
\begin{equation} \label{eq:int.R-Re}
    \left| \int_0^x \left( N(x,t) - N_{\eps}(x,t) \right) e^{i b \l t} dt \right|
    \leqslant \eps e^{|b|h}, \quad x \in [0,1], \quad |\Im \l| \leqslant h.
\end{equation}
Since $N_{\eps} \in C^1(\Omega)$, integrating by parts  the
integral $\int_0^x N_{\eps}(x,t) e^{i b \l t} dt$ we obtain  the
following estimate uniformly in $x \in [0,1]$ with some $C > 0$
\begin{equation} \label{eq:wtKjk.ebk<C/l}
    \left|\int_0^x N_{\eps}(x,t) e^{i b \l t} dt \right| < \frac{C}{|\l|},
    \qquad \l \ne 0, \quad |\Im \l| \leqslant h.
\end{equation}
The desired formula~\eqref{eq:int0x.Rxt.e.iblt->0} now directly
follows from estimates~\eqref{eq:int.R-Re}
and~\eqref{eq:wtKjk.ebk<C/l}.
\end{proof}
\begin{remark} We demonstrate that the assumption  $N(\cdot,\cdot)\in
X_{\infty,0}(\Omega)$ is important for the validity of the
statement of Lemma~\ref{lem:Rxt.eiblt->0}. More precisely, we
show that for certain $N(\cdot,\cdot)\in
X_{\infty}(\Omega)\setminus X_{\infty,0}(\Omega)$ the pointwice
convergence in~\eqref{eq:int0x.Rxt.e.iblt->0}  holds but is not
uniform in $x\in [0,1]$.

Let  $N(x,t)=\frac{1}{x}k(\frac{t}{x})$  where $k(\cdot)$
satisfies the conditions of Proposition~\ref{prop2.3}. Then
   \begin{equation}\label{3.11}
\int^x_0 N(x,t)e^{it\lambda}dt = \frac{1}{x}\int^x_0
k\left(\frac{t}{x}\right)e^{it\lambda}dt=\int^1_0 k(s)e^{i\lambda s x}ds\to 0\quad
\text{as}\ \l \to\infty,\  \l \in \Pi_h,
  \end{equation}
for each $x\in(0,1]$ and $\lambda\in\Pi_h$. However this
convergence   is not uniform in $x\in[0,1]$.  Indeed,  since
$k(\cdot)\not \equiv 0$, its Fourier transform $\widehat
k(\cdot)$  does  not vanish identically, i.e. there exists $a\in
\Bbb R$ such that $\widehat k(a)\not =0$. Therefore  for $\l\in
\Bbb R$ big enough and  $x = a/\l\in (0,1)$ the right hand side
of~\eqref{3.11} is $\widehat k(a)\not =0$.
\end{remark}
%
%
%
In the sequel we need the following result on the asymptotic
behavior of solutions of the system~\eqref{eq:system} in the
strip
$$
\Pi_h := \{ \l \in \bC : |\Im\,\l| \leqslant h \}.
$$
%
%
      \begin{proposition} \label{prop:phi.jk.asymp}
Let $Q \in L^1[0,1] \otimes \bC^{2 \times 2}$. Then for any $h >
0$ the following asymptotic relations take place \emph{uniformly
in} $x \in [0,1]$
\begin{equation} \label{eq:phi.jkx}
    \varphi_{jk}(x, \l) = \delta_{jk} e^{i b_k \l x} + o(1)
    \quad\text{as}\quad \l \to \infty, \quad \l \in \Pi_h,
    \quad j,k \in \{1,2\}.
\end{equation}
   \end{proposition}
  \begin{proof}
The proof  immediately follows  by combining
Proposition~\ref{prop:phi.jk=e+int} with
Lemma~\ref{lem:Rxt.eiblt->0}.
  \end{proof}
Applying the same approximation procedure  as has just been used
in the proof of Lemma \ref{lem:Rxt.eiblt->0} to the space
$L^1[0,1]$ we obtain  the following simple statement useful  in
the sequel.
%
%
\begin{lemma} \label{lem:RimLeb}
Let $g \in L^1[0,1]$ and $c \in \bC \setminus \{0\}$. Then for
any $\eps>0$ there exists $M = M_{\eps}> 0$ such that
\begin{equation}
    \left|\int_0^1 g(t) e^{c \l t} dt \right| < \eps (e^{\Re(c \l)} + 1),
    \quad |\l| > M.
\end{equation}
\end{lemma}

Further, consider the  adjoint system
\begin{equation}\label{3.15}
    -iB^{-1} y' + Q^*(x) y = \l y, \quad x \in [0,1],
\end{equation}
and introduce its  fundamental matrix solution
\begin{equation}\label{3.16}
    \Psi(\cdot, \l) =
    \begin{pmatrix} \psi_{11}(\cdot, \l)  & \psi_{12}(\cdot, \l)\\
    \psi_{21}(\cdot,\l) & \psi_{22}(\cdot,\l)
    \end{pmatrix} =: \begin{pmatrix} \Psi_1(\cdot, \l) & \Psi_2(\cdot, \l) \end{pmatrix},
    \qquad \Psi(0,\l)=I_2.
\end{equation}
Here  $\Psi_k(\cdot, \l)$ is the $k$th column of  $\Psi(\cdot,
\l)$.  Clearly, Proposition~\ref{prop:phi.jk.asymp} holds  for
the matrix solution $\Psi(\cdot, \l)$ as well.
Hence~\eqref{eq:phi.jkx} and similar relations for
$\psi_{jk}(\cdot, \l)$ imply the following result.
%
%
\begin{corollary} \label{cor:Phi.Psi}
Let $h>0$. Then for $\l \to \infty$, $\l \in \Pi_h$, the
following asymptotic relations hold
\begin{align}\label{3.17}
    \left(\Phi_{j}(\cdot,\l), \Psi_{k}(\cdot,\overline{\l})\right) & = \delta_{jk} + o(1),
    \quad j,k \in \{1,2\}, \\
    \left(\Phi_{1}(\cdot,\l), \Phi_{2}(\cdot,\l)\right) & = o(1). \label{3.18}
\end{align}
Moreover, there exist constants  $M > 0$ and $C_1, C_2 > 0$,
such that
\begin{equation} \label{eq:Phij.Phij}
  0 <   C_1 < \left|\left(\Phi_{j}(\cdot,\l), \Phi_{j}(\cdot,\l)\right)\right| < C_2,
    \quad  \l \in \Pi_h, \quad |\l|
> M,  \quad j \in \{1,2\}.
  \end{equation}
\end{corollary}
%
%
\begin{proof}
First let us evaluate $\left(\Phi_{1}(\cdot,\l),
\Phi_{1}(\cdot,\l)\right)$.  Setting $f(x) := \frac{e^x-1}{x}$,
$x\in \mathbb R,$ and noting that  $e^{i a \l x}$, $a\in \mathbb
R$, is bounded for $(x,\l) \in [0,1] \times \Pi_h$,  one easily
deduces from the \emph{uniform} asymptotic
relations~\eqref{eq:phi.jkx}
\begin{align} \label{eq:Phi1.Phi1}
 \|\Phi_{1}(\cdot,\l)\|^2 =   \left(\Phi_{1}(\cdot,\l), \Phi_{1}(\cdot,\l)\right)
    &= \int_0^1 \left( \varphi_{11}(x, \l) \ol{\varphi_{11}(x, \l)} +
                       \varphi_{12}(x, \l) \ol{\varphi_{12}(x, \l)} \right) dx \nonumber \\
    &= \int_0^1 \left(e^{i b_1 (\l - \ol{\l})x} + o(1)\right)dx
     = \frac{e^{-2 b_1 \Im\l} - 1}{-2 b_1 \Im\l} + o(1) \nonumber \\
    &= f(-2 b_1 \Im\l) + o(1) \quad\text{as}\quad  \l \to \infty, \quad  \l \in \Pi_h.
\end{align}
Clearly, there exists $C_1, C_2 > 0$ such that
\begin{equation} \label{eq:C1<|f|<C2}
    C_1 \leqslant |f(x)| \leqslant C_2, \quad |x| \leqslant 2 |b_1| h.
\end{equation}
Combining~\eqref{eq:Phi1.Phi1} with~\eqref{eq:C1<|f|<C2} one
proves~\eqref{eq:Phij.Phij} for $j=1$. Relation
\eqref{eq:Phi1.Phi1} for $j=2$ as well as relations
\eqref{3.17},~\eqref{3.18} are proved similarly.
\end{proof}
%
%
\section{Regular  boundary conditions} \label{sec:Regular}
%
%
Here we consider $2\times 2$-Dirac type
equation~\eqref{eq:system},
\begin{equation}\label{eq:system}
    -i B^{-1} y'+Q(x)y=\l y, \qquad y={\rm col}(y_1,y_2), \qquad x\in[0,1],
\end{equation}
subject to the following general boundary conditions
\begin{equation}\label{eq:BC}
    U_j(y) := a_{j 1}y_1(0) + a_{j 2}y_2(0) + a_{j 3}y_1(1) + a_{j 4}y_2(1)= 0,
    \quad  j \in \{1,2\}.
\end{equation}
Denote by $L := L(Q, U_1, U_2)$ the operator associated in
$L^2([0,1]; \bC^2)$ with the
BVP~\eqref{eq:system}--\eqref{eq:BC}. It is defined as the
restriction of the maximal operator $L_{\max} =
L_{\max}(Q)$~\eqref{Max_oper_Intro} to the domain
\begin{equation} \label{eq:dom}
   \dom(L) = \dom(L(Q, U_1, U_2)) = \{y \in \dom(L_{\max}) : U_1(y) = U_2(y) = 0\}.
\end{equation}

The eigenvalues of the problem~\eqref{eq:system}--\eqref{eq:BC}
are the roots of the characteristic equation $\Delta(\l) := \det
U(\l)=0$, where
\begin{equation}\label{eq:U}
    U(\l) :=
    \begin{pmatrix}
        U_1(\Phi_1(\cdot,\l)) & U_1(\Phi_2(\cdot,\l)) \\
        U_2(\Phi_1(\cdot,\l)) & U_2(\Phi_2(\cdot,\l))
    \end{pmatrix}
    =: \begin{pmatrix} u_{11}(\l) & u_{12}(\l) \\ u_{21}(\l) & u_{22}(\l) \end{pmatrix}.
\end{equation}
Putting $A_{jk} =
    \begin{pmatrix}
        a_{1j}&a_{1k} \\
        a_{2j}&a_{2k}
    \end{pmatrix}$, and $J_{jk} = \det (A_{jk}), \ j,k\in\{1,\ldots,4\}$,
we  obtain  the following expression for the characteristic
determinant
\begin{equation}\label{eq:Delta}
    \Delta(\l) = J_{12} + J_{34}e^{i(b_1+b_2)\l}
    + J_{32}\varphi_{11}(\l) + J_{13}\varphi_{12}(\l)
    + J_{42}\varphi_{21}(\l) + J_{14}\varphi_{22}(\l),
\end{equation}
where $\varphi_{jk}(\l) := \varphi_{jk}(1,\l)$. If $Q=0$ then
$\varphi_{12}(x,\l) = \varphi_{21}(x,\l) = 0$ and the
characteristic determinant $\Delta_0(\cdot)$ becomes
\begin{equation}\label{eq:Delta0}
    \Delta_0(\l) = J_{12} + J_{34}e^{i(b_1+b_2)\l}
    + J_{32}e^{ib_1\l} + J_{14}e^{ib_2\l}.
\end{equation}
In the case of Dirac system $(B =\diag (-1,1))$ this formula is
simplified to
\begin{equation} \label{eq:Delta0_Dirac}
    \Delta_0(\l) = J_{12} + J_{34} + J_{32}e^{-i\l} + J_{14}e^{i\l}.
\end{equation}
Substituting formulas~\eqref{eq:phi11}--\eqref{eq:phi22} at
$x=1$ to~\eqref{eq:Delta} and taking into
account~\eqref{eq:Delta0}, we get the following expression for
the characteristic determinant.
%
%
\begin{lemma} \label{lem:Delta=Delta0+}
The characteristic determinant $\Delta(\cdot)$  of the
problem~\eqref{eq:system}--\eqref{eq:BC} is an entire function
admitting  the following representation
\begin{equation}\label{eq:Delta=Delta0+}
    \Delta(\l) = \Delta_0(\l)
    + \int^1_0 g_1(t) e^{i b_1 \l t} dt
    + \int^1_0 g_2(t) e^{i b_2 \l t} dt,
\end{equation}
with  $g_1, g_2 \in L^1[0,1]$.
\end{lemma}
%
%
\begin{proof}
Consider  representations~\eqref{eq:phi11}--\eqref{eq:phi22} for
$\varphi_{jk}(\cdot,\l)$, $j,k \in \{1,2\}$. By
Proposition~\ref{prop:phi.jk=e+int},
$R_{jk}^{\pm}(\cdot,\cdot)\in X_{1,0}(\Omega) \cap
X_{\infty,0}(\Omega),$ \ $j,k \in \{1,2\}$. Therefore by Lemma
\ref{Trace_lemma}, the trace functions $R_{jk}^{\pm}(1,\cdot)$
are well defined and summable, $R_{jk}^{\pm}(1,\cdot)\in
L^1[0,1]$,  $j,k \in \{1,2\}$. Therefore one can substitute
$x=1$ in formulas~\eqref{eq:phi11}--\eqref{eq:phi22}  and obtain
special representations for $\varphi_{jk}(\cdot)$, $j,k \in
\{1,2\}$. For instance,
\begin{equation}
  \varphi_{jk}(\l) :=   \varphi_{jk}(1,\l) =
         \int_0^1 R_{jk}^{+}(1,t)e^{ib_k\l t}dt + \int_0^1 R_{jj}^{-}(1,t)e^{ib_j\l
         t}dt,\quad j\not =k.
\end{equation}
Inserting these expressions and similar expressions for
$\varphi_{jj}(\cdot)$ in \eqref{eq:Delta} and taking
formula~\eqref{eq:Delta0} for $\Delta_0(\cdot)$ into account we
arrive at~\eqref{eq:Delta=Delta0+} with  $g_j(\cdot)$, $j \in
\{1,2\}$, being  a linear combination of the functions
$R_{jk}^{\pm}(1,\cdot)$, $j,k \in \{1,2\}$.
\end{proof}
In the sequel we need the following definitions
(cf.~\cite{Katsn71}).
%
%
\begin{definition} \label{def:sequences}
\textbf{(i)} A sequence $\Lambda := \{\l_n\}_{n \in \bZ}$   of
complex numbers is said to be \textbf{separated} if for some
positive $\delta > 0,$
  \begin{equation}\label{separ_cond}
|\l_j - \l_k| > 2 \delta \quad \text{whenever}\quad  j \ne k.
   \end{equation}
In particular,  all entries of a separated  sequence are
distinct.

\textbf{(ii)}  The sequence $\Lambda$  is said to be
\textbf{asymptotically separated} if for some $n_0 \in \bN$ the
subsequence $\Lambda_{n_0} := \{\l_n\}_{|n| > n_0}$ is
separated.

\textbf{(iii)} Let $\Lambda$ lie in the strip $\Pi_h$. It is
called \textbf{incondensable} if for some $L > 0$ and $N \in
\bN$ every rectangle $[t-L,t+L] \times [-h,h] \subset \bC$
contains at  most $N$ entries of the sequence, i.e. for each $t
\in \bR$  the number of integers $\{n \in \bZ : |\Re \l_n - t|
\leqslant L, |\Im \l_n| \leqslant h\}$ does not exceed  $N$.
\end{definition}
%
%
We need the following simple property of incondensable
sequences.
%
%
\begin{lemma} \label{lem:incondensable}
Let $\Lambda = \{\l_n\}_{n \in \bZ}$ be an incondensable
sequence lying in the strip $\Pi_h$. Then there exists $\eps_0 >
0$ and $N_0 \in \bN$ such that for any $\eps \in (0, \eps_0)$
every connected component of the union of discs $\cup_{n \in
\bZ} \bD_{\eps}(\l_n)$ has at most $N_0$ discs
$\bD_{\eps}(\l_n)$.
\end{lemma}
%
%
\begin{proof}
Assume the contrary, i. e., for any $\eps > 0$ and $K \in \bN$
there exists connected component of $\Omega_{\eps} := \cup_{n
\in \bZ} \bD_{\eps}(\l_n)$ that has at least $K$ discs
$\bD_{\eps}(\l_n)$. By definition of incondensable sequence for
some $L, N > 0$ every rectangle $[t-L,t+L] \times [-h, h]$, $t
\in \bR$, contains at most $N$ entries of the sequence
$\Lambda$. Let $K > N$ be some positive integer and pick $\eps$
to be such that $2K\eps < L$. Consider some connected component
of $\Omega_{\eps}$ that has $M \ge K$ discs $\bD_{\eps}(\l_n)$,
denote it by $C$. Let's pick one of the discs $D_0 =
\bD_{\eps}(\l_{n_0})$, $n_0 \in \bZ$, in $C$, and let $t_0 = \Re
\l_{n_0}$. Due to above, rectangle $[t_0 - L, t_0 + L] \times
[-h, h]$ contains at most $N$ entries of the sequence $\Lambda$.
Consider the sequence $B$ of all discs $\bD_{\eps}(\l_n)$ in $C$
that have graph distance at most $K$ from $D_0$. Let's shows
that $B$ has more than $N$ discs. If no disc in $C$ has graph
distance at least $K$ from $D_0$, then $B$ contains all discs
from $C$ and thus cardinality of $B$ is $M \ge K > N$.
Otherwise, $B$ has some disc $D$ with distance $K$ from $D_0$.
All discs on the path from $D_0$ to $D$ belong to $B$ and hence
$B$ has at least $K > N$ discs. For each disc $\bD_{\eps}(\l_n)$
in $B$ since graph distance from it to $D_0$ is at most $K$ and
disc radii are $\eps$ we have $|\l_n - \l_{n_0}| < 2 K \eps$.
Thus, $\l_n \in [t_0 - L, t_0 + L] \times [-h, h]$ since $2 K
\eps < L$. Thus centers of all discs in $B$ lie in $[t_0 - L,
t_0 + L] \times [-h, h]$. Since there more than $N$ discs in $B$
it contradicts incondensability property of the sequence
$\Lambda$.
\end{proof}
To get the asymptotic behavior of the eigenvalues of the
problem~\eqref{eq:system}--\eqref{eq:BC} with regular boundary
conditions we also need the following definition.
%
%
\begin{definition}\cite{Lev61} \label{def:regular}
An entire function $F(\cdot)$ of exponential type is said to be
of \textbf{sine-type} if

\textbf{(i)} all zeros of $F(\cdot)$  lie in the strip $\Pi_h$
for some $h>0$, and

\textbf{(ii)}  there exists  $C_1, C_2 > 0$ and $h_0 > h$ such
that
\begin{equation} \label{eq:C1<|Fz|<C2}
     0 < C_1 \leqslant |F(x + ih_0)| \leqslant C_2 <\infty, \quad x\in \Bbb R.
\end{equation}
\end{definition}
%
%
This definition is borrowed from~\cite{Lev61} (see
also~\cite{Katsn71}). It  differs from that contained
in~\cite{Lev96}. Namely, it is  assumed in~\cite{Lev96} that the
sequence of zeros of $F(\cdot)$ is \emph{separated}  and the
indicator function $h_F(\cdot)$ of $F(\cdot)$,
\begin{equation} \label{eq:hF.phi.def}
    h_F(\varphi) := \varlimsup_{r \to +\infty}
    \frac{\ln\left|F\left(r e^{i \varphi}\right)\right|}{r},
    \quad \varphi \in (-\pi,\pi],
\end{equation}
satisfies the  condition $h_F(\pi/2) = h_F(-\pi/2)$. The latter
is imposed for convenience and can easily be achieved by
multiplication of $F(\cdot)$ by a function $e^{i\gamma z}$ with
an appropriate $\gamma \in \mathbb R.$

Recall also the definition of regular boundary conditions.
%
%
\begin{definition} \label{def:regular}
Boundary conditions~\eqref{eq:BC}  are called \textbf{regular}
if
\begin{equation} \label{eq:J32J14ne0}
    J_{14} J_{32} \ne 0.
\end{equation}
\end{definition}
%
%
In the case of regular boundary conditions  the characteristic
determinant $\Delta_0(\cdot)$ has certain  important properties.
%
%
\begin{proposition} \label{prop:sine.type}
Let the boundary conditions~\eqref{eq:BC} be regular and let
$\Delta(\cdot)$ be  the characteristic determinant of the
problem~\eqref{eq:system}--\eqref{eq:BC} given by
\eqref{eq:Delta}. Then the following hold:

\textbf{(i)} The characteristic determinant $\Delta(\cdot)$ is a
sine-type function with $h_{\Delta}(\pi/2) = -b_1$ and
$h_{\Delta}(-\pi/2) = b_2$. In particular, $\Delta(\cdot)$ has
infinitely many  zeros
\begin{equation} \label{eq:Lam0.def}
    \Lambda := \{\l_n\}_{n \in \bZ}
\end{equation}
counting multiplicities  and $\Lambda \subset \Pi_h$ for some
$h>0$.

\textbf{(ii)} The sequence $\Lambda$ is incondensable.

\textbf{(iii)} For any $\eps > 0$  the determinant
$\Delta(\cdot)$ admits the following estimate from below
\begin{equation} \label{eq:Delta0>=}
    |\Delta(\l)| \geqslant C_{\eps}(e^{-b_1 \Im \l} + e^{-b_2 \Im \l}),
    \quad \l \in \bC \setminus \bigcup_{n \in \bZ} \bD_{\eps}(\l_n),
\end{equation}
with some $C_{\eps} > 0$.

\textbf{(iv)} The sequence $\Lambda$ can be ordered in such a
way that the following asymptotical formula  holds
\begin{equation} \label{eq:lam.n=an+o1}
    \l_n = \frac{2 \pi n}{b_2 - b_1} (1 + o(1)) \quad\text{as}\quad n \to \infty.
\end{equation}
\end{proposition}
%
%
\begin{proof}
\textbf{(i).} Let $\Delta_0(\cdot)$  be the characteristic
determinant of the problem~\eqref{eq:system}--\eqref{eq:BC} with
$Q=0$. It easily follows from \eqref{eq:Delta0} that
$\Delta_0(\cdot)$ admits a representation
   \begin{equation}\label{4.12}
\Delta_0(\l)  = \int_{b_1}^{b_2}e^{it\l}d\sigma_0(t), \qquad \l\in \mathbb C,
    \end{equation}
with a  piecewise constant  function  $\sigma_0(\cdot)$ having
precisely four jump-points $\{0, b_1, b_1 + b_2, b_2\}$. In
particular,
  \begin{equation}\label{4.13}
\sigma_0(b_1+0) - \sigma_0(b_1) = J_{32}\not = 0\quad  \text{and} \quad \sigma_0(b_2) -
\sigma_0(b_2 -0) = J_{14} \not = 0.
   \end{equation}
Let us set
\begin{equation}
g(t)=
\begin{cases}
-\frac{1}{b_1}g_1(\frac{t}{b_1}),&t\in[b_1, 0),  \\
\frac{1}{b_2}g_2(\frac{t}{b_2}),&t\in[0,b_2],
\end{cases}
  \end{equation}
and
  \begin{equation}\label{4.15}
\sigma(t) = \sigma_0(t) + \int_{b_1}^t g(s)ds.
  \end{equation}
Combining these notations with formulas~\eqref{eq:Delta=Delta0+}
and~\eqref{4.12} we arrive at the following  representation for
the characteristic  determinant
   \begin{equation}\label{4.12New}
\Delta(\l)  = \int_{b_1}^{b_2}e^{it\l}d\sigma(t), \qquad \l\in \mathbb C,
    \end{equation}
It follows from~\eqref{4.15} and~\eqref{4.13} that
     \begin{equation}\label{4.17}
 \sigma(b_1+0) - \sigma(b_1)  = J_{32}\not = 0 \quad \text{and}\quad  \sigma(b_2) -
\sigma(b_2 -0) = J_{14} \not = 0.
     \end{equation}
Due to the property~\eqref{4.17} representation~\eqref{4.12New}
ensures that $\Delta(\cdot)$ is  a sine-type function with
$h_{\Delta_0}(\pi/2) = -b_1$ and $h_{\Delta_0}(-\pi/2) = b_2$
(see~\cite{Lev96}). Moreover, statement \textbf{(i)} is also
implied by the representation~\eqref{4.12New} (see~\cite[Chapter
1.4.3]{Leon76}).

\textbf{(ii)} and \textbf{(iii)}. These statements  coincide
with the corresponding statements of~\cite[Lemmas 3 and
4]{Katsn71}  for sine-type functions (see also \cite[Lemma
22.1]{Lev96} in connection with part \textbf{(iii)}).

\textbf{(iv)} The determinant $\Delta(\cdot)$ belongs to the
class $A$  since its zeros lie in the strip $\Pi_h$ (this fact
is also immediate from representation~\eqref{4.12New}).
Therefore it follows from~\cite[Theorem 1.4.6]{Leon76} that for
any $\eps \in (0, \pi/2)$
\begin{equation} \label{eq:nt/t}
    \lim_{t \to \infty} \frac{n_{\pm}^{(\eps)}(t)}{t} = \frac{2\pi}{b_2-b_1}.
\end{equation}
Here $n_+^{(\eps)}(t) = \card \left\{n \in \bZ : |\l_n| < t,
|\arg \l_n| < \eps \right\}$ is the number of zeros of
$\Delta(\cdot)$ in the domain $\{z : |\arg z| < \eps, |z| < t\}$
counting multiplicity, and $n_-^{(\eps)}(t) = \card \left\{n \in
\bZ : |\l_n| < t, |\pi - \arg \l_n| < \eps \right\}$. Since
$\Lambda$ lies in the strip $\Pi_h$, asymptotic
formula~\eqref{eq:lam.n=an+o1} directly follows
from~\eqref{eq:nt/t} (see e.g.~\cite[Proposition
13.1]{Shubin87}).
\end{proof}
Clearly, the conclusions of Proposition~\ref{prop:sine.type} are
valid for the perturbation determinant  $\Delta_0(\cdot)$ given
by~\eqref{eq:Delta0}. Let $\Lambda_0 = \{\l_n^0\}_{n \in \bZ}$
be the sequence  of its zeros counting multiplicity.  Let us
order the sequence $\Lambda_0$  in a (possibly non-unique) way
such that $\Re \l_n^0 \le \Re \l_{n+1}^0$, $n\in \bZ.$
%
%
\begin{proposition} \label{prop:Delta.regular.basic}
Let $Q \in L^1[0,1] \otimes \bC^{2 \times 2}$,  let boundary
conditions~\eqref{eq:BC} be regular, and let $\Delta(\cdot)$ be
the corresponding characteristic determinant. Then the sequence
$\Lambda = \{\l_n\}_{n \in \bZ}$ of its zeros can be ordered in
such a way that the following asymptotic formula holds
\begin{equation} \label{eq:l.n=l.n0+o(1)}
    \l_n = \l_n^0 + o(1), \quad\text{as}\quad n \to \infty, \quad n \in \bZ.
\end{equation}
\end{proposition}
%
%
\begin{proof}
Let $\eps \in (0, 1)$. By Proposition~\ref{prop:sine.type}(iii)
there exists $C_{\eps}
> 0$ such that the estimate~\eqref{eq:Delta0>=} holds.
Combining  Lemma~\ref{lem:Delta=Delta0+} with
Lemma~\ref{lem:RimLeb} yields the following estimate
\begin{equation} \label{eq:Delta-Delta0}
    |\Delta(\l) - \Delta_0(\l)| < 4^{-1}C_{\eps} (e^{-b_1 \Im \l} + e^{-b_2 \Im \l}+ 2\eps)
    \le 2^{-1}C_{\eps} (e^{-b_1 \Im \l} + e^{-b_2 \Im \l}).
        \quad |\l| \geqslant \  M_{\eps}.
\end{equation}
with certain $M_{\eps} > 0$. Here in the last inequality we have
used that $b_1 < 0 < b_2$.

Due to estimates~\eqref{eq:Delta0>=}
and~\eqref{eq:Delta-Delta0}, the Rouche theorem implies that all
zeros of $\Delta(\cdot)$ lie in the domain
\begin{equation}
    \wt{\Omega}_{\eps} := \bD_{M_{\eps}}(0) \cup \Omega_{\eps}, \quad
    \Omega_{\eps} := \bigcup\limits_{n \in \bZ} \bD_{\eps}(\l_n^0),
\end{equation}
and in each connected component of $\wt{\Omega}_{\eps}$ the
functions  $\Delta(\cdot)$ and  $\Delta_0(\cdot)$ have  the same
number of zeros counting multiplicity. Since in accordance with
Proposition \ref{prop:sine.type}(ii), the sequence of zeros
$\{\l_n^0\}_{n \in \bZ}$ is incondensable,
Lemma~\ref{lem:incondensable} implies that for $\eps$ small
enough each connected component of $\Omega_{\eps}$ contains at
most $N_0$ discs $\bD_{\eps}(\l_n^0)$ with $N_0$ not depending
on $\eps$. Hence the diameter of each connected component of
$\Omega_{\eps}$ does not exceed $2 \eps N_0$. Since ${\eps}>0$
is arbitrary small, the latter implies the desired asymptotic
formula~\eqref{eq:l.n=l.n0+o(1)}.
\end{proof}
%
%
\section{Strictly regular boundary conditions}
\label{sec:StrictRegular}
%
%
Assuming boundary conditions~\eqref{eq:BC} to be regular, let us
rewrite them in a more convenient form. Since $J_{14} \ne 0$,
the inverse matrix $A_{14}^{-1}$ exists. Therefore writing down
boundary conditions~\eqref{eq:BC} as  the vector equation
$\binom {U_1(y)} {U_2(y)} = 0$ and multiplying  it by the matrix
$A_{14}^{-1}$ we  transform them as follows
  \begin{equation} \label{eq:BC.new}
\begin{cases}
   {\widehat U}_{1}(y) = y_1(0) + b y_2(0) + a y_1(1) &= 0, \\
    {\widehat U}_{2}(y) = d y_2(0) + c y_1(1) + y_2(1) &= 0,
\end{cases}
\end{equation}
with some $a,b,c,d \in \bC$. Now $J_{14} = 1$ and  the boundary
conditions ~\eqref{eq:BC.new} are regular if and only if
$J_{32} = ad-bc \ne 0$. So, the characteristic determinants
$\Delta_0(\cdot)$ and $\Delta(\cdot)$ take the form
\begin{align}
    \label{eq:Delta0.new}
    \Delta_0(\l) &= d + a e^{i (b_1+b_2) \l} + (ad-bc) e^{i b_1 \l} + e^{i b_2 \l}, \\
    \label{eq:Delta.new}
    \Delta(\l)   &= d + a e^{i (b_1+b_2) \l} + (ad-bc) \varphi_{11}(\l) + \varphi_{22}(\l)
    + c \varphi_{12}(\l) + b \varphi_{21}(\l).
\end{align}

Now we are ready to introduce a notion of strictly regular
boundary conditions.
%
%
\begin{definition} \label{def:strictly.regular}
Boundary conditions~\eqref{eq:BC} are called \textbf{strictly
regular}, if they are regular, i.e. $J_{14} J_{32} \ne 0$, and
the sequence of zeros  $\Lambda_0 = \{\l_n^0\}_{n \in \bZ}$ of
the characteristic determinant $\Delta_0(\cdot)$ is
asymptotically separated. In particular, there exists $n_0$ such
that zeros $\{\l_n^0\}_{|n| > n_0}$  are geometrically and
algebraically simple.
\end{definition}
%
%
It follows from Proposition~\ref{prop:Delta.regular.basic} that
the sequence $\Lambda = \{\l_n\}_{n \in \bZ}$ of zeros of
$\Delta(\cdot)$ is asymptotically separated if the boundary
conditions are strictly regular.
%
%
\begin{remark} \label{rem:cond.examples}
Let us list some types of \emph{strictly regular} boundary
conditions~\eqref{eq:BC.new}. In all of these cases the set of
zeros of $\Delta_0$ is a union of finite number of arithmetic
progressions.

\textbf{(i)} Separated boundary conditions ($a=d=0$, $bc \ne 0$)
are always strictly regular.

\textbf{(ii)} Let $b_1 / b_2 \in \bQ$, i.e. $b_1 = -n_1 b$, $b_2
= n_2 b$, $n_1, n_2 \in \bN$, $b>0$. Since $ad \ne bc$,
$\Delta_0(\cdot)$ is a polynomial at $e^{i b \l}$ of degree $n_1
+ n_2$. Hence, boundary conditions~\eqref{eq:BC.new} are
strictly regular if and only if this polynomial does not have
multiple roots. In particular, regular boundary
conditions~\eqref{eq:BC.new} for Dirac operator are strictly
regular if and only if  $(a-d)^2 \ne -4bc$.
\end{remark}
%
%
\begin{lemma} \label{lem:bc=0.regular}
Let $bc=0$ and boundary conditions~\eqref{eq:BC.new} are regular
(i.e. $ad \ne 0$).

\textbf{(i)} Let
\begin{equation} \label{eq:b1.lnd+b2.lna}
    b_1 \ln |d| + b_2 \ln |a| \ne 0.
\end{equation}
Then  conditions~\eqref{eq:BC.new} are strictly regular.

\textbf{(ii)} Let $b_1/b_2 \not\in \bQ$. Then
condition~\eqref{eq:b1.lnd+b2.lna} is necessary for the strict
regularity of boundary conditions~\eqref{eq:BC.new}.

\textbf{(iii)} Let $b_1/b_2 \in \bQ$ and
condition~\eqref{eq:b1.lnd+b2.lna} is violated. Then boundary
conditions~\eqref{eq:BC.new} are strictly regular if and only if
\begin{equation} \label{eq:gcd.b1.b2}
    \frac{b_1 \arg(-d) + b_2 \arg(-a)}{2 \pi \gcd(b_1, b_2)} \not\in \bZ,
\end{equation}
where $\gcd(b_1, b_2)$ is the greatest common divisor of real
numbers $b_1$, $b_2$, i.e. the largest number $b>0$ such that
$b_1/b$ and $b_2/b$ are integers.
\end{lemma}
%
%
\begin{proof}
\textbf{(i)} Since $bc=0$, the characteristic determinant $
\Delta_0(\cdot)$ in~\eqref{eq:Delta0.new} becomes
\begin{equation} \label{eq:Delta0.bc=0}
    \Delta_0(\l) = d + a e^{i (b_1+b_2) \l} + a d e^{i b_1 \l} + e^{i b_2 \l}
                 = (1 + a e^{i b_1 \l}) (d + e^{i b_2\l}).
\end{equation}
Let $\Lambda_1 = \{\l_{1,n}\}_{n \in \bZ}$ and $\Lambda_2 =
\{\l_{2,n}\}_{n \in \bZ}$ be the sequences of zeros of the first
and second factor, respectively. Clearly,
\begin{equation} \label{eq:l1n.l2n.bc=0}
    \l_{1,n} = \frac{\arg(-a^{-1}) + 2 \pi n}{b_1} + i\frac{\ln|a|}{b_1}, \qquad
    \l_{2,n} = \frac{\arg(-d) + 2 \pi n}{b_2} - i\frac{\ln|d|}{b_2}, \qquad n \in \bZ.
\end{equation}
Thus, $\Lambda_1$ and $\Lambda_2$ are algebraically simple and
constitute  two arithmetic progressions that lie on two lines
parallel to the real axis. Condition~\eqref{eq:b1.lnd+b2.lna}
written in the form $\frac{\ln|a|}{b_1} \ne -\frac{\ln|d|}{b_2}$
implies  that  these horizontal lines are different. It follows
that the sequence of zeros of $\Delta_0(\cdot)$ is separated and
hence boundary conditions~\eqref{eq:BC.new} are strictly
regular.

\textbf{(ii)} Now assume that $\alpha := -b_1/b_2 \not \in \bQ$
and condition~\eqref{eq:b1.lnd+b2.lna} is violated. In this case
$\Im \l_{1,n} = \Im \l_{2,m} = \frac{\ln|a|}{b_1} =
-\frac{\ln|d|}{b_2}$  for  $n,m \in \bZ$, i.e. the progressions
$\Lambda_1$ and $\Lambda_2$ lie on the same line parallel to the
real axis. Hence
\begin{equation} \label{eq:l1n-l2m}
    |\l_{1,n} - \l_{2,m}| = 2 \pi b_1^{-1} | r + n + \alpha m|, \quad
    r = \frac{\arg(-a^{-1})+\alpha \arg(-d)}{2 \pi} \in \bR.
\end{equation}
Since $\alpha$ is irrational, the Kronecker theorem ensures that
for any $\eps > 0$ and  $M > 0$ there exist $n,m \in \bZ$ such
that $|n|, |m| > M$ and $|r + n + \alpha m| < \eps$. This means
that the zeros of $\Delta_0(\cdot)$ are not asymptotically
separated, which  proves  the result.

\textbf{(iii)} Finally, assume that
condition~\eqref{eq:b1.lnd+b2.lna} is violated while   $\alpha =
- b_1/b_2 \in \bQ$. Now as on the previous step,  the
progressions  $\Lambda_1$ and $\Lambda_2$ lie on the same line
parallel to the real axis  and  condition \eqref{eq:l1n-l2m}
holds. Since $b_1/b_2$ is rational, the union of the arithmetic
progressions $\Lambda_1$ and $\Lambda_2$ is asymptotically
separated if and only if they have no common entries. Due
to~\eqref{eq:l1n-l2m} this is equivalent to the fact that
Diophantine equation $n + \alpha m = -r$ does not have integer
solutions $n,m$. It is well-known that such equation has
solutions if and only if $r / \gcd(\alpha, 1) \in \bZ$. Since
$\arg(-a^{-1}) = -\arg(-a)$, this is equivalent to the condition
opposite to~\eqref{eq:gcd.b1.b2}, which completes the proof.
\end{proof}
\begin{remark}
 Consider the case  $b = c=0$. It includes
\emph{periodic} $(a=d= -1)$ and \emph{antiperiodic} $(a=d= 1)$
boundary conditions. So, it follows from the statements (i)
$($see~\eqref{eq:b1.lnd+b2.lna}$)$ and (ii)  that the
\emph{periodic}  and \emph{antiperiodic} BC are \textbf{strictly
regular} if and only if  $b_1 + b_2\not =0$. This fact
demonstrates a substantial difference between Dirac and Dirac
type  operators.
  \end{remark}
The following result demonstrates that in the case $b_1 / b_2
\not\in \bQ$ the problem of strict regularity of boundary
conditions~\eqref{eq:BC.new} is much  more complicated than the
one discussed  in Remark~\ref{rem:cond.examples} and
Lemma~\ref{lem:bc=0.regular}.
%
%
\begin{proposition}\label{prop_strict_regular}
\textbf{(i)} Let $\alpha := -b_1/b_2 \not \in \bQ$, $a=0$, $bc,
d \in \bR \setminus \{0\}$, then boundary
conditions~\eqref{eq:BC.new} are strictly regular if and only if
\begin{equation} \label{eq:a=0.crit}
    d \ne -(\alpha+1)\left(|bc| \alpha^{-\alpha}\right)^{\frac{1}{\alpha+1}}.
\end{equation}

\textbf{(ii)} Let $a=0$, $bc \ne 0$, $b_1 = -n_1 b$, $b_2 = n_2
b$, $n_1, n_2 \in \bN$, $b>0$ and $\gcd(n_1,n_2)=1$. Then
boundary conditions~\eqref{eq:BC.new} are strictly regular if
and only if
\begin{equation} \label{eq:a=0.crit.rat}
    n_1^{n_1} n_2^{n_2} (-d)^{n_1 + n_2} \ne (n_1 + n_2)^{n_1 + n_2} (-b c)^{n_2}.
\end{equation}
\end{proposition}
%
%
\begin{proof}
\textbf{(i)} It follows from~\eqref{eq:Delta0.new} that
\begin{equation} \label{eq:Delta0.bc=1}
    \Delta_0(\l) = d - bc \cdot e^{i b_1 \l} + e^{i b_2 \l}.
\end{equation}
Let $b_2 \l =: \pi x + i y$, $x, y \in \bR$. Then $e^{i b_1 \l}
= e^{-i \alpha \pi x + \alpha y}$, $e^{i b_2 \l} = e^{i \pi x -
y}$. Since $bc, d \in \bR$, equation $\Delta_0(\l) = 0$ is
equivalent to the system
\begin{equation} \label{eq:e-y.system}
\begin{cases}
    e^{-y} \cos \pi x &= bc \cdot e^{\alpha y} \cos \alpha \pi x - d, \\
    e^{-y} \sin \pi x &= bc \cdot e^{\alpha y} \sin \alpha \pi x.
\end{cases}
\end{equation}
From the second equation in~\eqref{eq:e-y.system} we have
\begin{equation} \label{eq:y=frac}
    y = - \frac{\ln\left(\frac{-bc \cdot \sin \alpha \pi x}{\sin \pi x}\right)}{\alpha+1}, \qquad
    \frac{bc \cdot \sin \alpha \pi x}{\sin \pi x} < 0.
\end{equation}
Substituting it to the first equation we get
\begin{equation}
    \left(\frac{-bc \cdot \sin \alpha \pi x}{\sin \pi x}\right)^{\frac{1}{\alpha+1}} \cos \pi x =
    bc \left(\frac{-bc \cdot \sin \alpha \pi x}{\sin \pi x}\right)^{\frac{-\alpha}{\alpha+1}} \cos \alpha \pi x - d,
\end{equation}
which is equivalent to
\begin{equation} \label{eq:a=0.fx.def}
    \left(\frac{-bc \cdot \sin \alpha \pi x}{\sin \pi x}\right)^{\frac{1}{\alpha+1}}
    \frac{\sin (\alpha+1) \pi x}{\sin \alpha \pi x} = -d.
\end{equation}
For simplicity we assume that $bc=-1$ and $(0<)\alpha<1$. Then
summarizing all previous formulas we see that $\l$ is zero of
$\Delta_0(\cdot)$ if and only if
\begin{equation} \label{eq:a=0.final.system}
\begin{cases}
    b_2 \l = \pi x + i y, \quad x, y \in \bR, \\
    \frac{\sin \alpha \pi x}{\sin \pi x} > 0, \\
    y = - \frac{\ln\left(\frac{\sin \alpha \pi x}{\sin \pi x}\right)}{\alpha+1}, \\
    f(x) := \left(\frac{\sin \alpha \pi x}{\sin \pi x}\right)^{\frac{1}{\alpha+1}}
    \frac{\sin (\alpha+1) \pi x}{\sin \alpha \pi x} = -d.
\end{cases}
\end{equation}
Note that the second relation in~\eqref{eq:a=0.final.system} is
equivalent to
\begin{align} \label{eq:a=0.x.in...}
    x \in & A:= \left(\left(\bigcup_{n \in \bZ} (2n, 2n+1)\right) \bigcap
                \left(\bigcup_{n \in \bZ} \left(\frac{2n}{\alpha}, \frac{2n+1}{\alpha}\right)\right)\right) \nonumber \\
    &\bigcup \left(\left(\bigcup_{n \in \bZ} (2n-1, 2n)\right) \bigcap
                \left(\bigcup_{n \in \bZ} \left(\frac{2n-1}{\alpha}, \frac{2n}{\alpha}\right)\right)\right).
\end{align}
Let us describe the set $A$ in a more explicit way. Let $n \in
\bZ$ be fixed and let's find intersection $A \cap (n, n+1)$.
Since $\alpha < 1$, then at most two intervals of the form
$\bigl(m/\alpha, (m+1)/\alpha\bigl)$, $m \in \bZ$, intersect
with $(n,n+1)$. There are two cases possible. First, for some $m
\in \bZ$ interval $\bigl(m/\alpha, (m+1)/\alpha\bigr)$ fully
covers interval $(n,n+1)$ (i.e. $m/\alpha < n < n+1 <
(m+1)/\alpha$). In this case, if $n$ and $m$ are of the same
parity then $(n,n+1) \in A$, otherwise $A \cap (n,n+1) =
\varnothing$. The second case, is when for some $m \in \bZ$ we
have $n < m/\alpha < n+1$. In this case it is clear that exactly
one of the intervals $(n, m/\alpha)$ and $(m/\alpha, n+1)$
belongs to $A$ depending on parity of $m-n$. Thus, the domain of
the 4th equation in~\eqref{eq:a=0.final.system} is the union
$\cup_{n\in\bZ}I_n$, where $I_n$ is possibly empty subinterval
of $(n,n+1)$.

Put $r(x) := \sin \alpha \pi x / \sin \pi x$. Straightforward
calculation shows that for $x \in I_n$, $n \in \bZ$, we have
\begin{align} \label{e}
    f'(x) &= \frac{r^{\frac{1}{\alpha+1}}(x)}{(\alpha+1) \sin \alpha \pi x}
    \left((\alpha+1)^2 \cos(\alpha+1) \pi x -
    (\cot \pi x + \alpha^2 \cot \alpha \pi x) \sin (\alpha+1) \pi x \right) \nonumber \\
    &= \frac{-r^{\frac{1}{\alpha+1}}(x)}{(\alpha+1) \sin \alpha \pi x}
    \left((\alpha+1)^2 \sin \pi x \sin \alpha \pi x +
    \left(\sqrt{r(x)} \cos \pi x + \frac{\alpha \cos \alpha \pi x}{\sqrt{r(x)}}\right)^2\right).
\end{align}
Since $\sin \pi x$ and $\sin \alpha \pi x$ are of the same sign
on each interval $I_n$ it is clear that $f'(x)$ has fixed sign
on $I_n$. Hence, $f(x)$ is strictly monotonic on $I_n$. In
particular, the equation $f(x) = -d$, $x \in I_n$, has at most
one solution. Denote it by $x_n$ if it exists and let $y_n$ be
the corresponding value of $y$ from the third equation
in~\eqref{eq:a=0.final.system}. Clearly, all $x_n$ are different
and hence all zeros of $\Delta_0(\cdot)$ are simple.

Now let~\eqref{eq:a=0.crit} be satisfied and assume that
boundary conditions are not strictly regular. Since zeros of
$\Delta_0(\cdot)$ are simple it means that there exists infinite
set $S \subset \bN$ such that $x_{n-1}, x_n$ exist for $n \in S$
and
\begin{equation} \label{eq:a=0.xn.yn.diff.to.0}
    x_n - x_{n-1} \to 0, \quad y_n - y_{n-1} \to 0
    \quad\text{as}\quad n \to \infty, \quad n \in S.
\end{equation}
Since zeros of $\Delta_0(\cdot)$ lie in the strip it follows
that $|y_n| \le H$, $n \in S$, with some $H > 0$. Hence, the
third relation in~\eqref{eq:a=0.final.system} implies that
\begin{equation} \label{eq:C1<sin/sin<C2}
    0 < C_1 < t_n := \frac{\sin \alpha \pi x_n}{\sin \pi x_n} < C_2, \quad n \in S,
\end{equation}
with some $C_1, C_2 > 0$. Since $t_n = e^{-(\alpha+1)y_n}$, it
is clear that $y_n - y_{n-1} \to 0$ is equivalent to $t_n -
t_{n-1} \to 0$.

Taking into account the form of the set $A$ described after the
formula~\eqref{eq:a=0.x.in...} we see that there exists unique
$m = m_n \in \bZ$ such that either $x_{n-1} < m/\alpha < n <
x_n$ or $x_{n-1} < n < m/\alpha < x_n$. Moreover, $m_n$ has the
same parity as $n$. Let $S_1 (S_2)$ be the set of those $n$ in
$S$ for which the first (the second) inequality is satisfied.
Since $S$ is infinite and $S = S_1 \cup S_2$, either $S_1$ or
$S_2$ is infinite. First consider the case when $S_1$ is
infinite. For $n \in S_1$ we put
\begin{equation} \label{eq:a=0.delta.eps.n}
    \delta_{0n} := m/\alpha - x_{n-1}, \quad \eps_n = n - m/\alpha,
    \quad \delta_{1n} := x_n - n.
\end{equation}
Since $x_{n-1} < m/\alpha < n < x_n$, then $\eps_n, \delta_{0n},
\delta_{1n} > 0$. Further, since $x_n - x_{n-1} \to 0$ as $n \to
\infty$, $n \in S_1$, then $\eps_n \to 0$, $\delta_{0n} \to 0$,
$\delta_{1n} \to 0$ as $n \to \infty$, $n \in S_1$. Hence, for
large $n \in S_1$ we have taking into account that $m$ and $n$
are of the same parity
\begin{align}
    t_n &= \frac{\sin \alpha \pi x_n}{\sin \pi x_n}
    = \frac{\sin \pi (m + \alpha(\delta_{1n} + \eps_n))}{\sin \pi (n + \delta_{1n})}
    = \frac{\sin \pi \alpha(\delta_{1n} + \eps_n)}{\sin \pi \delta_{1n}}
    > \alpha, \label{eq:a=0.tn} \\
    t_{n-1} &= \frac{\sin \alpha \pi x_{n-1}}{\sin \pi x_{n-1}}
    = \frac{\sin \pi (m - \alpha \delta_{0n})}{\sin \pi (n - \delta_{0n} - \eps_n)}
    = \frac{\sin \pi \alpha \delta_{0n}}{\sin \pi (\delta_{0n} + \eps_n)}
    < \alpha, \label{eq:a=0.tn-1}
\end{align}
Here we used the inequality $\sin \alpha u > \alpha \sin v$ for
$0 < \alpha < 1$ and $0 < v < u < \pi$.

Since $t_n - t_{n-1} \to 0$, it follows from~\eqref{eq:a=0.tn}
and~\eqref{eq:a=0.tn-1} that $t_n \to \alpha$, $t_{n-1} \to
\alpha$ as $n \to \infty$, $n \in S_1$. This implies that
$\eps_n / \delta_{1n} \to 0$ as $n \to \infty$, $n \in S_1$.
Indeed, since $\eps_n \to 0$ and $\delta_{1n} \to 0$ as $n \to
\infty$, $n \in S_1$, then
\begin{equation}
    \alpha = \lim_{n \to \infty \atop{n \in S_1}} t_n
    = \lim_{n \to \infty \atop{n \in S_1}}
    \frac{\sin \pi \alpha(\delta_{1n} + \eps_n)}{\sin \pi \delta_{1n}}
    = \lim_{n \to \infty \atop{n \in S_1}}
    \frac{\alpha(\delta_{1n} + \eps_n)}{\delta_{1n}}
    = \alpha + \alpha \lim_{n \to \infty \atop{n \in S_1}} \frac{\eps_n}{\delta_{1n}}.
\end{equation}
Further, note that
\begin{equation}
    \frac{\sin (\alpha+1) \pi x_n}{\sin \alpha \pi x_n}
    = \frac{\sin \pi (n + \delta_{1n} + m + \alpha(\delta_{1n} + \eps_n))}{\sin \pi (m + \alpha(\delta_{1n} + \eps_n))}
    = \frac{\sin \pi ((\alpha+1)\delta_{1n} + \alpha \eps_n)}{\sin \pi (\alpha\delta_{1n} + \alpha\eps_n))}.
\end{equation}
Finally, taking into account the last relation and the fact that
$\eps_n/\delta_{1n} \to 0$ and $t_n \to \alpha$ as $n \to
\infty$, $n \in S_1$, we have
\begin{align}
    -d &= \lim_{n \to \infty \atop{n \in S_1}} f(x_n)
     = \lim_{n \to \infty \atop{n \in S_1}} t_n^{\frac{1}{\alpha+1}}
       \frac{\sin (\alpha+1) \pi x_n}{\sin \alpha \pi x_n}
     = \alpha^{\frac{1}{\alpha+1}} \lim_{n \to \infty \atop{n \in S_1}}
       \frac{\sin \pi ((\alpha+1)\delta_{1n} + \alpha \eps_n)}{\sin \pi (\alpha\delta_{1n} + \alpha\eps_n)} \nonumber \\
     &= \alpha^{\frac{1}{\alpha+1}} \lim_{n \to \infty \atop{n \in S_1}}
       \frac{\alpha+1 + \alpha \eps_n / \delta_{1n}}{\alpha + \alpha \eps_n/\delta_{1n}}
     = \alpha^{\frac{1}{\alpha+1}} \frac{\alpha+1}{\alpha}
     = (\alpha+1) \alpha^{-\frac{\alpha}{\alpha+1}}.
\end{align}
Since $bc=-1$ this contradicts to~\eqref{eq:a=0.crit}.
Therefore, zeros of $\Delta_0(\cdot)$ are asymptotically
separated. The case of infinite $S_2$ is considered similarly.

Now let's prove that opposite statement. As above, for
simplicity we assume that $bc=-1$ and $\alpha < 1$. We need to
prove that if $d = -(\alpha+1)
\alpha^{-\frac{\alpha}{\alpha+1}}$ then zeros of
$\Delta_0(\cdot)$ are not asymptotically separated. For $n \in
\bZ$ we set $m = m_n = \lfloor \alpha n \rfloor$. Let $S \subset
\bZ$ be some infinite set such that $m_n - n \to 0$ as $n \to
\infty$ and $m_n$ is of the same parity as $n$. Since $\alpha
\not\in \bQ$ it is clear that such set exists. Let us prove that
for large enough $n \in S$ the equation $f(x) = -d$ has zeros
$x_{n-1}, x_n$ such that $x_{n-1} < m/\alpha < n < x_n$ and $x_n
-x_{n-1} \to 0$ as $n \to \infty$. Put
\begin{equation}
    g(u,v) := \left(\frac{\sin \alpha u}{\sin v}\right)^{\frac{1}{\alpha+1}}
            \frac{\sin (\alpha u + v)}{\sin \alpha u}.
\end{equation}
Clearly
\begin{equation}
    f(x_{n-1}) = g\left(\pi \delta_{0n}, \pi (\delta_{0n} + \eps_n)\right), \quad\text{and}\quad
    f(x_n) = g\left(\pi (\delta_{1n} + \eps_n), \pi \delta_{1n}\right),
\end{equation}
where $\eps_n, \delta_{0n}, \delta_{1n}$ are defined
in~\eqref{eq:a=0.delta.eps.n}. Due to the special form of $d$,
equation $g(u,v) = -d$ is equivalent to
\begin{equation}
    f(u,v) := \left(\frac{\sin(\alpha u + v)}{\alpha+1}\right)^{\alpha+1} -
    \sin v \cdot \left(\frac{\sin \alpha u}{\alpha}\right)^{\alpha} = 0.
\end{equation}
It is easy to prove that
\begin{equation}
    f(u, 0) > 0, f(u, u) < 0, f(u, 2u) > 0, \quad 0 < u < 1/2.
\end{equation}
Hence for each $u \in (0, 1/2)$ there exists $v_+ \in (u, 2u)$
and $v_- \in (0,u)$ such that $f(u, v_{\pm})=0$. Clearly,
$v_{\pm} = v_{\pm}(u)$ is continuous at $u$. Hence for
sufficiently small $\eps > 0$ there exists $u^{\pm}_{\eps},
v^{\pm}_{\eps}$ such that $u^{\pm}_{\eps} - v^{\pm}_{\eps} = \pm
\eps$, $g(u^{\pm}_{\eps}, v^{\pm}_{\eps}) = -d$ and
$u^{\pm}_{\eps}, v^{\pm}_{\eps} \to 0$ as $\eps \to 0$. Applying
this fact for $\eps_n$ shows existence of needed $x_{n-1}$ and
$x_n$. It is easy to prove from $f(x_n) = f(x_{n-1}) = -d$ that
$t_n \to \alpha$ and $t_{n-1} \to \alpha$ as $n \to \infty$
where $t_n$, $t_{n-1}$ are defined
in~\eqref{eq:a=0.tn}--\eqref{eq:a=0.tn-1}. Which implies $y_n -
y_{n-1} \to \infty$ and shows that corresponding zeros of
$\Delta_0(\cdot)$ are not asymptotically separated.

\textbf{(ii)} Since $b_1/b_2 \in \bQ$ the set of zeros of
$\Delta_0(\cdot)$ is the union of finite number of arithmetic
progressions. Hence zeros are asymptotically separated if and
only if $\Delta_0(\cdot)$ does not have multiple zeros, which is
equivalent to the fact that $\Delta_0(\cdot)$ and
$\Delta_0'(\cdot)$ have no common zeros. When $a=0$ we have
$\Delta_0'(\l) = i b_2 e^{i b_2 \l} - i b_1 \cdot bc \cdot e^{i
b_1 \l}$. Hence zeros of $\Delta_0'(\cdot)$ can be found
explicitly. Substituting these values into $\Delta_0(\l)$ and
performing straightforward calculations we see that
$\Delta_0(\cdot)$ and $\Delta_0'(\cdot)$ have no common zeros if
and only if condition~\eqref{eq:a=0.crit.rat} is satisfied.
\end{proof}
%
%
\section{The Riesz basis property of root vectors system}
\label{sec:RieszBasis}
%
%
\subsection{Some auxiliary results}
%
%
Recall the following definition.
%
%
\begin{definition}
\textbf{(i)} Let $\frak H$ be a separable Hilbert space. The
vectors system $\{f_k\}_{k \in \bZ} \subset \frak H$ is called
\textbf{Besselian} in $\frak H$ if $\{(f, f_k)\}_{k \in \bZ} \in
l^2(\bZ)$, $f \in \frak H$.

\textbf{(ii)} The vectors system $\{f_k\}$ is called a
\textbf{Riesz basis} in $\frak H$ if it constitutes  a basis
equivalent to an orthonormal one, i.e. there exists a linear
homeomorphism $T$ in $\frak H$ for which $\{T f_k\}$ is an
orthonormal basis.

\textbf{(iii)} A sequence of subspaces
$\{\fH_k\}_{k=1}^{\infty}$ is called a \textbf{Riesz basis of
subspaces} in $\fH$ if there exists a complete sequence of
mutually orthogonal subspaces $\{\fH'_k\}_{k=1}^{\infty}$ and a
bounded operator $T$ in $\fH$ with bounded inverse such that
$\fH_k = T \fH'_k$, $k \in \bN$.

\textbf{(iv)} It is said that a  sequence
$\{f_k\}_{k=1}^{\infty}$ of vectors in $\fH$  forms  a
\textbf{Riesz basis with parentheses} if each its finite
subsequence is linearly independent, and there exists an
increasing sequence $\{n_k\}_{k=0}^{\infty} \subset \bN$ such
that $n_0=1$ and the sequence $\fH_k :=
\Span\{f_j\}_{j=n_{k-1}}^{n_k-1}$  constitutes a Riesz basis of
subspaces in $\fH$. Subspaces $\fH_k$ are called blocks.
\end{definition}
%
%
Our investigation of the Riesz basis property of the root
vectors system of the operator $L(Q)$ is heavily relied on the
following well-known Bari criterion.
%
%
\begin{theorem} \label{th:Bari.crit}\cite[Theorem\ VI.2.1]{GohKre65}
Let $\frak H$ be a separable Hilbert space. The vectors system
$\{f_k\}_{k \in \bZ} \subset \frak H$ forms a Riesz basis in
$\frak H$ if and only if it is complete, minimal and Besselian
in $\frak H$, and the corresponding biorthogonal system
$\{g_k\}_{k \in \bZ}$ is also complete and Besselian.
\end{theorem}
%
%
It is well-known that the root vectors system of the operator
$L(Q)$ after proper normalization is biorthogonal to the root
vectors system of the adjoint operator $L(Q)^*$. In this
connection we give the explicit form of the operator $L(Q)^*$ in
the case of boundary conditions~\eqref{eq:BC.new}.
%
%
\begin{lemma} \label{lem:adjoint}
Let $L(Q)$ be an operator corresponding to the
problem~\eqref{eq:system},~\eqref{eq:BC.new}. Then the adjoint
operator $L^* := L(Q,{\widehat U}_1, {\widehat U}_2)^*$ is  $L^*
= L(Q^*, U_{*1}, U_{*2}))$, i.e. it is given by the differential
expression~\eqref{eq:system} with
$Q^*(x) = \begin{pmatrix} 0 & \overline{Q_{21}(x)} \\
\overline{Q_{12}(x)} & 0 \end{pmatrix}$ instead of $Q$ and the
boundary conditions
\begin{equation} \label{eq:BC*}
\begin{cases}
   U_{*1}(y) = k \overline{b} y_1(0) + y_2(0) + \overline{d} y_2(1) &= 0, \\
   U_{*2}(y) = \overline{a} y_1(0) + y_1(1) + k^{-1} \overline{c} y_2(1) &= 0,
\end{cases}
\end{equation}
where $k := - b_2 b_1^{-1}$. Moreover, boundary
conditions~\eqref{eq:BC*} are regular (strictly regular)
simultaneously with  boundary conditions~\eqref{eq:BC.new}.
\end{lemma}
%
%
%
The following lemma plays the key role in the proof of
Theorem~\ref{th:basis.strict}.
%
%
\begin{lemma} \label{lem:Phi.Bessel}
Let $Q \in L^1[0,1] \otimes \bC^{2 \times 2}$  and  let
$\Phi(\cdot,\l)$ and $\Psi(\cdot, \l)$ be the fundamental matrix
solutions of the equations~\eqref{eq:system} and~\eqref{3.15}
satisfying $\Phi(0,\l) = \Psi(0, \l) = I_2$, given by formulas
\eqref{3.2} and~\eqref{3.16}, respectively. Let also
$\Phi_j(\cdot,\l)$ and $\Psi_j(\cdot, \l)$, $j \in \{1,2\}$, be
the columns of these matrices $($cf.~\eqref{3.2} and
\eqref{3.16}$)$. Then for any incondensible sequence
$\{\mu_n\}_{n \in \bZ}$ the systems $\{\Phi_j(\cdot,\mu_n)\}_{n
\in \bZ}$ and $\{\Psi_j(\cdot, \overline{\mu_n})\}_{n \in \bZ}$
are Besselian in $L^2[0,1] \otimes \bC^2$, $j \in \{1,2\}$.
\end{lemma}
%
%
\begin{proof}
Consider the case of the system $\{\Phi_1(\cdot,\mu_n)\}_{n \in
\bZ}$. Let $f := \col(f_1, f_2) \in L^2[0,1] \otimes \bC^2$.
Taking into account formulas~\eqref{eq:phi11},~\eqref{eq:phi21}
we get
\begin{align} \label{eq:f.Phi1}
    \left( f , \Phi_1(\cdot, \mu_n) \right)_{L^2[0,1] \otimes \bC^{2}}
    &= \int_0^1 \left(f_1(x) \ol{\varphi_{11}(x, \mu_n)}
     + f_2(x) \ol{\varphi_{21}(x, \mu_n)}\right) dx \nonumber \\
    &= \int_0^1 f_1(x) \ol{e^{i b_1 \mu_n x}} dx + \sum_{j,k=1}^2
       \int_0^1 f_j(x) \left(\int_0^x \ol{N_{jk}(x,t) e^{i b_k \mu_n t}}\,dt \right)\,dx,
\end{align}
where $N_{j1}(\cdot,\cdot) := R_{j1}^{+}(\cdot,\cdot), \
N_{j2}(\cdot,\cdot) := R_{j2}^{-}(\cdot,\cdot), \ j \in
\{1,2\}$. Further, note
\begin{equation} \label{eq:int.fj.Njk.e}
    \int_0^1 f_j(x) \left(\int_0^x \ol{N_{jk}(x,t) e^{i b_k \mu_n t}} dt\right)\,dx =
    \int_0^1 g_{jk}(t) \ol{e^{i b_k \mu_n t}}\,dt
\end{equation}
where
\begin{equation} \label{eq:gjk.in.L2}
    g_{jk}(t) := \int_t^1 \ol{N_{jk}(x,t)} f_j(x) dx, \quad j,k \in \{1,2\}.
\end{equation}

By Proposition~\ref{prop:phi.jk=e+int}, $N_{jk}(\cdot, \cdot)
\in X_1(\Omega) \cap X_{\infty}(\Omega)$. Therefore by
Lemma~\ref{lem_Volterra_operGeneral}  the Volterra type
operators
\begin{equation*}
    N_{jk} : f \to \int_0^x N_{jk}(x,t) f(t)\,dt  \quad\text{and} \quad
      N_{jk}^* : f \to \int_t^1 \ol{N_{jk}(x,t)} f_j(x)\,dx
\end{equation*}
are bounded in $L^2[0,1]$, hence $g_{jk} \in L^2[0,1]$,\ $j,k
\in \{1,2\}$. Taking this inclusion into account and  inserting
expressions~\eqref{eq:int.fj.Njk.e} into~\eqref{eq:f.Phi1} one
rewrites this equality as
\begin{equation} \label{eq:f.Phi1.short}
    \left( f , \Phi_1(\cdot, \mu_n) \right)_{L^2[0,1] \otimes \bC^2} =
    \left(f, e^{i b_1 \mu_n x}\right)_{L^2[0,1]} +
    \sum_{j,k=1}^2 \left(g_{jk}, e^{i b_k \mu_n t} \right)_{L^2[0,1]}.
\end{equation}
Since the sequence $\{\mu_n\}_{n \in \bZ}$ is incondensible,
then by~\cite[Lemma 2]{Katsn71} the sequence of  exponents
$\{e^{i b_1 \mu_n x}\}_{n \in \bZ}$ is Besselian in $L^2[0,1]$.
The latter implies
$$
\{(g_{jk}, e^{i b_k \mu_n t})_{L^2[0,1]}\}_{n \in \bZ} \in l^2(\bZ) \quad
\text{and}\quad \{(f, e^{i b_k \mu_n t})_{L^2[0,1]}\}_{n \in \bZ} \in l^2(\bZ).
$$
Combining these inclusions with representation
~\eqref{eq:f.Phi1.short} shows that  the system $\{\Phi_1(\cdot,
\mu_n)\}_{n \in \bZ}$ is Besselian.
 The systems of functions  $\{\Phi_2(\cdot, \mu_n)\}_{n \in \bZ}$  and $\{\Psi_j(\cdot,
\ol{\mu_n})\}_{n \in \bZ}$ are treated  similarly.
\end{proof}
%
%
\subsection{Proof of the main result}
%
%
Now we are ready to prove the main result of the paper.
\begin{proof}[Proof of Theorem~\ref{th:basis.strict}]
According to Proposition~\ref{prop:Delta.regular.basic} and
Definition~\ref{def:strictly.regular}, the operator $L(Q)$ has
countably many eigenvalues $\{\l_n\}_{n \in \bZ}$. Moreover,
they are of finite multiplicity, asymptotically simple and
separated, and are located in the strip $\Pi_h = \{\l \in \bC :
|\Im\,\l| \leqslant h\}$.

Since boundary conditions are regular, one  can transform  them
to the form~\eqref{eq:BC.new}.

(i) In this step  assume that $|b| + |c| \ne 0$. Without loss of
generality it suffices to consider the case  $b \ne 0$. Let
$\frak F  = \{f_n\}_{n \in \bZ}$ and $\frak G = \{g_n\}_{n \in
\bZ}$ be the system of root vectors of the operators $L(Q) =
L(Q,{\widehat U}_1, {\widehat U}_2)$ and   $L(Q)^* = L(Q^*,
U_{*1}, U_{*2}))$, respectively.   By~\cite[Theorem
1.2]{MalOri12}, each of the systems $\frak F$ and $\frak G$ is
complete and minimal in $L^2[0,1] \otimes \bC^2$. Therefore
these systems can be chosen to be  biorthogonal to each other.

First we indicate the explicit form of the functions $f_n$ and
$g_n$ for $n$ large enough.

To this end one   easily gets  from~\eqref{3.2} and
\eqref{eq:BC.new} that
    \begin{align}
{\widehat U}_1(\Phi_1(\cdot, \l)) &  =\varphi_{11}(0,\l) + b\varphi_{21}(0,\lambda) +
a\varphi_{11}(1,\lambda) =
1+a\varphi_{11}(\lambda), \label{5.6}\\
{\widehat U}_1(\Phi_2(\cdot, \l)) & = \varphi_{12}(0,\lambda) + b\varphi_{22}(0,\lambda)
+ a\varphi_{12}(1,\lambda) = b + a\varphi_{12}(\lambda). \label{5.7}
    \end{align}
Since $\varphi_{12}(\l_n) = o(1)$ as $n\to\infty$
(see~\eqref{eq:phi.jkx}) and $b\not =0$, one gets that $b +
a\varphi_{12}(\lambda_n)\not =0$ for $|n|\ge n_1$ with some
$n_1\in \bN$.
Therefore one derives from~\eqref{5.6},~\eqref{5.7}, and
\eqref{3.2}  that  the vector-function
    \begin{align}\label{eq:fnx}
f_n(\cdot) := & \widehat U_1\left(\Phi_2(\cdot, \l_n)\right)\Phi_1(\cdot, \l_n) -
\widehat U_1\left(\Phi_1(\cdot, \l_n)\right) \Phi_2(\cdot, \l_n)  \nonumber \\
= & (b + a \varphi_{12}(\l_n)) \Phi_1(\cdot, \l_n)
    - (1 + a \varphi_{11}(\l_n)) \Phi_2(\cdot, \l_n), \quad |n|\ge n_1,
    \end{align}
is a non-trivial  eigenfunction  of the operator $L(Q)$
corresponding to the eigenvalue $\l_n.$
Since  boundary conditions~\eqref{eq:BC.new}  are strictly
regular, it follows from Proposition
\ref{prop:Delta.regular.basic}  that the sequence
$\Lambda=\{\l_n\}_{n\in \bZ}$  of the eigenvalues of $L(Q)$,
i.e. the zeros of $\Delta(\cdot)$, is asymptotically separated.
In particular, there exists $n_0\in \bN$ such that the
eigenvalues of $L(Q)$ are geometrically and algebraically
simple. Therefore $f_n$ is the unique up to a multiplicative
constant eigenfunction  of the operator $L(Q)$ corresponding to
$\l_n$  for $|n|\ge \max\{n_0,n_1\}$.

Similarly, one easily gets from~\eqref{eq:BC*}  that
     \begin{equation}
U_{*1}(\Psi_1(\cdot, \l)) = k\overline b + \overline d\psi_{21}(\l),  \qquad
U_{*1}(\Psi_2(\cdot, \l)) = 1  + \overline d\psi_{22}(\lambda).
     \end{equation}
Moreover, in accordance with~\eqref{eq:phi.jkx} there exists
$n_2\in \bN$ such that $k\overline b + \overline
d\psi_{21}(\l_n)\not =0$ for $|n| \ge n_2$. Therefore
Lemma~\ref{lem:adjoint} ensures that the vector-function
\begin{align} \label{eq:gnx}
    g_n(\cdot) &:= U_{*1}\bigl(\Psi_2(\cdot,\overline{\lambda}_n)\bigr)\Psi_1(\cdot,\overline{\lambda}_n) -
                    U_{*1}\bigl(\Psi_1(\cdot,\overline{\lambda}_n)\bigr)\Psi_2(\cdot,\overline{\lambda}_n)  \nonumber  \\
         &= (1 + \overline{d} \psi_{22}(\overline{\l_n})) \Psi_1(\cdot, \overline{\l_n})
          - (k \overline{b} + \overline{d} \psi_{21}(\overline{\l_n})) \Psi_2(\cdot, \overline{\l_n}),\quad |n|\ge n_2,
\end{align}
is a non-trivial  eigenfunction  of the operator $L(Q)^*$
corresponding to the eigenvalue $\overline{\l_n}.$   Since the
eigenvalues of $L(Q)^*$ constitutes  a sequence
$\{\overline{\l_n}\}_{n\in \bZ}$, they are geometrically and
algebraically simple simultaneously with $\{{\l_n}\}_{n\in\bZ}$,
i.e.  for $|n|\ge n_0$. Therefore $g_n$ is the unique up to a
multiplicative constant eigenfunction  of the operator $L(Q)^*$
corresponding to $\overline{\l_n}$  for $|n|\ge
\max\{n_0,n_2\}$.

Further,  it follows from~\eqref{eq:phi.jkx}  and~\eqref{3.18}
that
  \begin{align}
\|f_n\|^2 = |b + a o(1)|^2\cdot \|\Phi_1(\cdot, \l_n)\|^2 + |1 + ae^{ib_1\l_n} +
o(1)|^2 \cdot \|\Phi_2(\cdot, \l_n)\|^2  \nonumber \\
- 2\Re (b + a o(1)) (1 + ae^{ib_1\l_n} + o(1)) (\Phi_1(\cdot, \l_n), \Phi_2(\cdot,
\l_n))  \nonumber \\
= |b|^2\cdot \|\Phi_1(\cdot, \l_n)\|^2 + |1 + ae^{ib_1\l_n}|^2 \|\Phi_2(\cdot, \l_n)\|^2
+ o(1) \quad \text{as} \quad  n\to \infty.
   \end{align}
Similar relation is valid for $g_n$. Combining these relations
with  estimates \eqref{eq:Phij.Phij}  and noting that $\l_n\in
\Pi_h$, $n\in \bN,$ yields
  \begin{equation} \label{eq:phin,phin}
\|f_n\|   \asymp 1,  \qquad \|g_n\|  \asymp 1,
    \qquad\text{for large}\ \ n.
  \end{equation}

Here the symbol  $a_n \asymp 1$ for large $n$ means that there
exists $n_0 \in \bN$ and $C_1, C_2
> 0$ such that $C_1 < |a_n| < C_2$, $|n| > n_0$.
In particular, vector-functions $f_n$, $g_n$ are non-zero for
large $n$.

Similarly starting with~\eqref{eq:fnx},~\eqref{eq:gnx}, and
using relations~\eqref{eq:phi.jkx},~\eqref{3.17}, and noting
that $k \in \mathbb R,$ one arrives at the following important
asymptotic relation
   \begin{align} \label{eq:phin,psin}
\left(f_n, g_n\right)  =\bigl(b + a\cdot o(1)\bigr)\bigl(1 + d e^{-i b_2\lambda_n}
+ o(1)\bigr)(\Phi_1,\Psi_1) \nonumber \\
 +  \left(1+a e^{i b_1\lambda_n} + o(1)\right) \bigl(\overline k b + d\cdot
  o(1)\bigr)(\Phi_2,\Psi_2)  \nonumber  \\
 =  b (d e^{-i b_2 \l_n} + k a e^{i b_1 \l_n} + k + 1) + o(1)\quad \text{as}
  \quad n \to \infty.
\end{align}

Further, using the formula~\eqref{eq:Delta0} for the
perturbation determinant $\Delta_0(\cdot)$ one easily  gets
\begin{align} \label{eq:dl(Delta0.e)}
    \frac{d}{d\l}\left(\Delta_0(\l) e^{-i b_1 \l}\right)
    &= \frac{d}{d\l}\left(d e^{-i b_1 \l} + a e^{i b_2 \l}
    + ad-bc + e^{i (b_2-b_1) \l}\right) \nonumber \\
    &= -i b_1 e^{i (b_2-b_1) \l} (d e^{-i b_2 \l}
    + k a e^{i b_1 \l} + k + 1).
\end{align}
Since  boundary conditions~\eqref{eq:BC} and~\eqref{eq:BC.new}
are equivalent, the latter are also  strictly regular. Therefore
in accordance with Definition~\ref{def:strictly.regular} the
sequence of zeros $\Lambda_0 := \{\l_n^0\}_{n \in \bZ}$ of the
determinant $\Delta_0(\cdot)$ is  asymptotically separated. In
other words, there exist $n_0\in \mathbb N$ and $\delta>0$ such
that the separation condition~\eqref{separ_cond} is satisfied
for $|j|, |k|\ge n_0$.

Combining  estimate~\eqref{eq:Delta0>=} with  the Minimum and
Maximum Principles for analytic functions (cf.~\cite[Lemma
22.2]{Lev96}), we arrive at the following two-sided estimate
\begin{equation} \label{eq:C1<Delta<C2}
    C_1 < |\Delta_0'(\l)| < C_2, \qquad \l \in \bD_{\delta}(\l_n^0), \quad |n| > n_0.
\end{equation}
with certain  $C_1 = C_1(\delta)>0, C_2 = C_2(\delta) >0$.  It
follows from asymptotic formula~\eqref{eq:l.n=l.n0+o(1)} that
$\l_n = \l_n^0 + o(1) \in \bD_{\delta}(\l_n^0)$ for $n$ large
enough. Combining this inclusion with
estimates~\eqref{eq:C1<Delta<C2} yields  $\Delta_0'(\l_n) \asymp
1$ for large $n$. Since $|\Im\,\l_n| \leqslant h$,
relations~\eqref{eq:phin,psin},~\eqref{eq:dl(Delta0.e)},
and~\eqref{eq:C1<Delta<C2} imply
\begin{equation} \label{eq:(phin,psin)=1}
    \left(f_n, g_n\right) \asymp 1 \quad\text{for large}\ \ n.
\end{equation}
Thus we can normalize the systems $\{f_n\}$  and $\{g_n\}$ by
putting
\begin{equation} \label{eq:fn,gn.def}
    {\widehat f}_n := \frac{f_n}{\|f_n\|}, \qquad
    {\widehat g}_n := \frac{\|f_n\| g_n}{(f_n, g_n)}, \qquad |n|\ge m :=\max\{n_0,n_1, n_2\}.
\end{equation}
Clearly, $\|{\widehat f}_n\|=1$ and $({\widehat f}_n, {\widehat
g}_n)=1$ for $|n| > m$, i.e. the  sequences
  \begin{equation}\label{5.18}
\mathcal F := \{{\widehat f}_n\}_{|n|>m} \quad \text{and} \quad \mathcal G :=
\{{\widehat g}_n\}_{|n|> m}
     \end{equation}
are biorthogonal. It follows from~\eqref{eq:fnx},~\eqref{eq:gnx}
and~\eqref{eq:fn,gn.def} that
\begin{align}
  \label{eq:fn}
       {\widehat f}_n(\cdot) &= \frac{b + a \varphi_{12}(\l_n)}{\|f_n\|} \Phi_1(\cdot, \l_n)
                - \frac{1 + a \varphi_{11}(\l_n)}{\|f_n\|} \Phi_2(\cdot, \l_n), \\
  \label{eq:gn}
        {\widehat g}_n(\cdot) &= \frac{\|f_n\|(1 + \overline{d} \psi_{22}(\overline{\l_n}))}
                       {(f_n, g_n)} \Psi_1(\cdot, \overline{\l_n})
                - \frac{\|f_n\|(k \overline{b} + \overline{d} \psi_{21}(\overline{\l_n}))}
                       {(f_n, g_n)} \Psi_2(\cdot, \overline{\l_n}).
\end{align}
By Proposition~\ref{prop:sine.type},  the sequence of
eigenvalues $\Lambda =\{\l_n\}_{n\in \Bbb Z}$ is condensible.
Therefore, by Lemma~\ref{lem:Phi.Bessel}, the sequences of
vector-functions $\{\Phi_j(\cdot,\l_n)\}_{|n|>m}$ and
$\{\Psi_j(\cdot, \overline{\l_n})\}_{|n|>m}$, $j \in \{1,2\}$,
are Besselian in $L^2[0,1] \otimes \bC^2$. It follows
from~\eqref{eq:phi.jkx},~\eqref{eq:phin,phin},
and~\eqref{eq:(phin,psin)=1} with account of the inclusion
$\Lambda \subset \Pi_h$, that the coefficients at $\Phi_j(\cdot,
\l_n)$ and $\Psi_j(\cdot, \overline{\l_n})$, $j \in \{1,2\}$,
in~\eqref{eq:fn}  and~\eqref{eq:gn} are bounded in  $n$. Hence,
the sequences $\mathcal F$ and  $\mathcal G$ given by
\eqref{5.18}  are also Besselian in $L^2[0,1] \otimes \bC^2$.

Clearly,  the systems  $\frak F$ and $\frak G$  contain
$\mathcal F$ and  $\mathcal G$, respectively. Since $\frak F$
differs from $\mathcal F$ by a finite number (in fact, $2m+1$)
entries  and $\mathcal F$ is Besselian in $L^2[0,1] \otimes
\bC^2$, the system $\frak F$ is Besselian too. Similarly $\frak
G$ is Besselian in $L^2[0,1] \otimes \bC^2$. The Riesz basis
property of the system $\frak F$ in $L^2[0,1] \otimes \bC^2$  is
now implied by the Bari criterion (Theorem~\ref{th:Bari.crit}).

(ii) Now assume that $b=c=0$. In this case $a d \ne 0$ and
$\Delta_0(\l) = (d + e^{i b_2 \l})(1 + a e^{i b_1 \l})$. Let
$\Lambda_1^0 = \{\l_{1,n}^{0}\}_{n \in \bZ}$ and $\Lambda_2^0 =
\{\l_{2,n}^{0}\}_{n \in \bZ}$ be the sequences of zeros of the
first and second factor, respectively. Clearly,  these sequences
constitute  arithmetic progressions  lying  on the lines
parallel to the real axis.
 Since the boundary conditions are strictly regular, by the assumption, the
arithmetic progressions $\Lambda_1^0$ and $\Lambda_2^0$ are
separated, i.e., $|\l_{1,n}^0 - \l_{2,m}^0|
> 2 \delta$, $m,n \in \bZ$ for some $\delta > 0$.  This implies the following asymptotic
relations
\begin{equation} \label{eq:1+ae=1,d+e=1}
    1 + a e^{i b_1 \l_{1,n}^0} \asymp 1, \quad
    d + e^{i b_2 \l_{2,n}^0} \asymp 1 \quad\text{for large}\  |n|.
\end{equation}
In accordance with Proposition~\ref{prop:Delta.regular.basic}
the sequence $\Lambda = \{\l_n\}_{n \in \bZ}$ of zeros of
$\Delta(\cdot)$ admits a decomposition $\Lambda = \Lambda_1 \cup
\Lambda_2$   where $\Lambda_j = \{\l_{j,n}\}_{n \in \bZ}$, $j
\in \{1,2\}$, and  $\l_{j,n} = \l_{j,n}^0 + o(1)$ as $n \to
\infty$. Combining this representation with
Proposition~\ref{prop:phi.jk.asymp} and~\eqref{eq:1+ae=1,d+e=1}
yields
\begin{equation} \label{eq.1+a.phi11=1,d+phi22=1}
    1 + a \varphi_{11}(\l_{1,n}) \asymp 1, \quad
    d + \varphi_{22}(\l_{2,n}) \asymp 1 \quad\text{for large}\ |n|.
\end{equation}
Similarly to~\eqref{eq.1+a.phi11=1,d+phi22=1} one derives
   \begin{equation} \label{eq.1,a+psi11=1}
\ol{a} + \psi_{11}(\ol{\l_{1,n}}) \asymp 1, \quad  (1 +
\overline{d}\psi_{22}(\overline{\l_{2,n}}))
    \asymp 1\quad  \text{for large}\quad |n|.
  \end{equation}
Combining this relation with~\eqref{eq.1+a.phi11=1,d+phi22=1}
yields the existence of $n_1'\in \bN$ such that
     \begin{equation} \label{Non-trivial_coeficients}
(1 + a \varphi_{11}(\l_{1,n}))(\ol{a} + \psi_{11}(\ol{\l_{1,n}})) \not = 0 \quad
\text{for}\quad |n| \ge n_1'.
    \end{equation}
Therefore  it follows from~\eqref{eq:fnx}  (with $b=0$ and
$a\not = 0$) that for $|n| \ge n_1'$ the vector-function
\begin{align}
    f_{1,n}(\cdot) & := a \varphi_{12}(\l_{1,n}) \Phi_1(\cdot, \l_{1,n})
    - (1 + a \varphi_{11}(\l_{1,n})) \Phi_2(\cdot, \l_{1,n}), \label{eq:phi2nx}
\end{align}
is the non-trivial eigenfunction  of the operator $L(Q)$
corresponding to the eigenvalue $\l_{1,n}$. Moreover, it follows
from the second equation in~\eqref{eq:BC*} and
\eqref{Non-trivial_coeficients} that  the vector-function
  \begin{align}
g_{1,n}(\cdot) & :=
U_{*2}\bigl(\Psi_2(\cdot,\overline{\l_{1,n}})\bigr)\Psi_1(\cdot,\overline{\l_{1,n}}) -
U_{*2}\bigl(\Psi_1(\cdot,\overline{\l_{1,n}})\bigr)\Psi_2(\cdot,\overline{\l_{1,n}})  \nonumber  \\
     & := \psi_{12}(\overline{\l_{1,n}}) \Psi_1(\cdot, \overline{\l_{1,n}})
    - (\overline{a} + \psi_{11}(\overline{\l_{1,n}})) \Psi_2(\cdot,
    \overline{\l_{1,n}}),\quad |n| \ge n_1',
    \label{eq:psi2nx}
\end{align}
is a  non-trivial eigenfunction  of the operator $L(Q)^*$
corresponding to the eigenvalue $\overline{\l_{1,n}}$.  Since
boundary conditions~\eqref{eq:BC.new} are strictly regular,
Proposition~\ref{prop:Delta.regular.basic} ensures existence of
$n_0'\in \bN$ such that the eigenvalues $\{\l_{1,n}\}$ of $L(Q)$
are geometrically and algebraically simple for $|n|\ge n_0'$.
Therefore $f_{1,n}(\cdot)$ ($g_{1,n}(\cdot)$) is the unique up
to a multiplicative constant eigenfunction of the operator
$L(Q)$ ($L(Q)^*$ ) corresponding to  the eigenvalue $\l_{1,n}$
($\overline{\l_{1,n}}$) for $|n|\ge \max\{n_0',n_1'\}$.

Further, combining~\eqref{eq.1+a.phi11=1,d+phi22=1} with
relation \eqref{eq.1,a+psi11=1} for large enough $|n|$  and
applying, Proposition~\ref{prop:phi.jk.asymp} and
Corollary~\ref{cor:Phi.Psi}, we arrive at the following
asymptotic relations
\begin{equation} \label{eq:phi2n,psi2n}
    \|f_{1,n}\| \asymp 1, \quad \|g_{1,n}\| \asymp 1,
    \quad \left(f_{1,n}, g_{1,n}\right) \asymp 1 \quad\text{for large}\ \ n,
\end{equation}
(cf.~\eqref{eq:phin,phin} and~\eqref{eq:phin,psin}).  Performing
normalization  of the sequences $\{f_{2,n}\}_{n \in \bZ}$ and
$\{g_{2,n}\}_{n \in \bZ}$ in the same way as for the
sequence~\eqref{eq:fn,gn.def} and repeating  the same argument
we get that their normalizations  are besselian biorthogonal
sequences.

Going over to  the second branch $\{\l_{2,n}\}_{n \in \bZ}$ of
eigenvalues we obtain from~\eqref{3.2} and~\eqref{eq:BC.new}
with account of the assumption $c=0$ that
   \begin{align}\label{5.10Second_eigenf-n}
f_{2,n}(\cdot) & :=  \widehat U_2\left(\Phi_2(\cdot, \l_{2,n})\right)\Phi_1(\cdot,
\l_{2,n}) -
\widehat U_2\left(\Phi_1(\cdot, \l_{2,n})\right) \Phi_2(\cdot, \l_{2,n})  \nonumber    \\
& =  (d +  \varphi_{22}(\l_{2,n}))\Phi_1(\cdot, \l_{2,n}) -
\varphi_{21}(\l_{2,n})\Phi_2(\cdot, \l_{2,n})
   \end{align}
 It follows from  the second relation in
\eqref{eq.1+a.phi11=1,d+phi22=1}  that for $n$ big enough
$f_{2,n}(\cdot)$ is a non-trivial eigenfunction of the operator
$L(Q)$ corresponding to the eigenvalue $\l_{2,n}$. Similarly, it
follows from the first  equation in~\eqref{eq:BC*} (with $b=0$)
and~\eqref{eq.1,a+psi11=1}  that for $n$ big enough the
vector-function
   \begin{align}\label{5.10Second}
g_{2,n}(\cdot) :=
U_{*1}\bigl(\Psi_2(\cdot,\overline{\l_{2,n}})\bigr)\Psi_1(\cdot,\overline{\l_{2,n}}) -
U_{*1}\bigl(\Psi_1(\cdot,\overline{\l_{2,n}})\bigr)\Psi_2(\cdot,\overline{\l_{2,n}})  \nonumber  \\
 =  (1+ \overline{d}\psi_{22}(\overline{\l_{2,n}})) \Psi_1(\cdot, \overline{\l_{2,n}})
    -  \overline{d} \psi_{11}(\overline{\l_{2,n}}) \Psi_2(\cdot, \overline{\l_{2,n}})
   \end{align}
is  a non-trivial  eigenfunction  of the operator $L(Q)^*$
corresponding to the eigenvalue $\overline{\l_{2,n}}$. Starting
with~\eqref{5.10Second_eigenf-n}  and~\eqref{5.10Second} one
completes the proof in the case of strictly regular BC by
repeating the above reasonings.
\end{proof}
%
%
\subsection{The case of general potential matrix}
%
%
In applications systems~\eqref{eq:system}--\eqref{eq:BC} appear
with potential matrices having non-trivial diagonal, i.e. of the
form
\begin{equation} \label{eq:Q.gen.def}
    Q = \begin{pmatrix} Q_{11} & Q_{12} \\ Q_{21} & Q_{22} \end{pmatrix}
    \in L^1[0,1] \otimes \bC^{2 \times 2}.
\end{equation}
First we apply gauge transformation to reduce
system~\eqref{eq:system}--\eqref{eq:BC}, with a potential matrix
$Q(\cdot)$ of the form~\eqref{eq:Q.gen.def} to similar system
with off-diagonal potential matrix $\wt{Q}$. To this end we put
\begin{equation}
    w_j(x) := \exp\left(-i b_j \int_0^x Q_{jj}(t) dt\right),
    \quad x \in [0,1], \quad j \in \{1,2\}, \label{eq:wjx.def}
\end{equation}
%
%
\begin{lemma} \label{lem:LbcQ.similarity}
Let $Q$ be a summable matrix given by~\eqref{eq:Q.gen.def}. Then
the operator $L(Q) = L(Q,{\widehat U}_1, {\widehat U}_2)$ is
similar to the operator $L(\wt{Q}) = L(\wt{Q}, \wt{U}_{1},
\wt{U}_{2})$  given by~\eqref{eq:system}--\eqref{eq:BC} with the
same $B$, a potential off-diagonal matrix  $\wt{Q}(\cdot)$,
\begin{equation}\label{eq:wtQ}
    \wt{Q}(x) := \begin{pmatrix} 0 & k(x) Q_{12}(x) \\ k^{-1}(x) Q_{21}(x) & 0\end{pmatrix},
    \qquad k(x) := w_1^{-1}(x) w_2(x),
\end{equation}
instead of ${Q}$ and the boundary conditions
\begin{equation}\label{eq:wtU}
    \wt{U}_j(y) := a_{j 1} y_1(0) + a_{j 2} y_2(0) + w_1(1) a_{j 3} y_1(1)
    + w_2(1) a_{j 4} y_2(1) = 0, \quad  j \in \{1,2\}.
\end{equation}
\end{lemma}
%
\begin{proof}
See the first part of the proof of~\cite[Proposition
3.4]{LunMal14JST}.
\end{proof}
%
%
\begin{corollary} \label{cor:separ.regul.basis}
Let $Q$ be a summable potential matrix given
by~\eqref{eq:Q.gen.def} and let boundary
conditions~\eqref{eq:BC} be separated and regular, i.e.
$$
a y_1(0) + b y_2(0) = c y_1(1) + d y_2(1) = 0 \quad \text{and}\quad  abcd \ne 0.
$$
Then  the eigenvalues of the corresponding operator $L_{C,D}(Q)$
are asymptotically separated and there exists $\beta \in \bC$
such that the following asymptotic formula holds
$$
\l_n = \frac{2 \pi n} {b_2 - b_1} + \beta + o(1) \quad \text{as}\quad  |n| \to \infty.
$$
Moreover, the system of  root vectors of  $L_{C,D}(Q)$ forms a
Riesz basis in $L^2[0,1] \otimes \bC^{2}$.
\end{corollary}
%
%
\begin{lemma} \label{lem:U.perturb}
Let boundary conditions~\eqref{eq:BC} be regular. Then there
exists $w \ne 0$ such that the boundary conditions
\begin{equation}\label{eq:wtU2}
    \wt{U}_j(y) := a_{j 1} y_1(0) + a_{j 2} y_2(0) + w a_{j 3} y_1(1)
    + a_{j 4} y_2(1) = 0, \quad  j \in \{1,2\},
\end{equation}
are strictly regular.
\end{lemma}
%
%
\begin{proof}
Since boundary conditions~\eqref{eq:BC} are regular,
$J_{14}J_{32}\not = 0$, we can assume without loos of generality
that $J_{14}= 1$. Let $\wt{\Delta}_0(\cdot)$ be the
characteristic determinant corresponding to the
BVP~\eqref{eq:system},~\eqref{eq:wtU2}. It is easily seen that
\begin{equation} \label{eq:wtDelta0}
    \wt{\Delta}_0(\l) = J_{12} + w J_{34} e^{i (b_1+b_2) \l}
    + w J_{32}e^{i b_1 \l} + e^{i b_2 \l}.
\end{equation}
Let us assume for simplicity  that $b_1/b_2 \in \bQ$, i.e. $b_1
= -n_1 b$, $b_2 = n_2 b$, $n_1, n_2 \in \bN$, $b>0$. In this
case we can rewrite $\wt{\Delta_0}(\cdot)$ in the following form
\begin{equation} \label{eq:wtDelta0=P}
    \wt{\Delta}_0(\l) = e^{i b_1 \l} P_w(e^{i b \l}), \quad
    P_w(z) := z^{n_1+n_2} + J_{12} z^{n_1} + w J_{34} z^{n_2} + w J_{32}.
\end{equation}
Since $J_{32} \ne 0$,  one  can choose $w \ne 0$ such that the
zeros of the polynomial $P_w(z)$ are simple. Clearly, for such
$w$ the zeros of $\wt{\Delta}_0(\cdot)$ are asymptotically
separated and thus boundary conditions~\eqref{eq:wtU2} are
strictly regular.
\end{proof}
To state the next result we recall that $m_a(\l_0)$ and
$m_g(\l_0)$ denote the algebraic and geometric multiplicities of
$\l_0$, respectively. Moreover, if $\l_0$ is an isolated
eigenvalue, then  $m_a(\l_0)$  equals  the dimension of the
Riesz projection.

We need the following known abstract result (see
e.g.~\cite{SavShk14}).
%
%
\begin{proposition} \label{prop:Riesz.basis.abstract}
Let $L$ be an operator with compact resolvent in a separable
Hilbert space $\fH$ and let $\{\l_n\}_{n \in \bZ}$ be the
sequence of its distinct eigenvalues. Assume that $m_a(\l_n) <
\infty$ for $n\in \bN$ and that $A$ has finitely many
associative vectors, i.e. there exists $n_0 \in \bN$ such that
$m_a(\l_n) = m_g(\l_n)$ for $|n| > n_0.$ Further, assume that
\begin{equation} \label{eq:ln>cn}
    |\l_n| \geqslant C |n|, \quad |\Im \l_n| \leqslant \tau, \qquad n \in \bZ,
\end{equation}
for some $C, \tau > 0$. Finally, let the system of root vectors
of the operator $L$ forms a Riesz basis in $\fH$. Then for any
bounded operator $T$ in $\fH$ the system of root vectors of the
perturbed operator $A=L+T$ forms a Riesz basis with parentheses
in $\fH$.
\end{proposition}
%
%
\begin{proof}
Since $L$ has finitely many associated vectors, there exists a
finite-dimensional operator $K$ such that the operator $L+K$ has
no  associative vectors, i.e.  $m_a(\l_n) = m_g(\l_n)$ for $n\in
\bN$. Then the system of eigenvectors of  $L + K$ constitutes  a
Riesz basis in $\fH$, i.e. the operator  $L + K$ is similar to a
normal operator  $H$. The latter admits a representation $H =
H_R + iH_I$ where the operators $H_R := \overline{(H +H^*)/2} =$
and $H_I:= \overline{(H-H^*)/2i}$ are self-adjoint and commute
(see~\cite[Theorem 6.6.1]{BirSol87}). Since the spectrum  of $H$
lie in a strip $\Pi_{\tau}$,  its imaginary part $H_I$ is
bounded,  $\|H_I\|\le \tau$. Clearly,
inequality~\eqref{eq:ln>cn} remains valid for eigenvalues of $H$
maybe with another constant  $C>0$. Therefore the operator $A$
is similar to a bounded perturbation of the self-adjoint
operator $H$ whose eigenvalues satisfy~\eqref{eq:ln>cn}. Hence,
by~\cite[Theorem 3.1]{Katsn67} (Katsnel'son-Markus-Matsaev
theorem, see also~\cite{MarMats84,Markus88}) the system of root
vectors of  $A$ constitutes a Riesz basis with parentheses.
\end{proof}
%
%
\begin{proposition} \label{prop:regul.basis.paren}
Let $Q$ be a summable potential matrix given
by~\eqref{eq:Q.gen.def} and let $L_{C, D}(Q)$ be the operator
associated in $L^2[0,1] \otimes \bC^2$ with the
BVP~\eqref{eq:system}--\eqref{eq:BC}. Assume that boundary
conditions~\eqref{eq:BC} are regular. Then root vectors system
of the operator $L_{C, D}(Q)$ forms a Riesz basis with
parentheses in $L^2[0,1] \otimes \bC^2$.
\end{proposition}
%
%
\begin{proof}
It is clear that the regularity of boundary conditions is
preserved under gauge transformation used in
Lemma~\ref{lem:LbcQ.similarity}. Therefore one can  assume that
$Q$ is off-diagonal. Now let  us consider a perturbation  of the
operator $L(Q)$ by a constant  diagonal potential matrix $Q_0 =
\diag(q_0, 0)$, $q_0 \in \bC$. Applying
Lemma~\ref{lem:LbcQ.similarity} again we see that the operator
$L_{C,D}(Q+Q_0)$ is similar to the operator $L_{C,
\wt{D}}(\wt{Q})$ with off-diagonal $\wt{Q}$ and with boundary
conditions $C y(0)+\wt{D} y(1)=0$, where $\wt{D} = D \cdot
\diag(w, 1)$, $w = e^{-i b_1 q_0}$. By
Lemma~\ref{lem:U.perturb}, we can choose $q_0 \in \bC$ such that
the boundary conditions $C y(0)+\wt{D} y(1)=0$ are strictly
regular. Let us verify that the operator $L(Q+Q_0)$ satisfies
conditions of Proposition~\ref{prop:Riesz.basis.abstract}. Since
eigenvalues of $L(Q+Q_0)$ are asymptotically separated, it
follows that their algebraic multiplicities are finite and
$L(Q+Q_0)$ has finitely many associated vectors.

According to Theorem~\ref{th:basis.strict}  the root vectors
system of the operator $L(Q+Q_0)$ forms a Riesz basis in
$L^2[0,1] \otimes \bC^{2}$. On the other hand,
inequalities~\eqref{eq:ln>cn} are implied by
Proposition~\ref{prop:sine.type}(iv). Thus, the operator
$L(Q+Q_0)$  satisfies the conditions of
Proposition~\ref{prop:Riesz.basis.abstract}, and hence  the root
vectors system of the original operator $L(Q) = L(Q+Q_0) - Q_0$
forms a Riesz basis with parentheses.
\end{proof}
%
%
\begin{remark}\label{rem_separ_BC}
\textbf{(i)} Note that inequalities \eqref{eq:ln>cn} are valid
for the roots of any entire function of  exponential  type
$\sigma<\infty$ with infinitely many zeros. Namely, the
following inequalities hold
\begin{equation} \label{eq:ln>cn}
    |\l_n| \geqslant \frac{ |n|}{e\sigma},   \qquad |n|\ge N,
\end{equation}
for all but finitely many numbers. In particular, they are valid
for the roots of  $\Delta(\cdot)$. However,
Proposition~\ref{prop:sine.type}(iv) gives sharp asymptotic.

\textbf{(ii)} The Riesz basis property for $2\times 2$ operators
$L_{C,D}(Q)$ with  separated boundary conditions was established
earlier than for the operators with general regular boundary
conditions. Namely, this property was proved firstly
in~\cite{TroYam02} and later on in~\cite{DjaMit10BariDir}
and~\cite{Bask11} for $B = \diag(-1,1),\ Q \in L^2[0,1] \otimes
\bC^{2 \times 2}$, and
 in~\cite{HasOri09} for $B = \diag(b_1,b_2),\ Q \in
C^1[0,1] \otimes \bC^{2 \times 2}$.

\textbf{(iii)} The Bari-Markus property of the Riesz projectors
of unperturbed and perturbed  BVPs for separated, periodic and
antiperiodic boundary conditions was established in
~\cite{DjaMit10BariDir} and reproved by another method in
~\cite{Bask11}. In~\cite{DjaMit12UncDir} similar  results have
been obtained  for general regular boundary conditions. Finally,
in the recent paper~\cite{MykPuy13} the results
of~\cite{DjaMit10BariDir} regarding the Bari-Markus property  in
$L^2[0,1] \otimes \Bbb C^{2m}$ were extended to the case of the
Dirichlet BVP for $2m \times 2m$ Dirac equation with  $Q \in
L^2([0,1]; \bC^{2m\times 2m})$.
\end{remark}
%
%
\section{Application to the Timoshenko beam model}
\label{sec:Timoshenko}
%
%
Consider the following linear system of two coupled hyperbolic
equations for $t \geqslant 0$
\begin{align}
    \label{eq:Tim.Ftt}
        I_{\rho}(x) \Phi_{tt} &=& K(x)(W_x-\Phi) + (EI(x) \Phi_x)_x - p_1(x) \Phi_t, \quad x \in [0, \ell],\\
    \label{eq:Tim.Wtt}
        \rho(x) W_{tt} &=& (K(x)(W_x-\Phi))_x - p_2(x) W_t, \qquad \qquad \qquad x \in [0, \ell].
\end{align}
The vibration of the Timoshenko beam of the length $\ell$
clamped at the left end is governed by the
system~\eqref{eq:Tim.Ftt}--\eqref{eq:Tim.Wtt} subject to the
following boundary conditions for $t \geqslant 0$~\cite{Tim55}:
\begin{align}
    \label{eq:Tim.W0F0}
        W(0,t) = \Phi(0,t) &=& 0, \\
    \label{eq:Tim.WLFLa1}
        \bigl(EI(x) \Phi_x(x,t) + \alpha_1 \Phi_t(x,t) + \beta_1 W_t(x,t)\bigr)\bigr|_{x=l} &=& 0, \\
    \label{eq:Tim.WLFLa2}
        \bigl(K(x)(W_x(x,t)-\Phi(x,t)) + \alpha_2 W_t(x,t) + \beta_2 \Phi_t(x,t)\bigr)\bigr|_{x=l} &=& 0.
\end{align}
Here $W(x,t)$ is the lateral displacement at a point $x$ and
time $t$, $\Phi(x,t)$ is the bending angle at a point $x$ and
time $t$, $\rho(x)$ is a mass density, $K(x)$ is the shear
stiffness of a uniform cross-section, $I_{\rho}(x)$ is the
rotary inertia, $EI(x)$ is the flexural rigidity at a point $x$,
$p_1(x)$ and $p_2(x)$ are locally distributed feedback
functions, $\alpha_j, \beta_j \in \bC$, $j \in \{1,2\}$.
Boundary conditions at the right end contain as partial cases
most of the known boundary conditions if $\alpha_1, \alpha_2$
are allowed to be infinity.

Regarding the coefficients we assume that they satisfy the
following general conditions:
\begin{align}
    \label{eq:Tim.coef.cond1}
        \rho, I_{\rho}, K, EI \in C[0,\ell], \qquad p_1, p_2 \in L^1[0,\ell],\\
    \label{eq:Tim.coef.cond2}
        0 < C_1 \leqslant \rho(x), I_{\rho}(x), K(x), EI(x) \leqslant C_2, \quad x \in [0,\ell].
\end{align}
The energy space associated with the
problem~\eqref{eq:Tim.Ftt}--\eqref{eq:Tim.WLFLa2} is
\begin{equation} \label{eq:cH.def}
    \fH := \wt{H}^1_0[0,\ell] \times L^2[0,\ell] \times \wt{H}^1_0[0,\ell] \times L^2[0,\ell],
\end{equation}
where $\wt{H}^1_0[0,\ell] := \{f \in W^{1,2}[0,\ell] :
f(0)=0\}$. The norm in the energy space is defined as follows:
\begin{equation} \label{eq:Tim.|y|H}
    \|y\|_{\fH}^2 = \int_0^\ell \bigl(EI|y_1'|^2+I_{\rho}|y_2|^2 + K|y_3'-y_1|^2+\rho|y_4|^2\bigr)dx,
    \quad y =\col(y_1,y_2,y_3,y_4).
\end{equation}
The problem~\eqref{eq:Tim.Ftt}--\eqref{eq:Tim.WLFLa2} can be
rewritten as
\begin{equation} \label{eq:Tim.yt=i.cLy}
    y_t = i \cL y, \quad y(x,t)|_{t=0} = y_0(x),
\end{equation}
where $y$ and $\cL$ are given by
\begin{equation} \label{eq:Tim.Ly.def}
    y = \begin{pmatrix} \Phi(x,t) \\ \Phi_t(x,t) \\ W(x,t) \\ W_t(x,t) \end{pmatrix}, \ \
    \cL \begin{pmatrix} y_1 \\ y_2 \\ y_3 \\ y_4 \end{pmatrix} = \frac{1}{i} \begin{pmatrix}
        y_2 \\ \frac{1}{I_{\rho}(x)}\Bigl(K(x)(y_3'-y_1) + \bigl(EI(x) y_1'\bigr)' - p_1(x) y_2\Bigr) \\
        y_4 \\ \frac{1}{\rho(x)}\Bigl(\bigl(K(x)(y_3'-y_1)\bigl)' - p_2(x) y_4\Bigr)
    \end{pmatrix}
\end{equation}
on the domain
\begin{align} \label{eq:Tim.dom.cL}
    \dom(\cL) &=& \left\{ y = \col(y_1,y_2,y_3,y_4) : y_1, y_2, y_3, y_4 \in \wt{H}^1_0[0,\ell]\right., \nonumber \\
    && EI \cdot y_1' \in AC[0,\ell], \ \ (EI\cdot y_1')' - p_1 y_2 \in L^2[0,\ell], \nonumber \\
    && K\cdot(y_3'-y_1) \in AC[0,\ell], \ \ (K\cdot(y_3'-y_1))' - p_2 y_4 \in L^2[0,\ell], \nonumber \\
    && \bigl(EI \cdot y_1'\bigr)(\ell) + \alpha_1 y_2(\ell) + \beta_1 y_4(\ell)= 0, \nonumber \\
    && \Bigl.\bigl(K \cdot (y_3'-y_1)\bigr)(\ell) + \alpha_2 y_4(\ell) + \beta_2 y_2(\ell)= 0 \Bigr\}.
\end{align}
Timoshenko beam model is investigated in numerous papers
(see~\cite{Tim55,KimRen87,Shub02,XuYung04,XuHanYung07,WuXue11}
and the references therein). A number of stability,
controllability, and optimization problems were studied. Note
also that the general
model~\eqref{eq:Tim.Ftt}--\eqref{eq:Tim.WLFLa2} of spatially
non-homogenous Timoshenko beam with both boundary and locally
distributed damping covers the cases studied by many authors.
Geometric properties of the system of root functions of the
operator $\cL$ play important role in investigation of different
properties of the
problem~\eqref{eq:Tim.Ftt}--\eqref{eq:Tim.WLFLa2}.

Below we establish the Riesz basis property with parentheses of
the operator $\cL$, without analyzing its spectrum. For
convenience we impose the following additional algebraic
assumption on $\cL$:
\begin{equation} \label{eq:Tim.Irho=C0 EI/K rho}
    \nu(x) :=  \frac{EI(x) \rho(x)} {K(x) I_{\rho}(x)} = \const, \quad x \in [0,\ell],
\end{equation}
Clearly,~\eqref{eq:Tim.Irho=C0 EI/K rho} is satisfied whenever
$I_{\rho}(x) = R \rho(x)$, where $R = \rm{const}$ is a
cross-sectional area of the beam, $EI$ and $K$ are constant
functions, while $\rho \in AC[0,\ell]$ and is arbitrary positive
(cf. condition~\eqref{eq:Tim.h1,h2.in.AC}). Our approach to the
spectral properties of the operator $\cL$ is based on the
similarity reduction of $\cL$ to a special $4\times 4$
Dirac-type operator. To state the result we need some additional
preparations.

Let $\gamma(\cdot)$ be given by
\begin{equation} \label{eq:Tim.Irho/EI=...}
    \sqrt{\frac{I_{\rho}(x)}{EI(x)}} = b_1 \gamma(x), \quad\text{where}\quad b_1>0
    \quad\text{and}\quad \int_0^\ell \gamma(x) dx = 1.
\end{equation}
Conditions~\eqref{eq:Tim.coef.cond1}
and~\eqref{eq:Tim.coef.cond2} imply together that $\gamma \in
C[0,\ell]$ and is positive. Further, in view
of~\eqref{eq:Tim.Irho=C0 EI/K rho} we have
\begin{equation} \label{eq:Tim.rho/K=...}
    \sqrt{\frac{\rho(x)}{K(x)}} = b_2 \gamma(x), \quad\text{where}\quad b_2>0.
\end{equation}
Let
\begin{align}
    \label{eq:Tim.B}
        B &:=& \diag(-b_1, b_1, -b_2, b_2). \\
    \label{eq:Tim.Theta(x)}
        \Theta(x) &:=& -2i \diag(I_{\rho}(x), I_{\rho}(x), \rho(x), \rho(x)), \\
    \label{eq:Tim.g1.g2.def}
        h_1(x) &:=& \sqrt{EI(x) I_{\rho}(x)}, \qquad h_2(x):=\sqrt{K(x) \rho(x)}.
\end{align}
In the sequel we assume that
\begin{equation} \label{eq:Tim.h1,h2.in.AC}
    h_1, h_2 \in AC[0,\ell].
\end{equation}
Therefore, according
to~\eqref{eq:Tim.coef.cond1}--\eqref{eq:Tim.coef.cond2} the
following matrix function is well-defined:
\begin{equation}
    \label{eq:Tim.Q(x)}
        \widehat{Q}(x) := \Theta^{-1}(x)
        \begin{pmatrix}
            p_1+h_1' & p_1-h_1' &     h_2  &    -h_2  \\
            p_1+h_1' & p_1-h_1' &     h_2  &    -h_2  \\
               -h_2  &    -h_2  & p_2+h_2' & p_2-h_2' \\
                h_2  &     h_2  & p_2+h_2' & p_2-h_2'
        \end{pmatrix}.
\end{equation}
Next, we set
\begin{equation} \label{eq:Tim.t=t(x)}
    t(x) = \int_0^x \gamma(s) ds, \quad x \in [0,\ell].
\end{equation}
Since $\gamma \in C[0,\ell]$ and is positive, the function
$t(\cdot)$ strictly increases on $[0,\ell]$, $t(\cdot) \in
C^1[0,\ell]$, and due to~\eqref{eq:Tim.Irho/EI=...} $t(\ell)=1$.
Hence, the inverse function $x(\cdot) := t^{-1}(\cdot)$ is well
defined, strictly increasing on $[0,1]$, and $x(\cdot) \in
C^1[0,1]$. Next, we put
\begin{equation} \label{eq:Tim.Q}
    Q(t) := \widehat{Q}(x(t)) =: (q_{jk}(t))_{j,k=1}^4, \quad t \in [0,1].
\end{equation}
Finally, let
\begin{equation} \label{eq:Tim.C.D}
    C = \begin{pmatrix}
        1 & 1 & 0 & 0 \\
        0 & 0 & 0 & 0 \\
        0 & 0 & 1 & 1 \\
        0 & 0 & 0 & 0
    \end{pmatrix}, \ \
    D = \begin{pmatrix}
        0                    & 0                    & 0                    & 0                    \\
        \alpha_1 - h_1(\ell) & \alpha_1 + h_1(\ell) & \beta_1              & \beta_1              \\
        0                    & 0                    & 0                    & 0                    \\
        \beta_2              & \beta_2              & \alpha_2 - h_2(\ell) & \alpha_2 + h_2(\ell) \\
    \end{pmatrix}.
\end{equation}
%
%
\begin{proposition} \label{prop:Tim.similar}
\cite[Proposition 6.1]{LunMal14Arx,LunMal14JST}
Let functions $\rho, I_{\rho}, K, EI, p_1, p_2, h_1, h_2$
satisfy
conditions~\eqref{eq:Tim.coef.cond1},~\eqref{eq:Tim.coef.cond2},~\eqref{eq:Tim.Irho=C0
EI/K rho} and~\eqref{eq:Tim.h1,h2.in.AC}. Then the operator
$\cL$ is similar to the $4 \times 4$ Dirac-type operator $L :=
L_{C,D}(Q)$ with the matrices $B,C,D,$ and $Q(\cdot)$ given
by~\eqref{eq:Tim.B},~\eqref{eq:Tim.C.D} and~\eqref{eq:Tim.Q}.
\end{proposition}
%
%
\begin{theorem} \label{th:Tim.weak}
Let
conditions~\eqref{eq:Tim.coef.cond1},~\eqref{eq:Tim.coef.cond2},~\eqref{eq:Tim.Irho=C0
EI/K rho},~\eqref{eq:Tim.h1,h2.in.AC} be satisfied and let also
\begin{equation} \label{eq:Tim.a1!=h1,a1!=h2}
    \beta_1 = \beta_2 = 0, \quad \alpha_1 \ne \pm h_1(\ell), \quad
    \alpha_2 \ne \pm h_2(\ell).
\end{equation}
Then the system of root functions of the operator $\cL$ forms a
Riesz basis with parentheses in $\fH$.
\end{theorem}
%
%
\begin{proof}
Consider the operator $L_{C,D}(Q)$ defined in
Proposition~\ref{prop:Tim.similar}. Since $\beta_1 = \beta_2 =
0$ we can represent it as bounded perturbation of the direct sum
of two $2 \times 2$ Dirac operators:
\begin{align}
    L_{C,D}(Q) &= L(U_{1}, V_{1}, Q_1) \oplus L(U_2, V_2, Q_2) + \wt{Q},
    \label{eq:LCDQ=Lbc1.Lbc2+wtQ} \\
    L(Q_j) &= -i \begin{pmatrix} -b_j & 0 \\ 0 & b_j \end{pmatrix} y' + Q_1 y,
    \quad y = \col(y_1, y_2), \label{eq:Lbcj.Qj} \\
    U_j(y) &:= y_1(0) + y_2(0), \quad V_j(y) := (\alpha_j - h_j(\ell)) y_1(1) +
    (\alpha_j + h_j(\ell)) y_2(1), \quad j \in \{1,2\}, \label{eq:bcj} \\
    Q_1(t) &= \widehat{Q}_1(x(t)), \quad
    \widehat{Q}_1 = (-2 i I_{\rho})^{-1} \begin{pmatrix}
        p_1+h_1' & p_1-h_1' \\
        p_1+h_1' & p_1-h_1' \\
    \end{pmatrix}, \label{eq:Q1t} \\
    Q_2(t) &= \widehat{Q}_1(x(t)), \quad
    \widehat{Q}_2 = (-2 i \rho)^{-1} \begin{pmatrix}
        p_2+h_2' & p_2-h_2' \\
        p_2+h_2' & p_2-h_2' \\
    \end{pmatrix}, \label{eq:Q2t} \\
    \wt{Q}(t) &= \widehat{\wt{Q}}(x(t)), \quad
    \widehat{\wt{Q}} = \Theta^{-1} \codiag\left(
        \begin{pmatrix} h_2 & -h_2 \\ h_2 & -h_2 \end{pmatrix},
        \begin{pmatrix} -h_2 & -h_2 \\ h_2 & h_2 \end{pmatrix}
    \right). \label{eq:wtQt}
\end{align}
It follows
from~\eqref{eq:Tim.coef.cond1},~\eqref{eq:Tim.coef.cond2},~\eqref{eq:Tim.h1,h2.in.AC}
that $Q_1, Q_2 \in L^1[0,1] \otimes \bC^{2 \times 2}$ and
$\wt{Q} \in L^{\infty}[0,1] \times \bC^{2 \times 2}$. Due to
conditions~\eqref{eq:Tim.a1!=h1,a1!=h2}, the operator $L(U_j,
V_j, Q_j)$ is a $2 \times 2$ Dirac operator with separated
regular boundary conditions. By
Corollary~\ref{cor:separ.regul.basis}, the system of its root
vectors forms a Riesz basis in $L^2[0,1] \otimes \bC^2$ and its
eigenvalues have a proper asymptotic, in particular,
inequality~\eqref{eq:ln>cn} is satisfied for them. It is also
clear that $L(U_j, V_j, Q_j)$ has finitely many associated
vectors. Clearly, the direct sum $L := L(U_1, V_1, Q_1) \oplus
L(U_2, V_2, Q_2)$ has the same properties. Since $\wt{Q}$ is
bounded, the operator $L_{C,D}(Q)$ is a bounded perturbation of
"spectral"\ operator $L$. Hence by
Proposition~\ref{prop:Riesz.basis.abstract}, the system of root
vectors of the operator $L_{C,D}(Q)$ forms a Riesz basis with
parentheses in $L^2[0,1] \otimes \bC^4$. Since, by
Proposition~\ref{prop:Tim.similar}, $\cL$ is similar to the
operator $L_{C,D}(Q)$, the system of root functions of $\cL$
forms a Riesz basis with parentheses in $\fH$.
\end{proof}
%
%
\begin{remark}
In~\cite{LunMal14Arx,LunMal14JST} the same result was proved
under additional smoothness assumptions
\begin{equation} \label{eq:Tim.p1,p2.inLinf}
    p_1, p_2 \in L^{\infty}[0,\ell], \quad h_1, h_2 \in \Lip_1[0,\ell].
\end{equation}
Hence Theorem~\ref{th:Tim.weak} is considerable generalization
to the most general conditions on coefficients.
\end{remark}
%
%

%
%
\end{document}